\documentclass{elsarticle}
\usepackage{graphicx} 
\usepackage{amsmath}
\usepackage{amsthm}
\usepackage{amssymb}
\usepackage{tikz} 
\usepackage{mathtools}
\usepackage{multirow}

\newcommand{\C}{\mathbb{C}}
\newcommand{\Z}{\mathbb{Z}}

\newcommand{\R}{\mathbb{R}}

\newcommand{\s}{\sigma}

\newcommand{\p}{\partial}

\newcommand\numberthis{\addtocounter{equation}{1}\tag{\theequation}}
\numberwithin{equation}{section}
\theoremstyle{plain}
 \newtheorem{thm}{Theorem}[section]
 \newtheorem{cor}[thm]{Corollary}
 \newtheorem{lem}[thm]{Lemma}
 \newtheorem{prop}[thm]{Proposition}
  \newtheorem{rem}[thm]{Remark}
\theoremstyle{definition}
 \newtheorem{defn}[thm]{Definition}

\DeclareMathOperator{\idx}{idx}
\DeclareMathOperator{\ord}{ord}
\DeclareMathOperator{\rank}{rank}
\DeclareMathOperator{\Ad}{Ad}

\title{Generalized hypergeometric functions with several variables}

\author[1]{Saiei-Jaeyeong Matsubara-Heo\fnref{fn1}}
\affiliation[1]{organization={Faculty of Advanced Science and Technology, Kumamoto University},
            country={Japan}}
\fntext[fn1]{e-mail: \texttt{saiei@educ.kumamoto-u.ac.jp}
}

\author[2]{Toshio Oshima\fnref{fn2}}
\affiliation[2]{organization={Center for Mathematics and Data Science, Josai University},
            country={Japan}}
\fntext[fn2]{e-mail: \texttt{oshima@ms.u-tokyo.ac.jp}
}
\date{}

\begin{document}
\begin{abstract}
We introduce a hypergeometirc series with several variables, which generalizes Appell's, Lauricella's and Kemp\'e de F\'eriet's hypergeometric series, and study the system of differential equations that it satisfies.
We determine the singularities, the rank and the condition for the reducibility of the system. 
We give complete local solutions of the system at many singular points of the system and solve the connection problem among these local solutions.
Under some assumptions, the system is written as a KZ-type equation.
We determine its spectral type in the direction of coordinates as well as simultaneous eigenspace decompositions of residue matrices.
The system may or may not be rigid in the sense of N.~Katz viewed as an ordinary differential equation in some direction.
We also show that the system is a special case of Gel'fand-Kapranov-Zelevinsky system.
From this point of view, we discuss multivariate generalizations.
\end{abstract}
\maketitle
\tableofcontents  
\section{Introduction}\label{sec:Intro}
In this paper, we study a bivariate hypergeometric series
\begin{equation*}\begin{split}
&F^{p,q,r}_{p',q',r'}\Bigl(
\begin{smallmatrix}
 \boldsymbol \alpha & \boldsymbol \beta & \boldsymbol \gamma\\
 \boldsymbol \alpha'& \boldsymbol \beta'& \boldsymbol \gamma'
\end{smallmatrix}
; x,y\Bigr)
:=\sum_{m=0}^\infty\sum_{n=0}^\infty\frac
{(\boldsymbol \alpha)_m(\boldsymbol \beta)_n(\boldsymbol \gamma)_{m+n}}
{(1-\boldsymbol \alpha')_m(1-\boldsymbol \beta')_n(1-\boldsymbol \gamma')_{m+n}}x^my^n,\\
&\qquad
\boldsymbol\alpha\in\mathbb C^p,\ \boldsymbol\beta\in\mathbb C^q,\ 
\boldsymbol\gamma\in\mathbb C^r,\ 
\boldsymbol\alpha'\in\mathbb C^{p'},\ \boldsymbol\beta'\in\mathbb C^{q'},\ 
\boldsymbol\gamma'\in\mathbb C^{r'} 
\end{split}\end{equation*}
in the case when this series satisfies a system of differential equations 
without an irregular singularity, which means $p'-p=q'-q=r-r'$. 
The notation such as $(\boldsymbol\alpha)_m$ is defined in \eqref{eq:notation}.
J.\ Kamp\'e de F\'eriet's hypergeometric  (HG in short) function given in
\cite{AK} is a special case of our hypergeometric function, namely, it is
\begin{align*}
\sum_{m=0}^\infty\sum_{n=0}^\infty\frac
{(\boldsymbol \alpha)_m(\boldsymbol \beta)_n(\boldsymbol \gamma)_{m+n}}
{(\boldsymbol \alpha')_m(\boldsymbol \beta')_n(\boldsymbol \gamma')_{m+n}}
 \frac{x^my^n}{m!n!}
= F^{p,p,r}_{p'+1,p'+1,r'}
\Bigl(
 \begin{smallmatrix}
  \boldsymbol \alpha & \boldsymbol \beta & \boldsymbol \gamma\\
  1-\boldsymbol \alpha',0 & 1-\boldsymbol \beta',0& 1-\boldsymbol \gamma'
 \end{smallmatrix}
 ; x,y\Bigr).
\end{align*}
Classical Appell's hypergeometric functions (cf.~\cite{AK}) are all special cases of $F^{p,q,r}_{p',q',r'}(x,y)$.
Other examples of $F^{p,q,r}_{p',q',r'}(x,y)$ were investigated in \cite{suzuki2018appell} and \cite{tsuda2015fundamental}, which corresponds to the cases $(q,r,q',r')=(1,1,2,0)$ and $(p,q,p',q')=(1,1,1,1)$. 
In \cite{Oi} the second author studies certain integral transformations of convergent power series and introduces $F^{p,q,r}_{p,q,r}(x,y)$ together with some related results as an application of the transformations without proof.

We investigate $F^{p,q,r}_{p',q',r'}(x,y)$ from two different points of views.
One is the viewpoint of a system of partial differential equations.
We introduce a system $\mathcal{M}^{\boldsymbol{\alpha},\boldsymbol{\beta},\boldsymbol{\gamma}}_{\boldsymbol{\alpha}',\boldsymbol{\beta}',\boldsymbol{\gamma}'}$ satisfied by $F^{p,q,r}_{p',q',r'}(x,y)$.
The system resembles Appell's system, in the sense that we can determine its irreducibility, singular locus, and some of its connection coefficients.
It is the symmetry of the system that is stressed in this paper.
There are distinguished coordinate transformations which exchange the role of parameters.
These transformations naturally act on the blowing-up of $\mathbb P^1\times\mathbb P^1$ at two points $(0,0)$ and $(\infty,\infty)$.
The divisors $\{x=0\},\{x=\infty\},\{y=0\},\{y=\infty\}$ and the exceptional divisors form a hexagon and the transformations act transitively on its vertices.
In the course of the study, it is also convenient to introduce another series 
$$
G^{p,q,r}_{p',q',r'}\Bigl(
\begin{smallmatrix}
 \boldsymbol \alpha & \boldsymbol \beta & \boldsymbol \gamma\\
 \boldsymbol \alpha'& \boldsymbol \beta'& \boldsymbol \gamma'
\end{smallmatrix}
 ; x,y\Bigr)
 =\sum_{m=0}^\infty\sum_{n=0}^\infty\frac
 {(\boldsymbol \alpha)_m(\boldsymbol \beta)_n(\boldsymbol \gamma)_{m-n}}
 {(1-\boldsymbol \alpha')_m(1-\boldsymbol \beta')_n(1-\boldsymbol \gamma')_{m-n}}
 x^my^n.
$$
Then, at the vertices of the hexagon, we can construct a basis of series solutions using $F^{p,q,r}_{p',q',r'}$ and $G^{p,q,r}_{p',q',r'}$.

The connection problem among the vertices of the hexagon can be solved only by using Gamma functions and trigonometric functions.
The idea of the computation is {\it boundary value problem} which goes back to a paper by the second author \cite{KO}.
In fact, the system $\mathcal M
 ^{\boldsymbol \alpha, \boldsymbol \beta, \boldsymbol \gamma}
 _{\boldsymbol \alpha',\boldsymbol \beta',\boldsymbol \gamma'} 
$ has regular singularities along an edge of the hexagon and the boundary value of the solution to the system is defined according to the characteristic exponent.
We reduce the connection problem along an edge to that of the boundary value.
The same idea is also used in the calculation of $c$-function of Heckman-Opdam's  hypergeometric function (\cite{OSh}).

At the end of this paper, we show that our system 
$\mathcal{M}^{\boldsymbol{\alpha},\boldsymbol{\beta},\boldsymbol{\gamma}}_{\boldsymbol{\alpha}',\boldsymbol{\beta}',\boldsymbol{\gamma}'}$ can be seen as a particular case of Gel'fand-Kapranov-Zelevinsky system  (\cite{gel1989hypergeometric}).
From this point of view, a multivariate generalization $\mathcal{M}^{\boldsymbol{\alpha}_1,\dots,\boldsymbol{\alpha}_n,\boldsymbol{\gamma}}_{\boldsymbol{\alpha}'_1,\dots,\boldsymbol{\alpha}'_n,\boldsymbol{\gamma}'}$ is introduced and analyzed. 
Through Gale duality in combinatorics (\cite[\S6.3]{ziegler2012lectures}), the secondary fan of the system $\mathcal{M}^{\boldsymbol{\alpha}_1,\dots,\boldsymbol{\alpha}_n,\boldsymbol{\gamma}}_{\boldsymbol{\alpha}'_1,\dots,\boldsymbol{\alpha}'_n,\boldsymbol{\gamma}'}$ is described.
Although our system is strictly more general than Appell-Lauricella's system, the secondary structure is the same.
From this point of view, some of the results on the system $\mathcal{M}^{\boldsymbol{\alpha},\boldsymbol{\beta},\boldsymbol{\gamma}}_{\boldsymbol{\alpha}',\boldsymbol{\beta}',\boldsymbol{\gamma}'}$ can be generalized to $\mathcal{M}^{\boldsymbol{\alpha}_1,\dots,\boldsymbol{\alpha}_n,\boldsymbol{\gamma}}_{\boldsymbol{\alpha}'_1,\dots,\boldsymbol{\alpha}'_n,\boldsymbol{\gamma}'}$.
Results of this paper can be related to the literature of Gel'fand-Kapranov-Zelevinsky system.
For example, the result on the connection problem can be seen as a generalization of \cite{saito1994restrictions} or a special case of \cite{matsubara2022global}.
When the singular locus of $\mathcal{M}^{\boldsymbol{\alpha},\boldsymbol{\beta},\boldsymbol{\gamma}}_{\boldsymbol{\alpha}',\boldsymbol{\beta}',\boldsymbol{\gamma}'}$ is a braid arrangement, its monodromy representation gives rise to a representation of the pure braid group.
Since the monodromy representation of Lauricella's $F_D$ system corresponds to the so-called Gassner representation (\cite[Example 4 of \S3.2]{birman1974braids}, see also \cite[\S5.2]{kohno2017quantum}), our result constructs a generalization of Gassner representation from the viewpoint of hypergeometric system.

The other point of view is that of Katz's theory.
For an ordinary differential equation on the Riemann sphere, Katz's theory provides a way to understand the global analysis (\cite{katz1996rigid},  \cite{Ow}).
An invariant called {\it rigidity index} and an integral transform called {\it middle convolution} play a fundamental role in Katz's theory.
The local structure of an ordinary differential equation is given by its local behavior called {\it generalized Riemann scheme}.
An ordinary differential equation is said to be {\it rigid} when the Riemann scheme uniquely determines the equation.
Katz \cite{katz1996rigid} proves that an ordinary differential equation is rigid if and only if it is constructed from the trivial equation $\frac{du}{dx}(x)=0$ by a successive application
of middle convolutions and gauge transformations by a power function ({\it addition}).

As the next step to Katz's theory, Y. Haraoka proposes to extend middle convolutions to Knizhnik-Zamolodchikov type equations (KZ-type equations in short) in \cite{Ha} and \cite{HaKZmulti}.
KZ-type equation is a particular class of integrable connections.
Any rigid ordinary differential equation can be extended to a KZ-type equation.
When $r=r'$, the system $\mathcal{M}^{\boldsymbol{\alpha},\boldsymbol{\beta},\boldsymbol{\gamma}}_{\boldsymbol{\alpha}',\boldsymbol{\beta}',\boldsymbol{\gamma}'}$ is isomorphic to a KZ-type equation.
We will show in \S\ref{sec:F1} that no restriction of the system $\mathcal{M}^{\boldsymbol{\alpha},\boldsymbol{\beta},\boldsymbol{\gamma}}_{\boldsymbol{\alpha}',\boldsymbol{\beta}',\boldsymbol{\gamma}'}$ to a line $\{x={\rm constant}\}$ or $\{y={\rm constant}\}$ is rigid if $(p-1)(q-1)(r-1)\neq 0$.
On the other hand, the system $\mathcal{M}^{\boldsymbol{\alpha},\boldsymbol{\beta},\boldsymbol{\gamma}}_{\boldsymbol{\alpha}',\boldsymbol{\beta}',\boldsymbol{\gamma}'}$ is constructed from the trivial KZ-type system $\frac{\partial u}{\partial x}(x,y)=\frac{\partial u}{\partial y}(x,y)=0$ by a successive application
of middle convolutions and additions.
The trick is to use several directions when we perform middle convolutions.
Thus, the restriction of the system $\mathcal{M}^{\boldsymbol{\alpha},\boldsymbol{\beta},\boldsymbol{\gamma}}_{\boldsymbol{\alpha}',\boldsymbol{\beta}',\boldsymbol{\gamma}'}$ to $x=\{{\rm constant}\}$ or $y=\{{\rm constant}\}$ provides a non-trivial example of non-rigid ordinary differential equation whose global monodromy can be determined explicitly.

From the viewpoint of KZ-type equation, it is more natural to study the blowing-up of $\mathbb P^1\times\mathbb P^1$ at three points $(0,0),(1,1)$ and $(\infty,\infty)$.
It has 10 distinguished divisors and at each of 15 normal crossing points, the corresponding pair of residue matrices admits a simultaneous eigenspace decomposition.
We use \cite{Okz} to see the change of this data via middle convolutions. 
The vertices of the hexagon corresponds to the six points where each simultaneous eigenspace is one-dimensional.
Moreover, in some cases, there are more than six singular points where this happens.

The generalized Riemann scheme of the system $\mathcal{M}^{\boldsymbol{\alpha},\boldsymbol{\beta},\boldsymbol{\gamma}}_{\boldsymbol{\alpha}',\boldsymbol{\beta}',\boldsymbol{\gamma}'}$ enjoys symmetry.
Let us write $F_{p,q,r}$ for $F^{p,q,r}_{p,q,r}$, $I_{p,q}$ for $F^{p-1,q-1,1}_{p,q,0}$ and $J_{p,q}$ for $F^{p,q,0}_{p-1,q-1,1}$.
Besides the case $r=r'$, the case when $p'=p+1,q'=q+1,r=1,r'=0$ and the case $p=p'+1,q=q'+1,r=0,r'=1$ both satisfy KZ-type equations.
The following picture shows the symmetry of equations:

\medskip
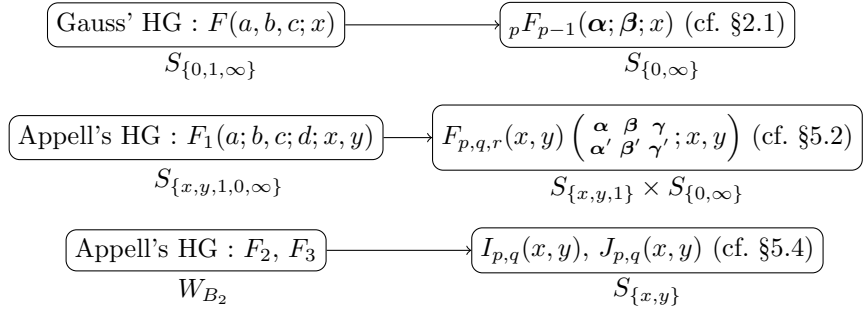
\begin{tikzpicture}
\node (f) [draw, rounded corners,rectangle] at (0,0) 
 {Gauss' HG : $F(a,b,c;x)$};
\node (g) [draw, rounded corners,rectangle] at (6,0) 
 {${}_pF_{p-1}(\boldsymbol\alpha;\boldsymbol\beta;x)$ (cf.~\S\ref{sec:HG1})};
\node (a) [draw, rounded corners,rectangle] at (0,-1.5)
 {Appell's HG : $F_1(a;b,c;d;x,y)$};
\node (b) [draw, rounded corners,rectangle] at (6,-1.5) 
{$F_{p,q,r}(x,y)\left(
  \begin{smallmatrix}\boldsymbol\alpha&\boldsymbol\beta&\boldsymbol\gamma\\
  \boldsymbol\alpha'&\boldsymbol\beta'&\boldsymbol\gamma'\end{smallmatrix}
  ;x,y\right)$ (cf.~\S\ref{sec:F1})};
\node (c) [draw, rounded corners,rectangle] at (0,-3)
 {Appell's HG : $F_2,\,F_3$};
\node (d) [draw, rounded corners,rectangle] at (6,-3) 
 {$I_{p,q}(x,y)$, $J_{p,q}(x,y)$ (cf.~\S\ref{sec:F2})};
\draw[->] (f)--(g);
\draw[->] (a)--(b);
\draw[->] (c)--(d);
\node at (0.2,-0.55) {$S_{\{0,1,\infty\}}$};
\node at (6.2,-0.55) {$S_{\{0,\infty\}}$};
\node at (0.3,-2.1) {$S_{\{x,y,1,0,\infty\}}$};
\node at (6,-2.2) {$S_{\{x,y,1\}}\times S_{\{0,\infty\}}$};
\node at (6,-3.55) {$S_{\{x,y\}}$};
\node at (0.1,-3.55){$W\!_{B_2}$};

\end{tikzpicture}

\medskip
\noindent
Here, $S_{\{x,y,1,0,\infty\}}$ means the group of permutations of 5 points $\{x,y,1,0,\infty\}$ on the Riemann sphere, which is the permutation group $S_5$.
In the above, $S_{\{x,y,1,0,\infty\}}$ acts on Appell's $F_1$ as coordinate transformations (cf.~\eqref{eq:dynkin}) and gives a symmetry of $F_1$.
On the other hand, 
$F_2$ and $F_3$  have coordinate symmetries which are identified with elements of the group $W\!_{B_2}$, the reflection group of type $B_2$ (cf.~Remark~\ref{rem:symF2}). 
The horizontal arrow means generalization.
Thus, the principal example in this paper $F_{p,q,r}$ is a direct generalization of the well-known function Appell's $F_1$.
We emphasize that any restriction of Appell's $F_1$ system to $x=\{{\rm constant}\}$ or $y=\{{\rm constant}\}$ gives rise to a rigid ordinary differential equation.

Katz's theory for KZ-type equation is still being developed. 
For example, the change of simultaneous eigenspace decomposition under middle convolution will be discribed in a forthcomming paper by the second author as in the case of two variables \cite{Okz}.
When the system $\mathcal{M}^{\boldsymbol{\alpha}_1,\dots,\boldsymbol{\alpha}_n,\boldsymbol{\gamma}}_{\boldsymbol{\alpha}'_1,\dots,\boldsymbol{\alpha}'_n,\boldsymbol{\gamma}'}$ is a KZ-type equation, it should be a good example of a {\it rigid} system, of which no sophisticated definition exists so far.

\medskip
The content of each section is as follows. 

In \S\ref{sec:HG1} we review some basic properties of the generalized hypergeometric function 
${}_pF_{p-1}(\boldsymbol\alpha;\boldsymbol\beta;x)$.
Most results stated in the following sections up to \S\ref{sec:rank-Sing} 
are reduced to the results 
given in \S\ref{sec:HG1}
.
In \S\ref{sec:HGEq} 
we introduce a system of differential equations $\mathcal M
 ^{\boldsymbol \alpha, \boldsymbol \beta, \boldsymbol \gamma}
 _{\boldsymbol \alpha',\boldsymbol \beta',\boldsymbol \gamma'} 
$
satisfied by
$F^{p,q,r}_{p',q',r'}\left(
\begin{smallmatrix}
 \boldsymbol \alpha & \boldsymbol \beta & \boldsymbol \gamma\\
 \boldsymbol \alpha'& \boldsymbol \beta'& \boldsymbol \gamma'
\end{smallmatrix};x,y\right)$.
In \S\ref{sec:Cond} we list some conditions on $p,q,r,p',q',r'$ used in this paper.
In \S\ref{sec:Known series} we define another series $G^{p,q,r}_{p',q',r'}$ which will 
appear in an analytic continuation
of $F^{p,q,r}_{p',q',r'}$ and we show some known series including Appell's 
hypergeometric series and some Horn's series are examples of our two series.
In \S\ref{sec:IntegRep} we give an integral representation of $F^{p,q,r}_{p',q',r'}$ using 
integral transformations in \cite{Oi}.
Some of the transformations are essentially middle convolutions.

In \S\ref{sec:coord trans} we study some coordinate transformations and its action on the system $\mathcal{M}^{\boldsymbol{\alpha},\boldsymbol{\beta},\boldsymbol{\gamma}}_{\boldsymbol{\alpha}',\boldsymbol{\beta}',\boldsymbol{\gamma}'}$.
In \S\ref{sec:connection} we solve the connection problem among the vertices of the hexagon. 

In \S\ref{sec:rank} we determine the rank of the system $\mathcal M
 ^{\boldsymbol \alpha, \boldsymbol \beta, \boldsymbol \gamma}
 _{\boldsymbol \alpha',\boldsymbol \beta',\boldsymbol \gamma'} 
$, the number of independent local solutions to the system at a  point, which equals the number of local solution  we have constructed.

In  \S\ref{sec:Sing} we determine  
the singular points of  the  system.
We call the number $L:=r-r'$ the level of the system.
By symmetry, we may assume that $L\geq 0$.
If $L=0$ or $L=1$, the singular points are a union of lines.
If $L>1$, the singular points are a union of lines and a set defined by zeros of an irreducible polynomial $f_L(x,y)$ with degree $L$.

In \S\ref{sec:KZ} we review basic results on KZ-type equations based on \cite{Okz}.
The system 
$\mathcal M
 ^{\boldsymbol \alpha, \boldsymbol \beta, \boldsymbol \gamma}
 _{\boldsymbol \alpha',\boldsymbol \beta',\boldsymbol \gamma'} 
$ becomes a KZ-type equation under some conditions. 
Of particular interest are the system satisfied by $F_{p,q,r}$ in \S\ref{sec:F1} and the one satisfied by $I_{p,q}$ or $J_{p,q}$ in \S\ref{sec:F2}.
In particular, $F_{p,q,1}$, $F_{p,1,r}$ and $I_{p, q}$ satisfy rigid ordinary differential equations for the variable $x$.
We show that the condition of the irreducibility of the system follows from a result given by \cite{Oir}.

Appell's $F_1,F_2,F_3$ systems are special examples of these with more symmetry.
More symmetry produces formulas on the hypergeometric series  like Kummer's formula of Gauss' hypergeometric series.  They are discussed in \S\ref{sec:AppellF1} and in \S\ref{sec:AppellF2}.

In \S\ref{sec:GKZ} we introduce hypergeometric series 
$F^{p_1,\dots,p_n,r}
_{p'_1,\dots,p'_n,r'}(x_1,\dots,x_n)$
with more than two variables.
Lauricella's hypergeometric series (cf.~\cite{lauricella}) is a special case of the series.
In \S\ref{sec:SystemM}, we give the system of differential equations 
$\mathcal{M}^{\boldsymbol{\alpha}_1,\dots,\boldsymbol{\alpha}_n,\boldsymbol{\gamma}}_{\boldsymbol{\alpha}'_1,\dots,\boldsymbol{\alpha}'_n,\boldsymbol{\gamma}'}$
satisfied by the series and the integral representation of the series. 

In \S\ref{sec:Singularset}, we determine the singular set of the system.
The formula generalizes known results for Appell-Lauricella's system \cite{hattori2014singular} and \cite{matsumoto2014monodromy}.
Theorem \ref{thm:Singn} is a generalization of Theorem \ref{thm:Sing}.

In \S\ref{sec:GKZsystem}, we review some basic properties and terminologies on GKZ system.
The exposition here is based on \cite{gel1992general} which is Gale dual to the standard literature on GKZ system.

In \S\ref{sec:GKZM},  
we relate the system
$\mathcal{M}^{\boldsymbol{\alpha}_1,\dots,\boldsymbol{\alpha}_n,\boldsymbol{\gamma}}_{\boldsymbol{\alpha}'_1,\dots,\boldsymbol{\alpha}'_n,\boldsymbol{\gamma}'}$
 to a GKZ system.
 A trick is to introduce some auxilliary variables.
 For later use, we provide a Gr\"obner basis of the associated toric ideal.

In \S\ref{sec:irreducible}, we describe the secondary fan of the system $\mathcal{M}^{\boldsymbol{\alpha}_1,\dots,\boldsymbol{\alpha}_n,\boldsymbol{\gamma}}_{\boldsymbol{\alpha}'_1,\dots,\boldsymbol{\alpha}'_n,\boldsymbol{\gamma}'}$ and determine the condition for 
the irreducibility by applying the criterion of \cite{schulze2012resonance}.
The special points without multiplicity of simultaneous eigenspace decomposition are labeled by a permutation group.
Therefore, the hexagon that appeared in the study of $\mathcal M^{\boldsymbol \alpha, \boldsymbol \beta, \boldsymbol \gamma}_{\boldsymbol \alpha',\boldsymbol \beta',\boldsymbol \gamma'}$ is replaced by a permutohedron.

In \S\ref{subsec:connection}, we derive a connection formula for $\mathcal{M}^{\boldsymbol{\alpha}_1,\dots,\boldsymbol{\alpha}_n,\boldsymbol{\gamma}}_{\boldsymbol{\alpha}'_1,\dots,\boldsymbol{\alpha}'_n,\boldsymbol{\gamma}'}$ which generalizes the same result for the system $\mathcal M
 ^{\boldsymbol \alpha, \boldsymbol \beta, \boldsymbol \gamma}
 _{\boldsymbol \alpha',\boldsymbol \beta',\boldsymbol \gamma'} 
$.
It turns out that a basis of local solutions at a vertex of a permutohedron is given in terms of two kinds of series $F^{p_1,\dots,p_n,r}
_{p'_1,\dots,p'_n,r'}(x_1,\dots,x_n)$ and $G^{p_1,\dots,p_n,r}
_{p'_1,\dots,p'_n,r',\epsilon}(x_1,\dots,x_n)$ where the latter is a multivariate generalization of the series $G^{p,q,r}_{p',q',r'}$.

In \S\ref{sec:braid}, we show that the connection problem discussed in \S\ref{subsec:connection} determines the monodromy representation of the pure braid group when $p_i=p_i'$ for $i=1,\dots,n$ and $r=r'$.

\section{Generalized hypergeometric functions of two variables}\label{sec:GHG}
\subsection{Generalized Hypergeometric Functions of one variable}\label{sec:HG1}
We review the generalized hypergeometric function ${}_pF_{p-1}(x)$ of one variable
(for example, see \cite[\S13.4]{Ow}).
We introduce the following notation for $\boldsymbol\alpha=(\alpha_1,\dots,\alpha_p)\in\C^p$, $c\in\C$ and $a\in \C$:
\begin{equation}
\begin{split}
 &(\boldsymbol\alpha)_m
:=\tfrac{\Gamma(\boldsymbol\alpha+m)}{\Gamma(\boldsymbol\alpha)}
= (\alpha_1)_m(\alpha_2)_m\cdots(\alpha_p)_m,\ \ 
 \Gamma(\boldsymbol\alpha):=\Gamma(\alpha_1)\cdots\Gamma(\alpha_p),\ \ \\
 &c-\boldsymbol\alpha=(c-\alpha_1,\dots,c-\alpha_p),\ \ 
|\boldsymbol\alpha|=\alpha_1+\cdots+\alpha_p,\ \ (1-a)_{-m}=\tfrac{(-1)^m}{(a)_m}.
\end{split}
\label{eq:notation}
\end{equation}
In general the generalized hypergeometric series
\begin{equation}
 {}_pF_q(\boldsymbol\alpha,\boldsymbol\beta;x):=\sum_{n=0}^\infty \frac{(\boldsymbol\alpha)_n}{(\boldsymbol\beta)_n}\frac{x^n}{n!}
 \quad(\boldsymbol\alpha\in\mathbb C^p,\ \boldsymbol\beta\in\mathbb C^q)
\end{equation}
is a convergent power series  if $p\le q+1$ and $\beta_i\notin\{0,-1,-2,\dots\}$.
The series satisfies a differential equation without an irregular singularity if $p=q+1$. 
A convergent power series 
\begin{align}\label{eq:Sols1}
 F_p\left(
\begin{smallmatrix} \boldsymbol \alpha\\ \boldsymbol \alpha'\end{smallmatrix};x
\right)&:=\sum_{n=0}^\infty\frac{(\boldsymbol \alpha)_n}{(1-\boldsymbol \alpha')_n}x^n
 \quad\bigl(\boldsymbol \alpha,\,\boldsymbol \alpha'\in\mathbb C^p\bigr)
\end{align}
satisfies an inhomogeneous ordinary differential equation
\begin{align}
\prod_{i=1}^{p}(\vartheta-\alpha'_i)u&\equiv x\prod_{i=1}^p(\vartheta+\alpha_i)u
\quad \mathrm{mod}\ \mathbb C\quad(\vartheta:=x\tfrac d{dx}).
\label{eq:inhomogeneousODE}
\end{align}
In fact, by comparing the terms of $x^n$  in the above equation, we have 
\begin{align*}
u=\sum_{n=0}^\infty c_nx^n&\Rightarrow \displaystyle\prod_{i=1}^{p}(n-\alpha'_i)c_nx^n=x\prod_{i=1}^{p}(n-1+\alpha_i)c_{n-1}x^{n-1}\quad(n>0)\\[-1mm]
 &\Rightarrow c_n=\displaystyle\frac{\prod_{i=1}^p(n-1+\alpha_i)}
{\prod_{i'=1}^{p'}(n-\alpha'_i)}c_{n-1}=\cdots
 =\displaystyle\frac{(\boldsymbol\alpha)_n}{(1-\boldsymbol\alpha')_n}c_0.
\end{align*}

Note that $(1)_n=n!$ and if $(\boldsymbol \alpha')_1\,(=\alpha'_1\cdots\alpha'_{p'})=0$, the modulo notation "$\equiv$" in \eqref{eq:inhomogeneousODE} can be replaced by the true equality "$=$".  
Since 
\begin{equation}
  P(\vartheta)x^\lambda u(x)=x^\lambda P(\vartheta+\lambda) u(x)
\end{equation}
for a polynomial $P(\vartheta)$ of $\vartheta$, 
the functions
\begin{align}\label{eq:HP1sols}
 F_p\left(
\begin{smallmatrix}\boldsymbol\alpha\\ \boldsymbol\alpha'
\end{smallmatrix};\alpha'_i;\pm; x\right)
&:=
 (\pm x)^{\alpha'_i}F_p\left(
\begin{smallmatrix}\boldsymbol\alpha+\alpha'_i\\ \boldsymbol\alpha'-\alpha'_i
\end{smallmatrix};x\right)\quad(i=1,\dots,p)
\\&
=(\pm x)^{\alpha'_i}{}_pF_{p-1}\bigl(\boldsymbol\alpha+\alpha'_i,
 \{1-\alpha'_\nu+\alpha'_i:\nu\ne i\};x\bigr)
\notag
\end{align}
give $p$ solutions to the generalized hypergeometric equation 
\begin{equation}\label{eq:EqHG1}
  \mathcal M^{\mathbf\alpha}_{\mathbf \alpha'}\, : \prod_{i=1}^{p} (\vartheta-\alpha'_i)=x\prod_{i=1}^{p}(\vartheta+\alpha_i)u
\end{equation}
if $\alpha'_i-\alpha'_j\notin\Z$ for any $i\neq j$.
In the sequel, we often use a convention such as
\begin{equation}\label{eq:abbrev}
(\vartheta+\boldsymbol{\alpha}):=\prod_{i=1}^{p}(\vartheta+\alpha_i).    
\end{equation}
The generalized Riemann scheme of the equation \eqref{eq:EqHG1} is
\begin{align*}
&
\begin{Bmatrix}
  x=0 & x=1 & x=\infty\\
  \alpha'_1 & [0]_{(p-1)} & \alpha_1\\
  \vdots    &     & \vdots\\
  \alpha'_p & \delta& \alpha_p
 \end{Bmatrix},\quad [0]_{(p-1)}:=\begin{pmatrix}0\\1\\\vdots\\p-1\end{pmatrix}, 
 \\
&\qquad 
 \sum_{\nu=1}^p \alpha_\nu+\sum_{\nu=1}^p \alpha'_\nu+\delta=p-1\quad\text{(Fuchs condition).}
\end{align*}
Here $[0]_{(p-1)}$ means the existence of $p-1$ linearly independent local holomorphic solutions 
at $x=1$.
The condition for the irreducibility of the equation equals
\[\alpha_i'+\alpha_j\not\in\mathbb Z\qquad(1\le i\le p,\ 1\le j\le p).\]

\noindent
Solutions at $x=\infty$ are given by
\[
F_p\left(
\begin{smallmatrix}\boldsymbol\alpha'\\ \boldsymbol\alpha
\end{smallmatrix};\alpha_{j};\pm;\tfrac1x\right)
 :=(\pm \tfrac1x)^{\alpha_{j}}{}_pF{}_{p-1}(\boldsymbol\alpha'+\alpha_{j},
 \{1-\alpha_\nu+\alpha_{j}:\nu\ne j\};\tfrac1x)
\]
if if $\alpha_i-\alpha_j\notin\Z$ for any $i\neq j$.

\noindent
The connection formula between $x=0$ and $\infty$ is given by
\begin{gather*}
F_p\left(
\begin{smallmatrix}\boldsymbol\alpha\\ \boldsymbol\alpha'
\end{smallmatrix};\alpha'_i;-;x\right)
=\sum_{j=1}^p 
c(0:\alpha'_i\rightsquigarrow \infty:\alpha_j)
F_p\left(
\begin{smallmatrix}\boldsymbol\alpha'\\ \boldsymbol\alpha
\end{smallmatrix};\alpha_{j};-;\tfrac1x\right) 
\quad(x\in\mathbb C\setminus[0,\infty)),\\
c(0:\alpha'_i\rightsquigarrow \infty:\alpha_{j})=\prod_{\substack{\nu\ne i\\ 1\le \nu\le p}}
 \frac{\Gamma(1+\alpha'_i-\alpha'_\nu)}{\Gamma(1-\alpha'_\nu-\alpha_j)}
 \prod_{\substack{\nu\ne j\\ 1\le \nu\le p}}
 \frac{\Gamma(\alpha_\nu-\alpha_j)}{\Gamma(\alpha'_i+\alpha_\nu)}.
\end{gather*}

\subsection{The system
$\mathcal M
 ^{\boldsymbol \alpha, \boldsymbol \beta, \boldsymbol \gamma}
 _{\boldsymbol \alpha',\boldsymbol \beta',\boldsymbol \gamma'}$
satisfied by Hypergeometric Series}\label{sec:HGEq}

For a given set of parameters $\boldsymbol\alpha\in\mathbb C^p,\ \boldsymbol\beta\in\mathbb C^q,\ 
\boldsymbol\gamma\in\mathbb C^r,\ 
\boldsymbol\alpha'\in\mathbb C^{p'},\ \boldsymbol\beta'\in\mathbb C^{q'},\ 
\boldsymbol\gamma'\in\mathbb C^{r'}$, we set
\begin{equation}
\begin{split}
&F^{p,q,r}_{p',q',r'}\Bigl(
\begin{smallmatrix}
 \boldsymbol \alpha & \boldsymbol \beta & \boldsymbol \gamma\\
 \boldsymbol \alpha'& \boldsymbol \beta'& \boldsymbol \gamma'
\end{smallmatrix}
; x,y\Bigr)
:=\sum_{m=0}^\infty\sum_{n=0}^\infty\frac
{(\boldsymbol \alpha)_m(\boldsymbol \beta)_n(\boldsymbol \gamma)_{m+n}}
{(1-\boldsymbol \alpha')_m(1-\boldsymbol \beta')_n(1-\boldsymbol \gamma')_{m+n}}x^my^n.
\end{split}
\label{eq:seriesF}
\end{equation}
Here, we note that 
\[
F^{p,q,r}_{p',q',r'}\bigl(
 \begin{smallmatrix}
  \boldsymbol \alpha & \boldsymbol \beta & \boldsymbol \gamma\\
  \boldsymbol \alpha'& \boldsymbol \beta'& \boldsymbol \gamma'
 \end{smallmatrix}
 ; x,y\Bigr)
=
F^{p+1,q+1,r}_{p'+1,q'+1,r'}\Bigl(
 \begin{smallmatrix}
  \boldsymbol \alpha,1 & \boldsymbol \beta,1 & \boldsymbol \gamma\\
  \boldsymbol \alpha,0 & \boldsymbol \beta',0& \boldsymbol \gamma'
 \end{smallmatrix}
 ; x,y\Bigr). 
\]
As in the last section, it is easy to see that the
$p'q'$ functions 
\begin{align}\label{eq:Sols}
\begin{split}
&F^{p,q,r}_{p',q',r'}\left(
\begin{smallmatrix}
 \boldsymbol \alpha&\boldsymbol \beta&\boldsymbol \gamma\\ \boldsymbol \alpha'&\boldsymbol \beta'&\boldsymbol \gamma'
\end{smallmatrix}
;\alpha'_i,\beta'_j;x,y\right)\\
&\quad:=
x^{\alpha'_i}y^{\beta'_j} F^{p,q,r}_{p',q',r'}\left(\begin{smallmatrix}
 \boldsymbol \alpha+\alpha'_i&\boldsymbol \beta+\beta'_j&\boldsymbol \gamma+\alpha'_i+\beta'_j\\ \boldsymbol \alpha'-\alpha'_i&\boldsymbol \beta'-\beta'_j&\boldsymbol \gamma'-\alpha'_i-\beta'_j
\end{smallmatrix}
;x,y\right)
\end{split}\end{align}
for $i=1,\dots,p'$ and $j=1,\dots,q'$ are solutions to the system of differential equations
\begin{equation}
 \mathcal M^{\boldsymbol \alpha,\boldsymbol \beta,\boldsymbol \gamma}_{\boldsymbol \alpha',\boldsymbol \beta',\boldsymbol \gamma'} : P_1u=P_2u=P_{12}u=0
\end{equation}
with
\begin{align*}
 P_1&=(\vartheta_x-\boldsymbol\alpha')(\vartheta_x+\vartheta_y-\boldsymbol\gamma')
  - x(\vartheta_x+\boldsymbol\alpha)(\vartheta_x+\vartheta_y+\boldsymbol\gamma),\\
 P_2&=
  (\vartheta_y-\boldsymbol\beta')(\vartheta_x+\vartheta_y-\boldsymbol\gamma')
  -y(\vartheta_y+\boldsymbol\beta)(\vartheta_x+\vartheta_y+\boldsymbol\gamma),\\
 P_{12}&=
  x(\vartheta_x+\boldsymbol\alpha)(\vartheta_y-\boldsymbol\beta')-
  y(\vartheta_y+\boldsymbol\beta)(\vartheta_x-\boldsymbol\alpha').
\end{align*}
Here, we use the notation $\p_x=\tfrac{\partial}{\partial x}$, $\p_y=\tfrac{\partial}{\partial y}$,
$\vartheta_x=x\p_x$ and $\vartheta_y=y\p_y$ and the convention \eqref{eq:abbrev}.
For $\epsilon_1,\epsilon_2=\pm 1$, $i=1,\dots,p'$ and $j=1,\dots,q'$ we set

\begin{align}\label{eq:pmSols}
\begin{split}
&F^{p,q,r}_{p',q',r'}\left(
\begin{smallmatrix}
 \boldsymbol \alpha&\boldsymbol \beta&\boldsymbol \gamma\\ \boldsymbol \alpha'&\boldsymbol \beta'&\boldsymbol \gamma'
\end{smallmatrix}
;\alpha'_i,\beta'_j;\epsilon_1,\epsilon_2;x,y\right)\\
&\quad:=
(\epsilon_1x)^{\alpha'_i}(\epsilon_2y)^{\beta'_j} F^{p,q,r}_{p',q',r'}\left(\begin{smallmatrix}
 \boldsymbol \alpha+\alpha'_i&\boldsymbol \beta+\beta'_j&\boldsymbol \gamma+\alpha'_i+\beta'_j\\ \boldsymbol \alpha'-\alpha'_i&\boldsymbol \beta'-\beta'_j&\boldsymbol \gamma'-\alpha'_i-\beta'_j
\end{smallmatrix}
;x,y\right).
\end{split}\end{align}

\begin{rem}\label{rem:red0}
If $\alpha'_{p'}=\beta'_{q'}=0$, the system
$\mathcal M^{\boldsymbol\alpha,\boldsymbol\beta,\boldsymbol\gamma}
_{\boldsymbol\alpha',\boldsymbol\beta',\boldsymbol\gamma'}$ 
is reduced to 
$\bar{\mathcal M}^{\boldsymbol\alpha,\boldsymbol\beta,\boldsymbol\gamma}
 _{\boldsymbol\alpha',\boldsymbol\beta',\boldsymbol\gamma'}
 : \bar P_1u=\bar P_2u=\bar P_{12}u=0$ with
\begin{align*}
 \bar P_1&=\p_x\prod_{i=1}^{p'-1}(\vartheta_x-\alpha'_i)\prod_{k=1}^{r'}(\vartheta_x+\vartheta_y-\gamma'_k)-
  \prod_{i=1}^p(\vartheta_x+\alpha_i)\prod_{k=1}^r(\vartheta_x+\vartheta_y+\gamma_k),\\
 \bar P_2&=
  \p_y\prod_{j=1}^{q'-1}(\vartheta_y-\beta'_i)\prod_{k=1}^{r'}(\vartheta_x+\vartheta_y-\gamma'_k)
  -
  \prod_{j=1}^q(\vartheta_y+\beta_j)\prod_{k=1}^r(\vartheta_x+\vartheta_y+\gamma_k),\\
\bar P_{12}&=
  \p_y\prod_{i=1}^p(\vartheta_x+\alpha_i) \prod_{j=1}^{q'-1}(\vartheta_y-\beta'_i)-
  \p_x\prod_{j=1}^q(\vartheta_y+\beta_j)\prod_{i=1}^{p'-1}(\vartheta_x-\alpha'_i).
\end{align*}
\end{rem}

We mention that the contiguity relation can be readily seen from the definition of the system $\mathcal M^{\boldsymbol\alpha,\boldsymbol\beta,\boldsymbol\gamma}
_{\boldsymbol\alpha',\boldsymbol\beta',\boldsymbol\gamma'}$.
Let $\mathcal Sol\left(\mathcal M^{\boldsymbol\alpha,\boldsymbol\beta,\boldsymbol\gamma}
_{\boldsymbol\alpha',\boldsymbol\beta',\boldsymbol\gamma'}\right)$ be the space of holomorphic solutions to the system $\mathcal M^{\boldsymbol\alpha,\boldsymbol\beta,\boldsymbol\gamma}
_{\boldsymbol\alpha',\boldsymbol\beta',\boldsymbol\gamma'}$.
We denote by $\mathcal Sol\left(\mathcal M^{\boldsymbol\alpha,\boldsymbol\beta,\boldsymbol\gamma}
_{\boldsymbol\alpha',\boldsymbol\beta',\boldsymbol\gamma'}|_{\alpha_{i}\mapsto\alpha_{i}+1}\right)$ the solution space of the same system with its parameter $\alpha_{i}$ replaced by $\alpha_{i}+1$ and the other parameters unchanged. 
We use similar notation for other parameters.
\begin{prop}\label{prop:shift}
There are induced morphisms
\begin{align*}
\vartheta_x+\alpha_{i} &:    \mathcal Sol\left(\mathcal M^{\boldsymbol\alpha,\boldsymbol\beta,\boldsymbol\gamma}
_{\boldsymbol\alpha',\boldsymbol\beta',\boldsymbol\gamma'}\right)\to \mathcal Sol\left(\mathcal M^{\boldsymbol\alpha,\boldsymbol\beta,\boldsymbol\gamma}
_{\boldsymbol\alpha',\boldsymbol\beta',\boldsymbol\gamma'}|_{\alpha_{i}\mapsto\alpha_{i}+1}\right) &(i=1,\dots,p),\\
\vartheta_x+\beta_{i} &:   \mathcal Sol\left(\mathcal M^{\boldsymbol\alpha,\boldsymbol\beta,\boldsymbol\gamma}
_{\boldsymbol\alpha',\boldsymbol\beta',\boldsymbol\gamma'}\right)\to \mathcal Sol\left(\mathcal M^{\boldsymbol\alpha,\boldsymbol\beta,\boldsymbol\gamma}
_{\boldsymbol\alpha',\boldsymbol\beta',\boldsymbol\gamma'}|_{\beta_{i}\mapsto\beta_{i}+1}\right) &(i=1,\dots,q),\\
\vartheta_x+\vartheta_y+\gamma_k &:   \mathcal Sol\left(\mathcal M^{\boldsymbol\alpha,\boldsymbol\beta,\boldsymbol\gamma}
_{\boldsymbol\alpha',\boldsymbol\beta',\boldsymbol\gamma'}\right)\to \mathcal Sol\left(\mathcal M^{\boldsymbol\alpha,\boldsymbol\beta,\boldsymbol\gamma}
_{\boldsymbol\alpha',\boldsymbol\beta',\boldsymbol\gamma'}|_{\gamma_k\mapsto\gamma_k+1}\right) & (k=1,\dots,r),\\
\vartheta_x-\alpha'_{i} &:   \mathcal Sol\left(\mathcal M^{\boldsymbol\alpha,\boldsymbol\beta,\boldsymbol\gamma}
_{\boldsymbol\alpha',\boldsymbol\beta',\boldsymbol\gamma'}\right)\to \mathcal Sol\left(\mathcal M^{\boldsymbol\alpha,\boldsymbol\beta,\boldsymbol\gamma}
_{\boldsymbol\alpha',\boldsymbol\beta',\boldsymbol\gamma'}|_{\alpha'_{i}\mapsto\alpha'_{i}+1}\right) &(i=1,\dots,p'),\\
\vartheta_x-\beta'_{i} &:   \mathcal Sol\left(\mathcal M^{\boldsymbol\alpha,\boldsymbol\beta,\boldsymbol\gamma}
_{\boldsymbol\alpha',\boldsymbol\beta',\boldsymbol\gamma'}\right)\to \mathcal Sol\left(\mathcal M^{\boldsymbol\alpha,\boldsymbol\beta,\boldsymbol\gamma}
_{\boldsymbol\alpha',\boldsymbol\beta',\boldsymbol\gamma'}|_{\beta'_{i}\mapsto\beta'_{i}+1}\right) &(i=1,\dots,q'),\\
\vartheta_x+\vartheta_y-\gamma'_k &:   \mathcal Sol\left(\mathcal M^{\boldsymbol\alpha,\boldsymbol\beta,\boldsymbol\gamma}
_{\boldsymbol\alpha',\boldsymbol\beta',\boldsymbol\gamma'}\right)\to \mathcal Sol\left(\mathcal M^{\boldsymbol\alpha,\boldsymbol\beta,\boldsymbol\gamma}
_{\boldsymbol\alpha',\boldsymbol\beta',\boldsymbol\gamma'}|_{\gamma'_k\mapsto\gamma'_k+1}\right) & (k=1,\dots,r').
\end{align*}
\end{prop}

\begin{proof}
    Using the identity $\vartheta_x\circ x=x\circ(\vartheta_x+1)$, we obtain relations
    \begin{align*}
P_1|_{\alpha_i\mapsto\alpha_i+1}\circ (\vartheta_x+\alpha_i)&=(\vartheta_x+\alpha_i)\circ P_1\\
P_2|_{\alpha_i\mapsto\alpha_i+1}&=P_2\\
P_{12}|_{\alpha_i\mapsto\alpha_i+1}\circ (\vartheta_x+\alpha_i)&=(\vartheta_x+\alpha_i)\circ P_{12}.
    \end{align*}
    Thus, the first line of the proposition is proved.
    The other lines can be proved in a similar way.
\end{proof}

\begin{rem}
Under the irreducibility condition on Theorem \ref{thm:irreducibility} and the condition \eqref{eq:nontrivial}, the maps above are isomorphisms.
\end{rem}

\subsection{Some Conditions}\label{sec:Cond}
We introduce the following conditions for the system
$\mathcal M^{\boldsymbol \alpha,\boldsymbol \beta,\boldsymbol \gamma}_{\boldsymbol \alpha',\boldsymbol \beta',\boldsymbol \gamma'}$. 
\begin{gather}
p+r=p'+r',\quad q+r=q'+r'
\tag{F}\\
p=p',\ q=q',\ r=r'\tag{F1}\\
p=p'-1,\ q=q'-1,\ r=1,\ r'=0\tag{F2}\\
p'=p-1,\ q'=q-1,\ r=0,\ r'=1\tag{F3}\\
p=p'-2,\ q=q'-2,\ r=2,\ r'=0\tag{F4}\\
(\boldsymbol{\alpha}')_1=(\boldsymbol{\beta}')_1=0\tag{T}\label{eq:terminate}
\end{gather}
Note that the condition $(\boldsymbol\alpha')_1=0$ in (T) is that at least one of $\alpha'_1,\dots,\alpha'_{p'}$ is zero.
We always assume the condition (F) in this paper
and call the number
\begin{equation}
    L:=r-r'
\end{equation}
the level of the system $\mathcal M^{\boldsymbol \alpha,\boldsymbol \beta,\boldsymbol \gamma}_{\boldsymbol \alpha',\boldsymbol \beta',\boldsymbol \gamma'}$.
We set
\begin{equation}
    \epsilon:=(-1)^L.
\end{equation}
In the following, we assume 
\begin{equation}\label{eq:nontrivial}
p+p'>1,\ q+q'> 1,\ r+r'> 0.
\end{equation}
\color{black}
Under the conditions (T) and \eqref{eq:nontrivial}
the conditions (F1), (F2), (F3) and (F4) with smallest possible $(p,q)$
correspond to Appell's $ F_1$, $F_2$, $F_3$ and $F_4$, respectively. Namely, in these cases, 
$\left\{\begin{smallmatrix}
p &q& r\\ p'& q'&r'
\end{smallmatrix}\right\}$ are as follows:

\medskip
\scalebox{1.05}{\begin{tabular}{|c|c|c|c|c|}\hline
& (F1) &(F2) & (F3) & (F4)\\ \hline
$\left\{\begin{smallmatrix}
p &q& r\\ p'& q'&r'
\end{smallmatrix}\right\}$
&
$\left\{\begin{smallmatrix}
p &q & r \\ p &q &r
\end{smallmatrix}\right\}$
&
$\left\{\begin{smallmatrix}
p-1 &q-1 & 1 \\ p &q &0
\end{smallmatrix}\right\}$
&
$\left\{\begin{smallmatrix}
p&q& 0 \\ p-1 & q-1 &1 
\end{smallmatrix}\right\}$
&
$\left\{\begin{smallmatrix}
p-2&q-2& 2 \\ p&q&0
\end{smallmatrix}\right\}$
\\ \hline
{\small Appell's HG}
 &
$\left\{\begin{smallmatrix}
1 &1 & 1 \\ 1 &1 &1
\end{smallmatrix}\right\}$
&
$\left\{\begin{smallmatrix}
1 &1 & 1 \\ 2 &2 &0
\end{smallmatrix}\right\}$
&
$\left\{\begin{smallmatrix}
2&2& 0 \\ 1 &1 &1
\end{smallmatrix}\right\}$
&
$\left\{\begin{smallmatrix}
0&0& 2 \\ 2&2&0
\end{smallmatrix}\right\}$
\\
\hline
\end{tabular}}

\subsection{Examples: Known Series}\label{sec:Known series}
We introduce series
\begin{align*}
 G^{p,q,r}_{p',q',r'}\Bigl(
\begin{smallmatrix}
 \boldsymbol \alpha & \boldsymbol \beta & \boldsymbol \gamma\\
 \boldsymbol \alpha'& \boldsymbol \beta'& \boldsymbol \gamma'
\end{smallmatrix}
 ; x,y\Bigr)
 &:=\sum_{m=0}^\infty\sum_{n=0}^\infty\frac
 {(\boldsymbol \alpha)_m(\boldsymbol \beta)_n(\boldsymbol \gamma)_{m-n}}
 {(1-\boldsymbol \alpha')_m(1-\boldsymbol \beta')_n(1-\boldsymbol \gamma')_{m-n}}
 x^my^n
\intertext{and}
 G^{p,q,r,r'}_{p',q'}\Bigl(
\begin{smallmatrix}
 \boldsymbol \alpha & \boldsymbol \beta & \boldsymbol \gamma& \boldsymbol \gamma'\\
 \boldsymbol \alpha'& \boldsymbol \beta'
\end{smallmatrix}
 ; x,y\Bigr)&:= G^{p,q,r}_{p',q',r'}\Bigl(
\begin{smallmatrix}
 \boldsymbol \alpha & \boldsymbol \beta & \boldsymbol \gamma\\
 \boldsymbol \alpha'& \boldsymbol \beta'& \boldsymbol \gamma'
\end{smallmatrix}
 ; x,(-1)^{r'}y\Bigr)\\
&\ =\sum_{m=0}^\infty\sum_{n=0}^\infty\frac
 {(\boldsymbol \alpha)_m(\boldsymbol \beta)_n(\boldsymbol \gamma)_{m-n}(\boldsymbol \gamma')_{n-m}}
 {(1-\boldsymbol \alpha')_m(1-\boldsymbol \beta')}
 x^my^n.
\end{align*}

\noindent
Appell's hypergeometric functions and some Horn's hypergeometric functions (cf.~\cite{Er}) are 
examples of the functions we have introduced:
\begin{align}
F_1(\alpha;\beta,\beta';\gamma;x,y)&=
F^{1,1,1}_{1,1,1}\left(\begin{smallmatrix}
 \beta&\beta'&\alpha\\ 0&0&1-\gamma 
\end{smallmatrix};x,y\right),\\
F_2(\alpha;\beta,\beta';\gamma,\gamma';x,y)&=
F^{1,1,1}_{2,2,0}\left(\begin{smallmatrix}
 \beta&\beta'&\alpha\\ 0,1-\gamma&0,1-\gamma'& 
\end{smallmatrix};x,y\right),\\
F_3(\alpha,\alpha';\beta,\beta';\gamma;x,y)&=
F^{2,2,0}_{1,1,1}\left(\begin{smallmatrix}
 \alpha,\alpha'&\beta,\beta'&\\ 0&0&1-\gamma
\end{smallmatrix};x,y\right),\\
F_4(\alpha;\beta;\gamma,\gamma';x,y)
&=
F^{0,0,2}_{2,2,0}\left(\begin{smallmatrix}
 \emptyset&\emptyset&\alpha,\beta\\0,1-\gamma&0,1-\gamma'& \emptyset
\end{smallmatrix};x,y\right),\\
G_2(\alpha,\alpha';\beta,\beta';x,y)&=
G^{1,1,1,1}_{1,1}\left(\begin{smallmatrix}
 \alpha&\alpha'&\beta'&\beta\\0&0
\end{smallmatrix};x,y\right),\\
H_2(\alpha,\beta,\gamma,\delta,\epsilon;x,y)&=
G^{1,2,1,0}_{2,1}\left(\begin{smallmatrix}
 \beta&\gamma,\delta&\alpha\\ 0,1-\epsilon&0
\end{smallmatrix};x,y\right).
\end{align}

\subsection{Integral Representation}\label{sec:IntegRep}
We recall the following integral transformations in \cite{Oi}.
Put $\mathbf{x}:=(x_1,\dots,x_n)$.
\begin{defn}\label{def:integraltransform}
\begin{align*}
\begin{split}
K_{\mathbf{x}}^{\mu,\boldsymbol\lambda} u
  &:=\displaystyle\frac{\Gamma(|\boldsymbol\lambda|+\mu)}{\Gamma(\mu)\Gamma(\boldsymbol\lambda)}
   \int_{\substack{t_1>0,\dots,t_n>0\\t_1+\cdots+t_n<1}}
   t^{\boldsymbol\lambda-1}(1-|\mathbf t|)^{\mu-1} u(t_1x_1,\dots,t_nx_n)\\
  &\qquad dt_1\dots dt_n,
\end{split}
\allowdisplaybreaks\\
\begin{split}
L_{\mathbf{x}}^{\mu,\boldsymbol\lambda} u &:=\frac{\Gamma(\mu+n)\Gamma(\boldsymbol\lambda)}{(2\pi i)^n\Gamma(|\boldsymbol\lambda|+\mu)}\int_{\tfrac1{n+1}-i\infty}^{\tfrac1{n+1}+i\infty}\!\cdots\!
\int_{\tfrac1{n+1}-i\infty}^{\tfrac1{n+1}+i\infty} t^{1-\boldsymbol\lambda}
(1-|\mathbf t|)^{-\mu-n}
\\&\qquad
u(\tfrac{x_1}{t_1},\dots,\tfrac{x_n}{t_n})
\tfrac{dt_1}{t_1}\!\cdots\!\tfrac{dt_n}{t_n},
\end{split}
\allowdisplaybreaks\\
(\tilde K_{\mathbf{x}}^{\mu,\lambda}u)(x)&:=\displaystyle\frac{\Gamma(\mu+\lambda)}{\Gamma(\mu)\Gamma(\lambda)}\int_0^1t^{\lambda-1}(1-t)^{\mu-1}u(tx_1,\dots,tx_n) dt.
\end{align*}
\end{defn}
\begin{rem}  By meromorphic continuations of convergent integrals with respect to $\boldsymbol\lambda$ and $\mu$,  the integral transformations  $K_{\mathbf{x}}^{\mu,\boldsymbol\lambda}$ and $L_{\mathbf{x}}^{\mu,\boldsymbol\lambda}$ are defined in \cite{Oi}. 
They are same transformations as the above ones up to constant multiples. 
\end{rem}
Putting 
\begin{align*}
(T_{\mathbf{x}\to R(\mathbf{x})}&u)(\mathbf{x}):=u(R(\mathbf{x}))
\end{align*}
for a coordinate transformation $\mathbf{x}\mapsto R(\mathbf{x})$, we have an identity
\begin{align}
\tilde K_{\mathbf{x}}^{\mu,\lambda}=T^{-1}_{\mathbf x \to (x_1,\frac {x_1}{x_2},\dots,\frac {x_1}{x_n})}\circ K_{x_1}^{\mu,\lambda}\circ T_{\mathbf x \to (x_1,\frac {x_1}{x_2},\dots,\frac {x_1}{x_n})}.\end{align}
These transformations have the following properties:

\medskip
\begin{tabular}{|c|c|c|c|c|}\hline
&$K_{x,y}^{\mu-\alpha-\beta,(\alpha,\beta)}$&
 $L_{x,y}^{\mu-\alpha-\beta,(\alpha,\beta)}$&
  $\tilde K_{x,y}^{\mu-\gamma,\gamma}$&
  $K_{x}^{\mu-\alpha,\alpha}$\\ \hline
\!$x^my^n\mapsto$\!&
$\frac{(\alpha)_m(\beta)_n}{(\mu)_{m+n}} x^my^n$&
$
 \frac{(\mu)_{m+n}}{(\alpha)_m(\beta)_n}x^my^n$&  
$
  \frac{(\gamma)_{m+n}}{(\mu)_{m+n}}x^my^n$&
$\frac{(\alpha)_m}{(\mu)_m}x^my^n$
\\ \hline
\end{tabular}

\medskip
\noindent
Hence, we have the following theorem.
\begin{thm}\label{thm:IntegralRepresentation}
Assume $\alpha'_{p'}=\beta'_{q'}=0$. Then, the following identity holds.
\begin{align*}
&F^{p,q,r}_{p',q',r'}\left(
\begin{smallmatrix}
\boldsymbol\alpha&\boldsymbol\beta&\boldsymbol\gamma \\
\boldsymbol\alpha'&\boldsymbol\beta'&\boldsymbol\gamma'    
\end{smallmatrix}
;x,y
\right)\\
&=\prod_{i=1}^{p'-1}K_x^{1-\alpha'_i-\alpha_i,\alpha_i}
\prod_{j=1}^{q'-1}K_y^{1-\beta'_j-\beta_j,\beta_j}
\prod_{k=1}^r \tilde K_{x,y}^{1-\gamma'_k-\gamma_k,\gamma_k}\\
&\quad\prod_{k=1}^{r'-r}K_{x,y}^{1-\gamma'_{r+k}-\alpha_{p'+k-1}-\beta_{q'+k-1},
(\alpha_{p'+k-1},\beta_{q'+k-1})}(1-x)^{-\alpha_p}(1-y)^{-\beta_q}\\
&\qquad \text{if }r\le r',
\\
&=\prod_{i=1}^{p}K_x^{1-\alpha'_i-\alpha_i,\alpha_i}
\prod_{j=1}^{q}K_y^{1-\beta'_j-\beta_j,\beta_j}
\prod_{k=1}^{r'} \tilde K_{x,y}^{1-\gamma'_k-\gamma_k,\gamma_k}\\
&\quad\prod_{k=1}^{r-r'-1}L_{x,y}^{1-\gamma_{r+k}-\alpha'_{p+k}-\beta'_{q+k},
(\alpha'_{p+k},\beta'_{q+k})}(1-x-y)^{-\gamma_r}\quad \text{ if }r>r'.
\end{align*}
\end{thm}

\section{Global structure of hypergeometric functions}\label{sec:Global}
\subsection{Coordinate transformations}\label{sec:coord trans}
By a fractional linear transformation of the Riemann sphere $\mathbb P^1$, 
5 points $x_0,\,x_1,\,x_2,\,x_3$ and $x_4$ in $\mathbb P^1$ 
can be uniquely transformed to 5 points $x,\,y,\,1,\,0$ and $\infty$ .  
Some hypergeometric functions with the variables $(x,y)$
has a symmetry corresponding to the permutation of 5 points $x_0,\,x_1,\,x_2,\,x_3$ and $x_4$
under this correspondence.
This permutation group isomorphic to $S_5$ is
realized by a group of coordinate transformations of $(x,y)$
generated by 4 involutive elements.
\begin{equation}
\label{eq:dynkin}
\begin{split}
&\qquad \mathbb P^1\ni(x_0,x_1,x_2,x_3,x_4)\mapsto (x,y,1,0,\infty)\\
&\begin{tikzpicture}
 [root/.style={draw,circle,inner sep=0mm,minimum size=1.8mm}]
\node[root] (a1) at (0,0) {};
\node[root] (a2) at (1.6,0) {};
\node[root,fill=black] (a3) at (3.2,0) {};
\node[root] (a4) at (4.8,0) {};
\node at (-1.1,-0.35) {$(x,y)\mapsto$};
\node at (-1.1,0.27) {$S_5\ni$};
\node at (0,-0.35)  {$(y,x)$};
\node at (0, 0.25)  {$x_0\leftrightarrow x_1$};
\node at (1.6,-0.35)  {$(\frac xy,\frac 1y)$};
\node at (1.6, 0.25)  {$x_1\leftrightarrow x_2$};
\node at (3.2,-0.35)  {$(1-x,1-y)$};
\node at (3.2, 0.25)  {$x_2\leftrightarrow x_3$};
\node at (4.8,-0.35)  {$(\frac1x,\frac1y)$};
\node at (4.8, 0.25)  {$x_3\leftrightarrow x_4$};
\draw (a1)--(a2);
\draw (a2)--(a3)--(a4);
\end{tikzpicture}
\end{split}
\end{equation}

\noindent 
The subgroup generated by the transformations $(x,y)\mapsto (y,x)$, $(\frac xy,\tfrac 1y)$ and 
$(\frac1x,\frac 1y)$ is isomorphic to $S_3\times S_2$ and defines transformations of the system 
${\mathcal M}^{\boldsymbol\alpha,\boldsymbol\beta,\boldsymbol\gamma}
_{\boldsymbol\alpha',\boldsymbol\beta',\boldsymbol\gamma'}$ 
as follows.

In general, we denote by $T_{(x,y)\to (X,Y)}\mathcal M$
the direct image of a system $\mathcal M$ 
by the transformation $(x,y)\mapsto (X,Y)=(y,x)$, $(\frac1x,\frac1y)$ or 
$(\frac xy,\frac1y)$.
We calculate these transformations below.

\begin{prop}\label{prop:transformations}
In this proposition \eqref{eq:nontrivial} is not assumed.
One has the following transformations of the system $\mathcal M
 ^{\boldsymbol \alpha, \boldsymbol \beta, \boldsymbol \gamma}
 _{\boldsymbol \alpha',\boldsymbol \beta',\boldsymbol \gamma'}$.
\begin{itemize}
\item[\rm (1)] $T_{(x,y)\to(y,x)}
\mathcal M
 ^{\boldsymbol \alpha, \boldsymbol \beta, \boldsymbol \gamma}
 _{\boldsymbol \alpha',\boldsymbol \beta',\boldsymbol \gamma'}
=
\mathcal M
 ^{\boldsymbol \beta, \boldsymbol \alpha, \boldsymbol \gamma}
 _{\boldsymbol \beta',\boldsymbol \alpha',\boldsymbol \gamma'}
.
$
\item[\rm (2)] $T_{(x,y)\to(\tfrac1x,\tfrac1y)}
\mathcal M
 ^{\boldsymbol \alpha, \boldsymbol \beta, \boldsymbol \gamma}
 _{\boldsymbol \alpha',\boldsymbol \beta',\boldsymbol \gamma'}
=
\mathcal M
 ^{\boldsymbol \alpha', \boldsymbol \beta', \boldsymbol \gamma'}
 _{\boldsymbol \alpha,\boldsymbol \beta,\boldsymbol \gamma}.$
 \item[\rm (3)] $T_{(x,y)\to(\epsilon\tfrac xy,\tfrac1y)}
 \mathcal M
 ^{\boldsymbol \alpha, \boldsymbol \beta, \boldsymbol \gamma}
 _{\boldsymbol \alpha',\boldsymbol \beta',\boldsymbol \gamma'}
=
\mathcal M
 ^{\boldsymbol \alpha,\boldsymbol \gamma',\boldsymbol \beta'}
 _{\boldsymbol \alpha',\boldsymbol \gamma,\boldsymbol \beta}$, where $\epsilon:=(-1)^{L}.$
\end{itemize}
\end{prop}

\begin{proof}
(1) is obvious.
Let us prove (2).
Putting $(x,y)=(\frac1X,\frac1Y)$, we have
$\vartheta_x=-\vartheta_X,\ \vartheta_y=-\vartheta_Y$.
The equation $P_1u=0$ is transformed into
\[
 \prod_{i=1}^p(\vartheta_X-\alpha_i)\prod_{k=1}^r(\vartheta_X+\vartheta_Y-\gamma_k)u=
 X\prod_{i=1}^{p'}(\vartheta_X+\alpha'_i)\prod_{k=1}^{r'}(\vartheta_X+\vartheta_Y+\gamma'_k)u.
\]
By similar calculations for $P_2$ and $P_{12}$, we obtain a relation
\begin{equation}
T_{(x,y)\to(\tfrac1x,\tfrac1y)}
\mathcal M
 ^{\boldsymbol \alpha, \boldsymbol \beta, \boldsymbol \gamma}
 _{\boldsymbol \alpha',\boldsymbol \beta',\boldsymbol \gamma'}
=
\mathcal M
 _{\boldsymbol \alpha, \boldsymbol \beta, \boldsymbol \gamma}
 ^{\boldsymbol \alpha',\boldsymbol \beta',\boldsymbol \gamma'}.
\end{equation}
As for (3), we use a relation $\vartheta_x=\vartheta_X$, 
$\vartheta_y=-\vartheta_X - \vartheta_Y$ and the proof is a direct computation.
\end{proof}

By successive applications of the above transformations, we obtain the following theorem.
 
\begin{thm}\label{thm:CoordSym}
One has the following relations.
\begin{gather*}
  T_{(x,y)\to(y,x)}^2=T_{(x,y)\to(\tfrac 1x,\tfrac 1y)}^2=T_{(x,y)\to(\tfrac xy,\tfrac 1y)}^2
  =\textrm{id},\\
  \bigl(T_{(x,y)\to(y,x)}\circ T_{(x,y)\to(\tfrac 1x,\tfrac 1y)}\bigr)^2
  =\bigl(T_{(x,y)\to(\epsilon\tfrac xy,\tfrac 1y)}\circ T_{(x,y)\to(\tfrac1x,\tfrac1y)}\bigr)^2=\textrm{id},\\  
  \Bigl(T_{(x,y)\to(y,x)}\circ T_{(x,y)\to(\epsilon\tfrac xy,\tfrac1y)}\bigr)^3=\textrm{id}.
\end{gather*}
Then the group generated by the involutions $T_{(x,y)\to(y,x)}$, 
$T_{(x,y)\to(\epsilon\tfrac xy,\tfrac 1y)}$ and $T_{(x,y)\to(\tfrac 1x,\tfrac 1y)}$
is isomorphic to $S_3\times S_2$ and the dihedral group of degree 6 with the order 12 which
is the group of symmetries of a regular hexagon. 
This regular hexagon can be realized in a blow-up of $\mathbb P^1\times \mathbb P^1$ as in {\bf Fig.1} (cf. \S \ref{sec:local sol}).
The transformation of 
$\mathcal M
 ^{\boldsymbol \alpha,\boldsymbol \gamma',\boldsymbol \beta'}
 _{\boldsymbol \alpha',\boldsymbol \gamma,\boldsymbol \beta}
$
by the element of this group and the corresponding coordinate transformation is given 
in the following table:

\medskip
Systems satisfied by $(X,Y)$ with $X=X(x,y)$, $Y=Y(x,y)$.
\\
\scalebox{1}
{\begin{tabular}{|c|c|c|c|c|c|c|}\hline
$(X,Y)$&$(x,y)$ & $ (\epsilon\tfrac xy,\tfrac1y)$ & $ (\tfrac1y,\epsilon\tfrac xy)$ & $ 
(\tfrac 1x,\epsilon\tfrac yx)$ & $  (\epsilon\tfrac yx ,\tfrac1x)$ & $ (y,x)$
\\ \hline
$(x,y)$&$(X,Y)$ & $(\epsilon\tfrac XY,\tfrac1Y)$ & $(\epsilon\tfrac YX,\tfrac 1X)$ & $
(\tfrac 1X,\epsilon\tfrac YX)$ &
$(\tfrac1Y,\epsilon\tfrac XY)$ & $(Y,X)$ 
\\ \hline 
\multirow{2}{13mm}
{\!$\mathcal M
 ^{\boldsymbol\alpha,\boldsymbol\beta,\boldsymbol\gamma}
 _{\boldsymbol\alpha',\boldsymbol\beta',\boldsymbol\gamma'}$}
&$\boldsymbol\alpha,\boldsymbol\beta,\boldsymbol\gamma$&
$\boldsymbol\alpha,\boldsymbol\gamma',\boldsymbol\beta'$&
$\boldsymbol\gamma',\boldsymbol\alpha,\boldsymbol\beta'$&
$\boldsymbol\gamma',\boldsymbol\beta,\boldsymbol\alpha'$&
$\boldsymbol\beta,\boldsymbol\gamma',\boldsymbol\alpha'$&
$\boldsymbol\beta,\boldsymbol\alpha,\boldsymbol\gamma$
\\
&$\boldsymbol\alpha',\boldsymbol\beta',\boldsymbol\gamma'$&
$\boldsymbol\alpha',\boldsymbol\gamma,\boldsymbol\beta$&
$\boldsymbol\gamma,\boldsymbol\alpha',\boldsymbol\beta$&
$\boldsymbol\gamma,\boldsymbol\beta',\boldsymbol\alpha$&
$\boldsymbol\beta',\boldsymbol\gamma,\boldsymbol\alpha$&
$\boldsymbol\beta',\boldsymbol\alpha',\boldsymbol\gamma'$

\\ \hline\hline
$(X,Y)$&$(\tfrac1x,\tfrac1y)$ & $ (\epsilon\tfrac yx,y)$ & $ (y,\epsilon\tfrac yx)$ & $ 
(x,\epsilon\tfrac xy)$ & $  (\epsilon\tfrac xy ,x)$ & $ (\tfrac1y,\tfrac1x)$
\\ \hline
$(x,y)$&$(\tfrac1X,\tfrac1Y)$ & $(\epsilon\tfrac YX,Y)$ & $(\epsilon\tfrac XY,X)$ & $
(X,\epsilon\tfrac XY)$ &
$(Y,\epsilon\tfrac YX)$ & $(\tfrac1Y,\tfrac1X)$%
\\ \hline 
\multirow{2}{13mm}
{\!$\mathcal M
 ^{\boldsymbol\alpha,\boldsymbol\beta,\boldsymbol\gamma}
 _{\boldsymbol\alpha',\boldsymbol\beta',\boldsymbol\gamma'}$}
&$\boldsymbol\alpha',\boldsymbol\beta',\boldsymbol\gamma'$&
$\boldsymbol\alpha',\boldsymbol\gamma,\boldsymbol\beta$&
$\boldsymbol\gamma,\boldsymbol\alpha',\boldsymbol\beta$&
$\boldsymbol\gamma,\boldsymbol\beta',\boldsymbol\alpha$&
$\boldsymbol\beta',\boldsymbol\gamma,\boldsymbol\alpha$&
$\boldsymbol\beta',\boldsymbol\alpha',\boldsymbol\gamma'$
\\
&$\boldsymbol\alpha,\boldsymbol\beta,\boldsymbol\gamma$&
$\boldsymbol\alpha,\boldsymbol\gamma',\boldsymbol\beta'$&
$\boldsymbol\gamma',\boldsymbol\alpha,\boldsymbol\beta'$&
$\boldsymbol\gamma',\boldsymbol\beta,\boldsymbol\alpha'$&
$\boldsymbol\beta,\boldsymbol\gamma',\boldsymbol\alpha'$&
$\boldsymbol\beta,\boldsymbol\alpha,\boldsymbol\gamma$
\\\hline
\end{tabular}}
\end{thm}

The transformations in Proposition \ref{prop:transformations} can be used to construct a local solution at a normal crossing point from that at another point.
For example, applying the transformation $(x,y)\mapsto(\tfrac1x,\tfrac1y)$ to local solutions \eqref{eq:Sols} at the origin, we obtain $pq$ solutions 
\[
  F^{p',q',r'}_{p,q,r}\Bigl(
\begin{smallmatrix}
 \boldsymbol \alpha' & \boldsymbol \beta' & \boldsymbol \gamma'\\
 \boldsymbol \alpha & \boldsymbol \beta & \boldsymbol \gamma
\end{smallmatrix};\alpha_i,\beta_j;
 \tfrac1x,\tfrac1y\Bigr)\qquad
(1\le i\le p,\ 1\le j\le q)
\]
in a neighborhood of  $(\infty,\infty)$. 
In the same manner, applying the transformation $(x,y)\mapsto(\epsilon\tfrac xy,\tfrac1y)$ to local solutions \eqref{eq:Sols} at the origin, we obtain $p'r$ solutions
\begin{gather}
  F^{p,r',q'}_{p',r,q}\Bigl(
\begin{smallmatrix}
 \boldsymbol \alpha & \boldsymbol \gamma' & \boldsymbol \beta'\\
 \boldsymbol \alpha' & \boldsymbol\gamma & \boldsymbol \beta
\end{smallmatrix};\alpha'_i,\gamma_k;
 \epsilon\tfrac xy,\tfrac1y\Bigr)\qquad(1\le i\le p',\ 1\le k\le r)
\notag
\end{gather}
which are defined along a line $y=\infty$.

Let us consider another transformation $(x,y)\mapsto(X,Y)=(x,\tfrac1y)$.
Then $\vartheta_x=\vartheta_X$, $\vartheta_y=-\vartheta_Y$ and the system 
${\mathcal M}^{\boldsymbol\alpha,\boldsymbol\beta,\boldsymbol\gamma}
_{\boldsymbol\alpha',\boldsymbol\beta',\boldsymbol\gamma'}$ in $(X,Y)$ coordinate is
\begin{align*}
&\begin{cases}
  \displaystyle
  \prod_{i=1}^{p'}(\vartheta_X-\alpha'_i)\prod_{k=1}^{r'}(\vartheta_X-\vartheta_Y-\gamma'_k)u=
    X\prod_{i=1}^p(\vartheta_X+\alpha_i)\prod_{k=1}^r(\vartheta_X-\vartheta_Y+\gamma_k)u,\\
  \displaystyle\prod_{j=1}^{q}(\vartheta_Y-\beta_j)\prod_{k=1}^{r}(\vartheta_X-\vartheta_Y+\gamma_k)u=\epsilon Y\prod_{j=1}^{q'}(\vartheta_Y+\beta_j')\prod_{k=1}^{r'}(\vartheta_X-\vartheta_Y-\gamma'_k)u,\\
 \displaystyle
  \prod_{i=1}^{p'}(\vartheta_X-\alpha'_i)\prod_{j=1}^{q}(\vartheta_Y-\beta_j)u
  =\epsilon XY
  \prod_{i=1}^p(\vartheta_X+\alpha_i)
  \prod_{j=1}^{q'}(\vartheta_Y+\beta'_j) u.
 \end{cases}
\end{align*}
Hence, we obtain the following proposition.
\begin{prop}
One has a relation
\begin{equation}
T_{(x,y)\to(x,\tfrac \epsilon y)}\mathcal M^{\boldsymbol \alpha, \boldsymbol \beta, \boldsymbol \gamma}
 _{\boldsymbol \alpha',\boldsymbol \beta',\boldsymbol \gamma'}
=\mathcal N^{\boldsymbol \alpha, \boldsymbol \beta', \boldsymbol \gamma}
 _{\boldsymbol \alpha',\boldsymbol \beta,\boldsymbol \gamma'},\label{eq:M2N}
\end{equation}
where the system $\mathcal N^{\boldsymbol \alpha, \boldsymbol \beta', \boldsymbol \gamma}
 _{\boldsymbol \alpha',\boldsymbol \beta,\boldsymbol \gamma'}$ is defined as follows:
\begin{align*}
&\begin{cases}
  (\vartheta_x-\boldsymbol{\alpha}')(\vartheta_x-\vartheta_y-\boldsymbol{\gamma}')u=
    x(\vartheta_x+\boldsymbol{\alpha})(\vartheta_x-\vartheta_y+\boldsymbol{\gamma})u,\\
  (\vartheta_y-\boldsymbol{\beta})(\vartheta_x-\vartheta_y-\boldsymbol{\gamma})u=y (\vartheta_y-\boldsymbol{\beta}')(\vartheta_x-\vartheta_y-\boldsymbol{\gamma}')u,\\
(\vartheta_x-\boldsymbol{\alpha}')(\vartheta_y-\boldsymbol{\beta})u
  =xy
  (\vartheta_x+\boldsymbol{\alpha})
  (\vartheta_y-\boldsymbol{\beta}')u.
 \end{cases}
\end{align*}
\end{prop}

The local solution to this system in a neighborhood of the origin in $(X,Y)$ coordinate corresponds
to a local solution in a neighborhood of 
$(0,\infty)\in \mathbb P^1\times \mathbb P^1$.  
In fact, we obtain the following $p'q$ solutions to $\mathcal M^{\boldsymbol \alpha,\boldsymbol \beta,\boldsymbol \gamma}_{\boldsymbol \alpha',\boldsymbol \beta',\boldsymbol \gamma'}$:
\begin{align*}
&G^{p,q',r}_{p',q,r'}\left(\begin{smallmatrix}
\boldsymbol\alpha&\boldsymbol\beta'&\boldsymbol\gamma\\
 \boldsymbol\alpha'&\boldsymbol\beta&\boldsymbol\gamma'
\end{smallmatrix};\alpha'_i,\beta_j;X,\epsilon Y\right)\\
&\qquad:=x^{\alpha'_i}(\tfrac\epsilon y)^{\beta_j}
G^{p,q',r}_{p',q,r'}\left(\begin{smallmatrix}
\boldsymbol\alpha+\alpha'_i&\boldsymbol\beta'+\beta_j&\boldsymbol\gamma+\alpha'_i-\beta_j\\
 \boldsymbol\alpha'-\alpha'_i&\boldsymbol\beta-\beta_j&\boldsymbol\gamma'-\alpha'_i+\beta_j
\end{smallmatrix};x,\tfrac \epsilon y\right).
\end{align*}
For later use, we set
\begin{align*}
    &G^{p,q',r}_{p',q,r'}\left(\begin{smallmatrix}
\boldsymbol\alpha&\boldsymbol\beta'&\boldsymbol\gamma\\
 \boldsymbol\alpha'&\boldsymbol\beta&\boldsymbol\gamma'
\end{smallmatrix};\alpha'_i,\beta_j;\epsilon_1,\epsilon_2;X,\epsilon Y\right)\\
&\qquad:=
(\epsilon_1x)^{\alpha'_i}(\epsilon_2\tfrac\epsilon y)^{\beta_j}
G^{p,q',r}_{p',q,r'}\left(\begin{smallmatrix}
\boldsymbol\alpha+\alpha'_i&\boldsymbol\beta'+\beta_j&\boldsymbol\gamma+\alpha'_i-\beta_j\\
 \boldsymbol\alpha'-\alpha'_i&\boldsymbol\beta-\beta_j&\boldsymbol\gamma'-\alpha'_i+\beta_j
\end{smallmatrix};x,\tfrac \epsilon y\right)
\end{align*}
for $1\le i\le p',\ 1\le j\le q$ and $\epsilon_1,\epsilon_2=\pm 1$.
\subsection{Local Solutions}\label{sec:local sol}
We consider 
$\mathcal M^{\boldsymbol \alpha,\boldsymbol \beta,\boldsymbol \gamma}_{\boldsymbol \alpha',\boldsymbol \beta',\boldsymbol \gamma'}$
on the space $\tilde X$ which is constructed from 
$\mathbb P^1\times\mathbb P^1$  by blowing up two points $(0,0)$ and 
$(\infty,\infty)$.
The local coordinates are $(\frac xy,y)$ and $(x,\frac yx)$ at $(0,0)$ 
and they are $(\frac yx,\frac1y)$ and $(\frac1x,\frac xy)$ at $(\infty,\infty)$.
Then for example, $(0^2,0)$ represents the point of $\tilde X$ corresponding to the origin 
under the coordinate $(\frac xy,y)$. 
Regarding {\bf Fig.1} below, the transformations 
$(x,y)\mapsto (y,\epsilon\tfrac yx)$  and $(x,y)\mapsto(y,x)$ 
correspond to the clockwise rotation of the hexagon 
by $\frac \pi 3$ and the reflection of the hexagon keeping the edge $(0,0)$.

\smallskip
\centerline{\begin{tikzpicture}
\draw
node (A) at (1,0) {$(0,0^2)$}
node (B) at (0,1) {$(0^2,0)$}
node (C) at (0,4) {$(0,\infty)$}
node (D) at (4,4) {$(\infty,\infty^2)$}
node (E) at (6,3) {$(\infty^2,\infty)$}
node (F) at (6,0) {$(\infty,0)$}
(A)--(B)--(C)--(D)--(E)--(F)--(A)
node at (-0.2,0.4) {\scriptsize$x_2=x_4$\ \eqref{eq:F0-0}}
node at (0.8,1) {\scriptsize\eqref{eq:G00-0}}
node at (1.1,0.4) {\scriptsize\eqref{eq:G0-00}}
node at (3.4,-0.2) {\scriptsize\eqref{eq:F-0}}
node at (3.5,0.15) {\scriptsize$x_1=x_3$}
node at (5.7,0.3) {\scriptsize\eqref{eq:G4-0}}
node at (-0.3,2.5) {\scriptsize\eqref{eq:F0-}}
node at (0.6,2.5) {\scriptsize$x_0=x_3$}
node at (6.3,1.5) {\scriptsize\eqref{eq:F4-}}
node at (4,3.7) {\scriptsize\eqref{eq:G4-44}} 
node at (5.4,1.5) {\scriptsize$x_0=x_4$}
node at (4.7,3.4) {\scriptsize\eqref{eq:F4-4}} 
node at (5.55,3.6) {\scriptsize$x_2=x_3$}
node at (0.3,3.65) {\scriptsize\eqref{eq:G0-4}} 
node at (1.95,4.2) {\scriptsize\eqref{eq:F-4}}   
node at (1.95,3.8) {\scriptsize$x_1=x_4$} 
node at (5.0,2.95) {\scriptsize\eqref{eq:G44-4}} 
node at (3,4.8) {local solutions of 12 types}
node at (3,5.3) {\large\bf Fig.1}
;
\end{tikzpicture}}

\noindent
Owing to the coordinate transformations in \S\ref{sec:coord trans}, 
we have the following local solutions to $\mathcal M^{\boldsymbol \alpha,\boldsymbol \beta,\boldsymbol \gamma}_{\boldsymbol \alpha',\boldsymbol \beta',\boldsymbol \gamma'}$:
for simplicity, we sometimes omit the suffices and superfices of $F^{p,q,r}_{p',q',r'}$ and 
$G^{p,q,r}_{p',q',r'}$ and write $\alpha$, $\beta,\ldots$ in place of 
$\boldsymbol\alpha$, $\boldsymbol\beta,\ldots$. 
\begin{align}
 p'q'&: F\left(\begin{smallmatrix}
  \alpha&\beta&\gamma\\\alpha'&\beta'&\gamma'
\end{smallmatrix};\alpha'_i,\beta'_j;x,y\right)
 && \text{around $(0,0)$}
 \tag{1}\label{eq:F0-0}
\allowdisplaybreaks\\
pq&:F\left(\begin{smallmatrix}
  \alpha'&\beta'&\gamma'&\\\alpha&\beta&\gamma
 \end{smallmatrix};\alpha_i,\beta_j;\tfrac1x,\tfrac1y\right)
 &&\text{around $(\infty,\infty)$}
\tag{2}\label{eq:F4-4}
\allowdisplaybreaks\\
 qr'&:F\left(\begin{smallmatrix}
  \gamma&\beta'& \alpha\\\gamma'&\beta&\alpha'&
\end{smallmatrix};\gamma'_k,\beta_j;x,\epsilon\tfrac xy\right)
 &&\text{along $x=0$}
\tag{3}\label{eq:F0-}
\allowdisplaybreaks\\
pr'&:F\left(\begin{smallmatrix}
  \alpha'&\gamma& \beta\\ \alpha&\gamma'&\beta'
\end{smallmatrix};\alpha_i,\gamma'_k;\epsilon\tfrac yx,y\right)
  &&\text{along $y=0$}
\tag{4}\label{eq:F-0}
\allowdisplaybreaks\\
p'r&: F\left(\begin{smallmatrix}
  \alpha&\gamma'&\beta'&\\\alpha'&\gamma& \beta
  \end{smallmatrix};\alpha'_i,\gamma_k;\epsilon\tfrac xy,\tfrac1y\right)
 &&\text{along $y=\infty$}\label{eq:F-4}
\tag{5}
\allowdisplaybreaks\\ 
 q'r&: F\left(\begin{smallmatrix}
  \gamma'&\beta&\alpha'\\\gamma&\beta'& \alpha
\end{smallmatrix};\gamma_k,\beta'_j;\tfrac1x,\epsilon\tfrac yx\right)
 && \text{along $x=\infty$}\label{eq:F4-}
\tag{6}
\allowdisplaybreaks\\ 
p'q&: 
 G\left(\begin{smallmatrix}
  \alpha&\beta'&\gamma\\\alpha'&\beta&\gamma'
 \end{smallmatrix};\alpha'_i,\beta_k;x,\epsilon\tfrac1y\right)
 && \text{at }(0,\infty)
 \tag{7}\label{eq:G0-4}
\allowdisplaybreaks\\
 pq'&:G\left(\begin{smallmatrix}
  \alpha'&\beta&\gamma'\\ \alpha&\beta'&\gamma
\end{smallmatrix};\alpha_i,\beta'_j;\tfrac1x,\epsilon y\right)
 &&\text{at }(\infty,0)
\tag{8}\label{eq:G4-0}
\allowdisplaybreaks\\
p'r'&:G\left(\begin{smallmatrix}
  \alpha&\gamma& \beta'\\ \alpha'&\gamma'&\beta
\end{smallmatrix};\alpha'_i,\gamma'_k;\epsilon\tfrac xy,\epsilon y\right)
 &&\text{at }(0^2,0)
 \tag{9}\label{eq:G00-0}
 \allowdisplaybreaks\\
 q'r'&:G\left(\begin{smallmatrix}
  \gamma&\beta&\alpha\\ \gamma'&\beta'&\alpha'
\end{smallmatrix} ;\gamma'_k,\beta'_j;x,\tfrac yx\right)
  &&\text{at }(0,0^2)
\tag{10}\label{eq:G0-00}
\allowdisplaybreaks\\
qr&:G\left(\begin{smallmatrix}
  \gamma'&\beta'&\alpha'\\\gamma&\beta&\alpha
\end{smallmatrix};\gamma_k,\beta_j;\tfrac1x,\tfrac xy\right)
 &&\text{at $(\infty,\infty^2)$}\label{eq:G4-44}
\tag{11}\allowdisplaybreaks\\
 pr&:G\left(\begin{smallmatrix}
  \alpha'&\gamma'&\beta\\ \alpha&\gamma&\beta'
\end{smallmatrix};\alpha_i,\gamma_k;\epsilon\tfrac1x,\epsilon\tfrac xy\right)
  &&\text{at }(\infty^2,\infty)
\tag{12}\label{eq:G44-4}
\end{align}

\noindent
For example, Theorem~\ref{thm:CoordSym} and \eqref{eq:M2N} show a relation
\begin{align*}
T_{(x,y)\to(x,\tfrac y x)}
\mathcal M
^{\boldsymbol \alpha,\boldsymbol \beta,\boldsymbol \gamma}
_{\boldsymbol \alpha',\boldsymbol \beta',\boldsymbol \gamma'}
&=T_{(x,y)\to(x,\tfrac \epsilon y)}
\circ T_{(x,y)\to(x,\epsilon\tfrac yx)}
\mathcal M
^{\boldsymbol \alpha,\boldsymbol \beta,\boldsymbol \gamma}
_{\boldsymbol \alpha',\boldsymbol \beta',\boldsymbol \gamma'}\\
&=T_{(x,y)\to(x,\tfrac \epsilon y)}
\mathcal M
^{\boldsymbol \gamma,\boldsymbol \beta',\boldsymbol \alpha}
_{\boldsymbol \gamma',\boldsymbol \beta,\boldsymbol \alpha'}\\
&=\mathcal N
^{\boldsymbol \gamma,\boldsymbol \beta,\boldsymbol \alpha}
_{\boldsymbol \gamma',\boldsymbol \beta',\boldsymbol \alpha'}
\end{align*}
and then we have the above $q'r'$ solutions \eqref{eq:G0-00}.

\medskip
Each solution in (7)--(12) expressed by $G$ is defined in a neighborhood of a normally crossing
singular point in $\tilde X$, namely, one of the 6 vertices in \textbf{Fig.1} 
in \S\ref{sec:connection}.  The solution in (1)--(6) expressed by $F$ 
is defined in a neighborhood of one of 6 singular lines in $\tilde X$ excluding normally crossing singular points, namely, one of six edges in \textbf{Fig.1}.

We examine the independent 
local solutions at a singular point in $\tilde X$ and calculate the numbers of 
the solutions in the following table:

\medskip\noindent\ 
\begin{tabular}{|c|c|c|c|}\hline
 & $(\infty,\infty^2)$ & $(0,\infty)$ & $(0^2,0)$
\\ \hline
sol.
 & \eqref{eq:F-4}$+$\eqref{eq:G4-44}$+$\eqref{eq:F4-4}
 & \eqref{eq:G0-4}$+$\eqref{eq:F-4}$+$\eqref{eq:F0-}
 & \eqref{eq:F0-0}$+$\eqref{eq:F0-}$+$\eqref{eq:G00-0}
\\ \hline
\# & $p'r+qr+pq$ & $p'q+p'r+qr'$ & $p'q'+qr'+p'r'$ 
\\ \hline
(F1) & $pr+qr+pq$ & $pq+pr+qr$ & $pq+qr+pr$ \\
(F2) & {\small $p+(q\!-\!1)+(p\!-\!1)(q\!-\!1)$} & $p(q\!-\!1)+p+0$ & $pq+0+0$
\\
(F3) & $0+0+pq$ & $(p\!-\!1)q+0+q$ & {\small $\!(p\!-\!1)(q\!-\!1)\!+\!q\!+\!(p\!-\!1)\!$}\\
(F4) & {\small$\!2p\!+\!2(q\!-\!2)\!+\!(p-2)(q-2)\!$} & $p(q\!\!-\!2)+2p+0 $ & $pq+0+0$
\\ \hline
\end{tabular}

\medskip
\noindent
For example, this table indicates that there exist $p'q'$ solutions by (1), $qr'$ solutions by (3) and $p'r'$ solutions by (9) at $(0^2,0)$.
Hence, for example, under the condition (F2), all the solutions we get at $(0^2,0)$ belong to (1) and then such solutions are defined around $(0,0)$, which are called solutions with a simple 
monodromy at $(0,0)$ in 
\cite[Definition 7.3]{Oi}.

For the points $(\infty^2,\infty)$, $(\infty,0)$ and $(0,0^2)$, we obtain the corresponding
table by the map $(x,y)\mapsto(y,x)$.  Namely we change $(p,q,r;p',q',r')$, $(2m)$ and $(2m-1)$ to
$(q,p,r;q',p',r')$, $(2m-1)$, $(2m)$ for $m=2,\dots,6$, respectively, in the above table.

\bigskip
Since $p'-p=q'-q=r-r'$, we have
\begin{align*}
\begin{split}
p'r+qr+pq&=p'q+p'r+qr'=p'q'+qr'+p'r'\\
&=pq+qr+rp+(p+q+r)(r-r')+(r-r')^2,
\end{split}
\end{align*}
which equals $R\left(\begin{smallmatrix}p&q&r\\p'&q'&r'\end{smallmatrix}\right)$ in \eqref{eq:dim}.  
Thus, we obtain the following theorem.

\begin{thm}\label{thm:LocalSols}
In a neighborhood of any one of 
6 points $(0^2,0)$, $(0,\infty)$, $(\infty,\infty^2)$, $(\infty^2,\infty)$, $(\infty,0)$ 
and $(0,0^2)$ of $\tilde X$, one has
\begin{equation}\label{eq:dim}
 R\left(\begin{smallmatrix}p&q&r\\p'&q'&r'\end{smallmatrix}\right):=
\begin{cases}
 \frac{1}{L}(p'q'r-pqr')&(L\ne 0),\\
 pq+qr+rp &(L=0)
 \end{cases}
\end{equation}
local solutions expressed by the series $F^{p,q,r}_{p',q',r'}$ and $G^{p,q,r}_{p',q',r'}$
given in \S\ref{sec:Intro} and \S\ref{sec:Known series}.
\end{thm}

\begin{rem}
It will be shown in \S\ref{sec:rank} that 
the solutions in this theorem
span the space of all local solutions to $\mathcal M^{\boldsymbol \alpha,\boldsymbol \beta,\boldsymbol \gamma}_{\boldsymbol \alpha',\boldsymbol \beta',\boldsymbol \gamma'}$ 
(cf.~Corollary~\ref{cor:spansol}) for generic parameters.
\end{rem}
\subsection{Connection Problem}\label{sec:connection}
We consider the connection problem between the local solutions in Theorem~\ref{thm:LocalSols}.
Since the local solutions are constructed according to the symmetry \eqref{eq:dynkin} and {\bf Fig.1}, it is sufficient only to calculate the connection problem 
from the point $(0^2,0)$ to $(0,\infty)$ along a suitable path 
\begin{equation}
(0,\infty)\ni t\mapsto (c(t),-t)\in(0,\infty)\times (-\infty,0)
\label{eq:path}
\end{equation}
with small $c(t)$ so that the system
$\mathcal M
 ^{\boldsymbol \alpha, \boldsymbol \beta, \boldsymbol \gamma}
 _{\boldsymbol \alpha',\boldsymbol \beta',\boldsymbol \gamma'} 
$
has no singularities in 
\begin{equation}\label{eq:nbd}
 U=\{(x,y)\mid 0<x<2c(-y),\ y\in(-\infty,0)\}.
\end{equation}
It follows from Theorem~\ref{thm:Sing} that 
there exists a connected neighborhood $U$ of 
$\{0\}\times (0,-\infty)$ such that 
$
\mathcal M
 ^{\boldsymbol \alpha, \boldsymbol \beta, \boldsymbol \gamma}
 _{\boldsymbol \alpha',\boldsymbol \beta',\boldsymbol \gamma'} 
$
has no singularities in $U\setminus\{0\}\times (0,-\infty)$.

The system
$\mathcal M
 ^{\boldsymbol \alpha, \boldsymbol \beta, \boldsymbol \gamma}
 _{\boldsymbol \alpha',\boldsymbol \beta',\boldsymbol \gamma'} 
$
has regular singularities along the wall $\{0\}\times (-\infty,0)$ and 
its solution is ideally analytic (cf.~\cite[Definition~5.1, see also Theorem 5.3]{Ob}).
More precisely, the system has regular singularities along the set of 
walls $x=0$ and $L_{(0,0)}$ in $\tilde X$ with the edge $(0^2,0)$.
Here $L_{(0,0)}$ is the exceptional fiber of the blowing up of the origin.
In general, if the edge is a point, the condition (5.3) in 
\cite{Ob} is satisfied and then it follows from 
\cite[Theorem~5.2 and 5.3]{Ob} that any solution at the edge  
is in the spanning subspace of local solutions with series expansion. 
Then the connection problem is reduced to that of the boundary values
of the solution in the sense of \cite{OS} and the study of induced equation
on the boundary in the sense of \cite{Ob}.
A schematic diagram for the connection problem between $(0^2,0)$ and $(0,\infty)$ is the following: we write $\mathcal Sol(\mathcal M
 ^{\boldsymbol \alpha, \boldsymbol \beta, \boldsymbol \gamma}
 _{\boldsymbol \alpha',\boldsymbol \beta',\boldsymbol \gamma'} )|_{x=0}$ for the space of (single-valued) holomorphic solutions to $\mathcal M
 ^{\boldsymbol \alpha, \boldsymbol \beta, \boldsymbol \gamma}
 _{\boldsymbol \alpha',\boldsymbol \beta',\boldsymbol \gamma'}$ near a generic point $(0,y)\in\C^2$ on the divisor $\{ x=0\}$.
 Let $\mathcal M
 ^{\boldsymbol \alpha, \boldsymbol \beta, \boldsymbol \gamma}
 _{\boldsymbol \alpha',\boldsymbol \beta',\boldsymbol \gamma'} |_{x=0}$ be the restriction of the system $\mathcal M
 ^{\boldsymbol \alpha, \boldsymbol \beta, \boldsymbol \gamma}
 _{\boldsymbol \alpha',\boldsymbol \beta',\boldsymbol \gamma'} |_{x=0}$ in the sense of $D$-modules.
 Then, the unique solvability of the boundary value problem can be summarized as the commutativity of the following diagram:

\begin{center}
\begin{tikzpicture}
\node at (0,0){$\mathcal Sol(\mathcal M
 ^{\boldsymbol \alpha, \boldsymbol \beta, \boldsymbol \gamma}
 _{\boldsymbol \alpha',\boldsymbol \beta',\boldsymbol \gamma'} )|_{x=0}$};
\node at (7,0){$\mathcal Sol(\mathcal M
 ^{\boldsymbol \alpha, \boldsymbol \beta, \boldsymbol \gamma}
 _{\boldsymbol \alpha',\boldsymbol \beta',\boldsymbol \gamma'} )|_{x=0}$};
 \draw[->] (2,0) -- (5,0); 
 \node at (3.5,0.3){$y:0\rightsquigarrow\infty$};
 \node at (0,-2){$\mathcal Sol(\mathcal M
 ^{\boldsymbol \alpha, \boldsymbol \beta, \boldsymbol \gamma}
 _{\boldsymbol \alpha',\boldsymbol \beta',\boldsymbol \gamma'} |_{x=0})$};
\node at (7,-2){$\mathcal Sol(\mathcal M
 ^{\boldsymbol \alpha, \boldsymbol \beta, \boldsymbol \gamma}
 _{\boldsymbol \alpha',\boldsymbol \beta',\boldsymbol \gamma'} |_{x=0})$.};
 \draw[->] (2,-2) -- (5,-2); 
 \node at (3.5,-1.7){$y:0\rightsquigarrow\infty$};
 \draw[->] (0,-0.5) -- (0,-1.5);
 \draw[->] (7,-0.5) -- (7,-1.5);
\node at (-0.6,-1){$x\rightsquigarrow 0$};
\node at (7.6,-1){$x\rightsquigarrow 0$};
\end{tikzpicture}
\end{center}
Here, the vertical arrow is the boundary value map onto the divisor $\{ x=0\}$ and the horizontal arrow is the analytic continuation along the path \eqref{eq:path}.
Note that the boundary value map is an isomorphism.
Connection formula for other exponents can be derived in the same manner by multiplying suitable power of $x$ to the unknown function as far as no integral difference appears among exponents.

\medskip
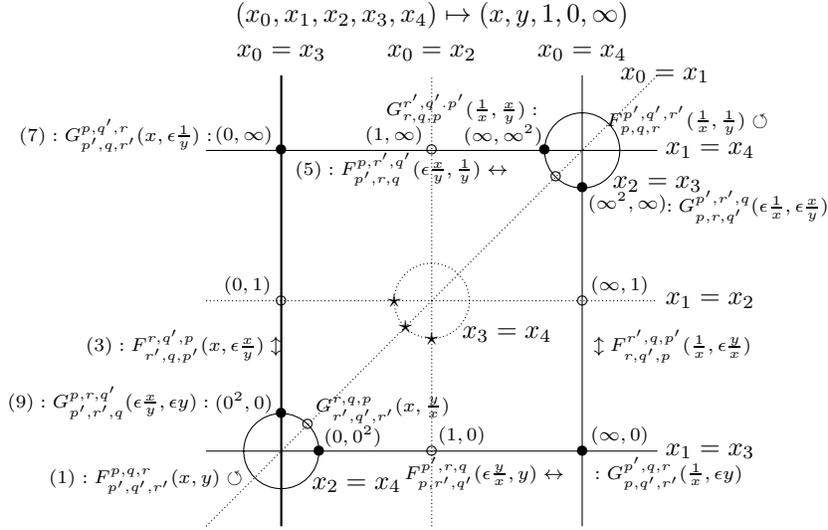
\begin{tikzpicture}
\draw[densely dotted]  (2,-1)--(2,5)  (-1,-1)--(5,5) (-1,2)--(5,2)
(2,2) circle[radius=0.5]
;
\draw [thick 
] (0,-1)--(0,5) ;
\draw (-1,4)--(5,4)  (4,-1)--(4,5)  (-1,0)--(5,0)
node at (5.7,0) {$x_1=x_3$}
node at (5.7,2) {$x_1=x_2$}
node at (5.7,4) {$x_1=x_4$}
node at (0,5.3) {$x_0=x_3$}
node at (2,5.3) {$x_0=x_2$}
node at (4,5.3) {$x_0=x_4$}
node at (1,-0.45) {$x_2=x_4$}
node at (3,1.55) {$x_3=x_4$}
node at (5,3.55) {$x_2=x_3$}
node at (5.1,5) {$x_0=x_1$}
(4,4) circle[radius=0.5]
(0,0) circle[radius=0.5]
node at (4,0)   {$\bullet$}
node at (0,4)   {$\bullet$}
node at (0,0.5) {$\bullet$}
node at (0.5,0) {$\bullet$}
node at (4,3.5) {$\bullet$}
node at (3.5,4) {$\bullet$}
node at (2,0)   {$\circ$}
node at (0,2)   {$\circ$}
node at (2,4)   {$\circ$}
node at (4,2)   {$\circ$}
node at (0.35,0.35) {$\circ$}
node at (3.65,3.65) {$\circ$}
node at (1.5,2) {$\star$}
node at (2,1.5) {$\star$}
node at (1.65,1.65) {$\star$}
node at (-0.5,4.2) {\scriptsize$(0,\infty)$}
node at (-0.45,2.2) {\scriptsize$(0,1)$}
node at (2.4,0.2) {\scriptsize$(1,0)$}
node at (1.5,4.2) {\scriptsize$(1,\infty)$}
node at (4.5,0.2) {\scriptsize$(\infty,0)$}
node at (4.5,2.2) {\scriptsize$(\infty,1)$}
node at (-0.5,0.65) {\scriptsize$(0^2,0)$}
node at (0.95,0.2) {\scriptsize$(0,0^2)$}
node at (2.95,4.2) {\scriptsize$(\infty,\infty^2)$}
node at (4.6,3.3) {\scriptsize$(\infty^2,\infty)$}
node at (-1.8,-0.4) {
\scriptsize$(1): F^{p,q,r}_{p',q',r'}(x,y)\circlearrowleft$}
node at (-1.3,1.4) {
\scriptsize$(3):F^{r,q',p}_{r',q,p'}(x,\epsilon\frac xy)\updownarrow$}
node at (-2.3,0.65) {
\scriptsize$(9):G^{p,r,q'}_{p',r',q}(\epsilon\frac xy,\epsilon y):$}
node at (-2.2,4.2) {
\scriptsize$(7):G^{p,q',r}_{p',q,r'}(x,\epsilon \frac 1y):$}
node at (1.6,3.75) {
\scriptsize$(5):F^{p,r',q'}_{p',r,q}(\epsilon \frac xy,\frac 1y)
\leftrightarrow$}
node at (2.4,4.55) {\scriptsize$G^{r',q'.p'}_{r,q,p}(\frac1x, \frac xy):$}
node at (5.4,4.4) {\scriptsize$F^{p',q',r'}_{p,q,r}(\frac 1x,\frac 1y)\circlearrowleft$}
node at (6.2,3.25) {\scriptsize$:G^{p',r',q}_{p,r,q'}(\epsilon\frac 1x,\epsilon\frac xy)$}
node at (2.7,-0.3) {\scriptsize$F^{p',r,q}_{p,r',q'}(\epsilon \frac yx,y)\leftrightarrow$}
node at (5.2,1.4) {\scriptsize$\updownarrow F^{r',q,p'}_{r,q',p}(\frac 1x,\epsilon\frac yx)$}
node at (5.15,-0.3) {\scriptsize$:G^{p',q,r}_{p,q',r'}(\frac1x, \epsilon y)$}
node at (1.35,0.55)  {\scriptsize$G^{r,q,p}_{r',q',r'}(x,\tfrac yx)$}
;
\draw 
node at (2,5.8) {$(x_0,x_1,x_2,x_3,x_4)\mapsto (x,y,1,0,\infty)$};
\end{tikzpicture}

The space of local solutions at $(0^2,0)$ (resp.~$(0,\infty))$
is spanned by the solutions \eqref{eq:F0-0}, \eqref{eq:F0-} 
and \eqref{eq:G00-0}
(resp.~\eqref{eq:G0-4}, \eqref{eq:F-4} and \eqref{eq:F0-})
given in \S\ref{sec:local sol}.
Since the solutions \eqref{eq:F0-} are globally defined along $x=0$, 
the connection problem is reduced to the problem expressing 
a solution in \eqref{eq:F0-0} or \eqref{eq:G00-0} by a linear 
combination of the solutions in \eqref{eq:F-4} and \eqref{eq:G0-4}.
These  solutions are classified by the characteristic exponent 
$\alpha'_i$ along the regular singularity $x=0$.

An ideally analytic solution along $x=0$ with the characteristic exponent
$\alpha'_i$ has the form
\[
  x^{\alpha'_i}\phi(x,y).
\]
Here $\phi(x,y)$ is holomorphic in a neighborhood of $\{0\}\times (-\infty,0)$
and its boundary value with respect to the characteristic exponent $\alpha' _i$
is $\phi(0,y)$.
Those solutions are given as follows under the notation in \S\ref{sec:HG1}.
Here we note that $q+r=q'+r'$ and $-\epsilon$ means $-$ or $+$ according to 
$\epsilon=1$ or $-1$, respectively.
{
\begin{align*}\eqref{eq:F0-0}: \ 
&F\left(\begin{smallmatrix}
  \alpha&\beta&\gamma\\\alpha'&\beta'&\gamma'
\end{smallmatrix};\alpha'_i,\beta'_j;+,-;x,y\right)\quad \text{around }(0,0)
\\
&\quad=x^{\alpha'_i}(-y)^{\beta'_j}F\left(\begin{smallmatrix}
\boldsymbol\alpha+\alpha'_i & \boldsymbol\beta+\beta'_j&\boldsymbol\gamma+\alpha'_i+\beta'_j\\
\boldsymbol\alpha'-\alpha'_i & \boldsymbol\beta'-\beta'_j&\boldsymbol\gamma-\alpha'_i-\beta'_j\end{smallmatrix}
;x,y\right)
\\&\quad
\overset{x\to0}\sim
x^{\alpha'_i}F_{q+r}\left(\begin{smallmatrix}
 \boldsymbol\beta,\boldsymbol\gamma+\alpha'_i\\
 \boldsymbol\beta',\boldsymbol\gamma'-\alpha'_i
\end{smallmatrix};\beta'_j;-;y\right)
\qquad(1\le i\le p',\ 1\le j\le q')
,\allowdisplaybreaks\\
\eqref{eq:G00-0}: \ 
&G\left(\begin{smallmatrix}
 \boldsymbol\alpha&\boldsymbol\gamma& \boldsymbol\beta'\\ 
\boldsymbol\alpha'&\boldsymbol\gamma'&\boldsymbol\beta
\end{smallmatrix};\alpha'_i,\gamma'_k;-\epsilon,-;\epsilon\tfrac xy,\epsilon y\right)\quad\text{at }(0^2,0)
\allowdisplaybreaks\\
&\quad=(-\tfrac xy)^{\alpha'_i}(-y)^{\gamma'_k}
G\left(\begin{smallmatrix}
  \boldsymbol\alpha+\alpha'_i&\boldsymbol\gamma+\gamma'_k& \boldsymbol\beta'+\alpha'_i-\gamma'_k\\
 \boldsymbol\alpha'-\alpha'_i&\boldsymbol\gamma'-\gamma'_k&\boldsymbol\beta-\alpha'_i+\gamma'_k
\end{smallmatrix};\epsilon\tfrac xy,\epsilon y\right)\\
&\quad\overset{x\to0}{\sim} 
x^{\alpha'_i} (-y)^{\gamma'_k-\alpha'_i}F_{q+r}\left(\begin{smallmatrix}
\boldsymbol\gamma+\gamma'_k,\boldsymbol\beta-\alpha'_i+\gamma'_k\\
\boldsymbol\gamma'-\gamma'_k,\boldsymbol\beta'+\alpha'_i-\gamma'_k
\end{smallmatrix};y
\right)\\
&\quad=x^{\alpha'_i}
F_{q+r}\left(\begin{smallmatrix}
\boldsymbol\beta,\boldsymbol\gamma+\alpha'_i\\
\boldsymbol\beta',\boldsymbol\gamma'-\alpha'_i
\end{smallmatrix};\gamma'_k-\alpha'_i;-;y\right)
\qquad(1\le i\le p',\ 1\le k\le r')
,\allowdisplaybreaks\\
\eqref{eq:F-4}: \ &F\left(\begin{smallmatrix}
  \boldsymbol\alpha&\boldsymbol\gamma'&\boldsymbol\beta'&\\
  \alpha'&\gamma& \beta
\end{smallmatrix};\alpha'_i,\gamma_k;-\epsilon,-1;\epsilon\tfrac xy,\tfrac1y\right)
\quad\text{along }y=\infty
\\
&\quad=(-\tfrac xy)^{\alpha'_i}(-\tfrac 1y)^{\gamma_k}
F\left(\begin{smallmatrix}
  \boldsymbol\alpha+\alpha'_i&\boldsymbol\gamma'+\gamma'_k
  &\boldsymbol\beta'+\alpha'_i+\gamma'_k\\
  \boldsymbol\alpha'-\alpha'_i&\boldsymbol\gamma-\gamma'_k
   & \boldsymbol\beta-\alpha'_i-\gamma'_k
\end{smallmatrix};\epsilon\tfrac xy,\tfrac1y\right)\\
%
&\quad\overset{x\to0}\sim
  x^{\alpha'_i}F_{q+r}\left(\begin{smallmatrix}
   \boldsymbol\beta',\boldsymbol\gamma'-\alpha'_i\\
   \boldsymbol\beta,\boldsymbol\gamma+\alpha'_i
\end{smallmatrix};\gamma_k+\alpha'_i;-;\tfrac1y\right)
\qquad(1\le i\le p',\ 1\le k\le r),\allowdisplaybreaks\\
\eqref{eq:G0-4}: \  & G\left(\begin{smallmatrix}
  \alpha&\beta'&\gamma&\\\alpha'&\beta&\gamma'
 \end{smallmatrix};\alpha'_i,\beta_j;+,-\epsilon;x,\epsilon\tfrac1y\right)
\quad\text{at }(0,\infty)
\\
&\quad 
 =x^{\alpha'_i}(-\tfrac1y)^{\beta_j}G\left(\begin{smallmatrix}
  \alpha+\alpha'_i&\beta'+\beta_j&\gamma+\alpha'_i-\beta_j&\\\alpha'-\alpha'_i&\beta-\beta_j
&\gamma'-\alpha'_i+\beta_j
 \end{smallmatrix};x,\epsilon\tfrac1y\right)\\
&\quad\overset{x\to0}\sim
 x^{\alpha'_i}F_{q+r}
\left(\begin{smallmatrix}
 \boldsymbol\beta',\boldsymbol\gamma'-\alpha'_i\\
 \boldsymbol\beta,\boldsymbol\gamma+\alpha'_i
\end{smallmatrix};\beta_j;-;\tfrac1y\right)\qquad(1\le i\le p',\ 1\le j\le q).
\end{align*}}

When we choose one of the the characteristic exponent $\alpha'_i$, the corresponding 
induced equation is $\mathcal M^{\beta',\gamma'-\alpha'_i}_{\beta,\gamma+\alpha'_i}$ 
and the boundary value of the hypergeometric function 
$F\left(\begin{smallmatrix}
  \alpha&\beta&\gamma\\\alpha'&\beta'&\gamma'
\end{smallmatrix};\alpha'_i,\beta'_j;+,-;x,y\right)$ 
is 
$F_{q+r}\left(\begin{smallmatrix}
 \boldsymbol\beta,\boldsymbol\gamma+\alpha'_i\\
 \boldsymbol\beta',\boldsymbol\gamma'-\alpha'_i
\end{smallmatrix};\beta'_j;-;y\right)$.
Note that the boundary value of the function with respect to $\alpha'_{i'}$ with $i\ne i'$
equals 0.
Thus our connection problem is reduced to that of the system $\mathcal M^{\beta',\gamma'-\alpha'_i}_{\beta,\gamma+\alpha'_i}$ explained in \S\ref{sec:HG1} and we have the following theorem from 
the results in \S\ref{sec:HG1}.

\begin{thm}\label{thm:conenct2}
One has the following connection formula
\begin{align*}
&F\left(\begin{smallmatrix}
  \alpha&\beta&\gamma\\\alpha'&\beta'&\gamma'
\end{smallmatrix};\alpha'_i,\beta'_j;+,-;x,y\right)
=
\sum_{k=1}^r a^i_{jk}F\left(\begin{smallmatrix}
  \alpha&\gamma'&\beta'&\\
  \alpha'&\gamma& \beta
\end{smallmatrix};\alpha'_i,\gamma_k;-\epsilon,-1;\epsilon\tfrac xy,\tfrac1y\right)\\
&\hspace{4.7cm}+
\sum_{\ell=1}^qb^i_{j\ell}G\left(\begin{smallmatrix}
  \alpha&\beta'&\gamma&\\\alpha'&\beta&\gamma'
 \end{smallmatrix};\alpha'_i,\beta_\ell;+,-\epsilon;x,\epsilon\tfrac1y\right),
\allowdisplaybreaks\\
&a^i_{jk}=
\text{\small
$
\frac
{ \prod\limits_{\nu\ne j}    \!\Gamma(1+\beta'_j-\beta'_\nu)
  \prod\limits_{\nu=1}^{r'}  \!\Gamma(1+\beta'_j-\gamma'_\nu+\alpha'_i)
  \prod\limits_{\nu\ne k}    \!\Gamma(\gamma_\nu-\gamma_k)
  \prod\limits_{\nu=1}^{q}   \!\Gamma(\beta_\nu-\alpha'_i-\gamma_k)}
{ 
  \prod\limits_{\nu\ne j}    \!\Gamma(1-\gamma_k-\alpha'_i-\beta'_\nu)
  \prod\limits_{\nu=1}^{r'}  \!\Gamma(1-\gamma_k-\gamma'_\nu)
  \prod\limits_{\nu\ne k}    \!\Gamma(\beta'_j+\gamma_\nu+\alpha'_j) 
  \prod\limits_{\nu=1}^{q}   \!\Gamma(\beta'_j+\beta_\nu)
}
$},
\\
&b^i_{j\ell}=
\text{\small
$
  \frac{\prod\limits_{\nu\ne j}  \!\Gamma(1+\beta'_j-\beta'_\nu)
  \prod\limits_{\nu=1}^{r'}     \!\Gamma(1+\beta'_j-\gamma'_\nu+\alpha'_i)
  \prod\limits_{\nu\ne\ell}  \!\Gamma(\beta_\nu-\beta_\ell)
  \prod\limits_{\nu=1}^r     \!\Gamma(\gamma_\nu+\alpha'_i-\beta_\ell)}
 { \prod\limits_{\nu\ne j}  
  \!\Gamma(1-\beta'_\nu-\beta_\ell)
  \prod\limits_{\nu=1}^{r'}     \!\Gamma(1-\gamma'_\nu+\alpha'_i-\beta_\ell)
  \prod\limits_{\nu\ne\ell}  \!\Gamma(\beta'_j+\beta_\nu)
  \prod\limits_{\nu=1}^r     \!\Gamma(\beta'_j+\gamma_\nu+\alpha'_i)}
$}
\intertext{and}
&G\left(\begin{smallmatrix}
 \boldsymbol\alpha&\boldsymbol\gamma& \boldsymbol\beta'\\ 
\boldsymbol\alpha'&\boldsymbol\gamma'&\boldsymbol\beta
\end{smallmatrix};\alpha'_i,\gamma'_k;-,-;\epsilon\tfrac xy,\epsilon y\right)
=
\sum_{\ell=1}^rc^i_{k\ell}
F\left(\begin{smallmatrix}
  \alpha&\gamma'&\beta'&\\
  \alpha'&\gamma& \beta
\end{smallmatrix};\alpha'_i,\gamma_\ell;-\epsilon,-1;\epsilon\tfrac xy,\tfrac1y\right)\\
&\hspace{5cm}{}+
\sum_{j=1}^q d^i_{kj}G\left(\begin{smallmatrix}
  \alpha&\beta'&\gamma&\\\alpha'&\beta&\gamma'
 \end{smallmatrix};\alpha'_i,\beta_j;+,-\epsilon;x,\epsilon\tfrac1y\right),
\allowdisplaybreaks\\
&c^i_{k\ell}=\text{\small$
\frac
{  \prod\limits_{\nu\ne k}   
  \!\Gamma(1+\gamma'_k-\gamma'_\nu)
  \prod\limits_{\nu=1}^{q'} 
  \!\Gamma(1+\gamma'_k-\alpha'_i-\beta'_\nu)     \prod\limits_{\nu\ne\ell} \!\Gamma(\gamma_\nu-\gamma_\ell) 
   \prod\limits_{\nu=1}^q    \!\Gamma(\beta_\nu-\gamma_\ell-\alpha'_i)
}
{   \prod\limits_{\nu\ne k} \!\Gamma(1- 
  \gamma_\ell-\gamma'_\nu) 
  \prod\limits_{\nu=1}^{q'} \!\Gamma(1-\gamma_\ell-\alpha'_i-\beta'_\nu)
 \prod\limits_{\nu\ne \ell} \!\Gamma(\gamma'_k+\gamma_\nu) \prod\limits_{\nu=1}^q \!\Gamma(\gamma'_k-\alpha'_i+\beta_\nu)
 }
$},
\\
&d^i_{kj}=\text{\small$
\frac
{   \prod\limits_{\nu\ne k}         
    \!\Gamma(1+\gamma'_k-\gamma'_\nu)
    \prod\limits_{\nu=1}^{q'} \!\Gamma(1+\gamma'_k-\alpha'_i-\beta'_\nu)
  \prod\limits_{\nu\ne j}   \!\Gamma(\beta_\nu-\beta_j)
  \prod\limits_{\nu=1}^r    \!\Gamma(\gamma_\nu+\alpha'_i-\beta_j)
}
{  \prod\limits_{\nu\ne k}  \!\Gamma(1- 
   \beta_j- \gamma'_\nu+\alpha'_i)
  \prod\limits_{\nu=1}^{q'} \!\Gamma(1-\beta_j-\beta'_\nu)
 \prod\limits_{\nu\ne j}   \!\Gamma(\gamma'_k-\alpha'_i+\beta_\nu)
  \prod\limits_{\nu=1}^r  \!\Gamma(\gamma'_k+\gamma_\nu)
}
$} 
\end{align*}
for $(x,y)\in U$ given by \eqref{eq:nbd}.
\end{thm}

Owing to the coordinate transformations in Proposition \ref{prop:transformations}, we have solutions of other connection problems.
\section{Rank and singular set}
\label{sec:rank-Sing}
\subsection{Rank of the system
$\mathcal M
 ^{\boldsymbol \alpha, \boldsymbol \beta, \boldsymbol \gamma}
 _{\boldsymbol \alpha',\boldsymbol \beta',\boldsymbol \gamma'}$}\label{sec:rank}
In this section, we prove that the rank of the hypergeometric equation $\mathcal M
 ^{\boldsymbol \alpha, \boldsymbol \beta, \boldsymbol \gamma}
 _{\boldsymbol \alpha',\boldsymbol \beta',\boldsymbol \gamma'}$
equals $R\left(\begin{smallmatrix}p&q&r\\ p'&q'&r'\end{smallmatrix}\right)$.
We prepare some notation used in this section.
Let $K$ be the field of rational functions of the variables $x$ and $y$ over $\C$.
We denote by $W(x,y)$ the ring of differential operators of these variables with coefficients in $K$.
For a polynomial $D=\sum_{a,b\geq 0}c_{ab}\xi^a\eta^b\in K[\xi,\eta]$ and a differential operator $P=\sum_{a,b\geq 0}\tilde{c}_{ab}\partial_x^a\partial_y^b\in W(x,y)$, we set
$$
\deg D:=\max\{ a+b\mid c_{ab}\neq 0\}\ \text{  and  }\ \ord P:=\max\{ a+b\mid \tilde{c}_{ab}\neq 0\}.
$$
For a non-negative integer $m$, we put
\begin{align*}
 K[\xi,\eta]_m&:=\{D\in K[\xi,\eta]\mid\deg D\le m\},\\
W(x,y)_m&:=\{P\in W(x,y)\mid \ord P\le m\}. 
\end{align*}
For a differential operator
\[
P=\sum\limits_{m+n\le\ord P}a_{m,n}(x,y)\p_x^m\p_y^n\in W(x,y), 
\]
the principal symbol of $P$ is defined by
\[
 \sigma(P):=\sum\limits_{m+n\le\ord P} a_{m,n}(x,y )\xi^m\eta^n.
\]
We define the rank $\mathcal M^{\boldsymbol\alpha,\boldsymbol\beta,\boldsymbol\gamma}
  _{\boldsymbol\alpha',\boldsymbol\beta',\boldsymbol\gamma'}$ of the system by
\begin{equation}
\rank \mathcal M^{\boldsymbol\alpha,\boldsymbol\beta,\boldsymbol\gamma}
  _{\boldsymbol\alpha',\boldsymbol\beta',\boldsymbol\gamma'} :=
  \dim_K W(x,y)/(W(x,y)P_1+W(x,y)P_2+W(x,y)P_{12}).    
  \label{eq:rank}
\end{equation}
It is well-known that the rank \eqref{eq:rank} is equal to the dimension of the space of holomorphic solutions to the system $\mathcal M^{\boldsymbol\alpha,\boldsymbol\beta,\boldsymbol\gamma}
  _{\boldsymbol\alpha',\boldsymbol\beta',\boldsymbol\gamma'}$ at a generic point (\cite[\S6.2,\S6.3]{hibi2003grobner}).
\begin{thm}\label{thm:rank}
The identity
\begin{equation*}
\begin{split}
 \rank \mathcal M^{\boldsymbol\alpha,\boldsymbol\beta,\boldsymbol\gamma}
  _{\boldsymbol\alpha',\boldsymbol\beta',\boldsymbol\gamma'} &= R\left(\begin{smallmatrix}p&q&r\\ p'&q'&r'\end{smallmatrix}\right)
\end{split}
\end{equation*}
holds for any $\boldsymbol\alpha\in\mathbb C^p,\,\boldsymbol\beta\in\mathbb C^q,\,\boldsymbol\gamma\in\mathbb C^r,\,\boldsymbol\alpha'\in\mathbb C^{p'},\,\boldsymbol\beta'\in\mathbb C^{q'}$ and $
\boldsymbol\gamma'\in\mathbb C^{r'}$.
\end{thm}
\begin{proof}
Since the transformation $(x,y)\mapsto (\tfrac1x,\tfrac1y)$ changes 
$\mathcal M^{\boldsymbol\alpha,\boldsymbol\beta,\boldsymbol\gamma}
_{\boldsymbol\alpha',\boldsymbol\beta',\boldsymbol\gamma'}$
into $\mathcal 
M_{\boldsymbol\alpha,\boldsymbol\beta,\boldsymbol\gamma}^{\boldsymbol\alpha',\boldsymbol\beta',\boldsymbol\gamma'}$,
we may assume $L\ge0$ for the proof.
Note that 
\[
\begin{split}
 \rank \mathcal M^{\boldsymbol\alpha,\boldsymbol\beta,\boldsymbol\gamma}
  _{\boldsymbol\alpha',\boldsymbol\beta',\boldsymbol\gamma'}&\le 
   \dim_K K[\xi,\eta]/(K[\xi,\eta]\sigma(P_1) +K[\xi,\eta]\sigma(P_2) + 
   K[\xi,\eta]\sigma(P_{12}))\\
 &=R\left(\begin{smallmatrix}p&q&r\\ p'&q'&r'\end{smallmatrix}\right).
\end{split}
 \]
The above equality is given by Lemma \ref{lem:dim}.

On the other hand,
in Theorem~\ref{thm:LocalSols} we constructed 
$R\left(\begin{smallmatrix}p&q&r\\p'&q'&r'\end{smallmatrix}\right)$
local solutions at $(0^2,0)$.  They meromorphically depend on the parameters and 
are linearly independent when the parameters are generic. By analytic continuation of suitable linear combinations of solutions with respect to a parameter 
\cite[Proposition~2.21]{OS} assures that the dimension of local solutions at $(0^2,0)$ which equals $\rank \mathcal M^{\boldsymbol\alpha,\boldsymbol\beta,\boldsymbol\gamma}
  _{\boldsymbol\alpha',\boldsymbol\beta',\boldsymbol\gamma'}$
  is not smaller than $R\left(\begin{smallmatrix}p&q&r\\ p'&q'&r'\end{smallmatrix}\right)$.
Hence we have the theorem.
\end{proof}
\begin{cor}\label{cor:spansol}
The $R\left(\begin{smallmatrix}p&q&r\\ p'&q'&r'\end{smallmatrix}\right)$ local solutions given in Theorem~\ref{thm:LocalSols} span the space of local solutions to the system
$M^{\boldsymbol\alpha,\boldsymbol\beta,\boldsymbol\gamma}
_{\boldsymbol\alpha',\boldsymbol\beta',\boldsymbol\gamma'}$
 if the parameters are generic. 
 Moreover
\cite[Lemma~6.3]{Ovm} assures that for any fixed parameters we have precisely $R\left(\begin{smallmatrix}p&q&r\\ p'&q'&r'\end{smallmatrix}\right)$
linearly independent solutions which holomorphically depend on the parameters 
in a neighborhood of the fixed parameters. 
They are suitable linear combinations of the 
solutions given in Theorem~\ref{thm:LocalSols}.
\end{cor}


Suppose $L\ge0$. 
For $\boldsymbol\alpha\in\mathbb C^p$,  
$\boldsymbol\beta\in\mathbb C^q$, 
$\boldsymbol\gamma\in\mathbb C^{r'}$, 
$\boldsymbol\alpha'\in\mathbb C^{L}$ and 
$\boldsymbol\beta'\in\mathbb C^{L}$ 
we define the following elements in $\mathbb C[x,y,\xi,\eta]$:
\begin{align*}
 D_1&=\prod_{\nu=1}^p(x\xi-\alpha_\nu),\ D_2=\prod_{\nu=1}^q(y\eta-\beta_\nu),\ 
 D_3=\prod_{\nu=1}^{r'}(x\xi+y\eta-\gamma_\nu),
  \\
 D_4&=\prod_{\nu=1}^L(x\xi-\alpha'_\nu)-x(x\xi+y\eta)^L,\ 
 D_5=\prod_{\nu=1}^L(y\eta-\beta'_\nu)-y(x\xi+y\eta)^L,\\ 
 D_6&=xD_5-yD_4,\ 
 Q_1=D_1D_3D_4,\ Q_2=D_2D_3D_5,\ Q_3=D_1D_2D_6.
\end{align*}
In the case when parameters $\alpha_\nu,\,\beta_\nu,\ldots$ are 0, we put 
\begin{align*}
 \bar D_i&=D_i|_{\boldsymbol\alpha=\boldsymbol\beta=\boldsymbol\gamma=\boldsymbol\alpha'=\boldsymbol\beta'=0},
\ \bar Q_j=Q_j|_{\boldsymbol\alpha=\boldsymbol\beta=\boldsymbol\gamma=\boldsymbol\alpha'=\boldsymbol\beta'=0}\quad(1\le i\le 6,\ 1\le j\le 3).
\end{align*}
Then we note that $\sigma(P_1)=\bar Q_1,\ 
  \sigma(P_2)=\bar Q_2$ and 
  $\sigma(P_{12})=\bar Q_3$.
\begin{lem}\label{lem:res}
Suppose there exist $C_i\in K[\xi,\eta]$ satisfying
\begin{equation}\label{eq:rel}
 C_1 D_1D_3D_4 + C_2 D_2D_3D_5 + C_3 D_1D_2D_6=0.
\end{equation}
Then there exist $A,\,B\in K[\xi,\eta]$ satisfying
\begin{equation}\label{eq:solres}
  \begin{cases}
   C_1=(AD_5+yB)D_2,\\
   C_2=-(AD_4+xB)D_1,\\
   C_3=BD_3.
 \end{cases}
\end{equation}
Here we note that
\eqref{eq:rel} is valid if $C_j$ are given by \eqref{eq:solres}.
\end{lem}
\begin{proof}
Note that $K[\xi,\eta]$ is a unique factorization domain and $D_i$ and $D_j$ are mutually 
prime for $1\le i<j\le 6$.
Hence $C_1\in K[\xi,\eta]D_2$, $C_2\in K[\xi,\eta]D_1$ and $C_3\in K[\xi,\eta]D_3$
and therefore $C_1=B_1D_2$, $C_2=B_2D_1$ and $C_3=B_3D_3$ with $B_j\in K[\xi,\eta]$. Then
$B_1D_4+B_2D_5+B_3D_6=0$ and therefore
\[
  (B_1-yB_3)D_4+(B_2+xB_3)D_5=0.
\]
Hence $B_1-yB_3=AD_5$ and $B_2+xB_3=-AD_4$ with $A\in K[\xi,\eta]$.  Putting $B=B_3$, we have 
\eqref{eq:solres}.
\end{proof}
\begin{lem}\label{lem:dim}
Put $I=K[\xi,\eta]Q_1+K[\xi,\eta]Q_2+K[\xi,\eta]Q_3$ and
$\bar I=K[\xi,\eta]\bar Q_1+K[\xi,\eta]\bar Q_2)+K[\xi,\eta]\bar Q_3$.  Then 
$\dim_KK[\xi,\eta]/I$ does not depend on the parameters $\boldsymbol\alpha$, $\boldsymbol\beta$,
$\boldsymbol\gamma$, $\boldsymbol\alpha'$ and $\boldsymbol\beta'$ and 
\begin{equation}\label{eq:Cdim}
\dim_KK[\xi,\eta]/I= \dim_KK[\xi,\eta]/\bar I
=R\left(\begin{smallmatrix}p&q&r\\ p'&q'&r'\end{smallmatrix}\right).
\end{equation}
\end{lem}
\begin{proof}
Let $D\in I$.
Then
\begin{align}\label{eq:DQ}
 D=A_1Q_1+A_2Q_2+A_3Q_3
\end{align}
with suitable $A_j\in K[\xi,\eta]$.
Put $N=\max_{1\le j\le 3}\deg A_jQ_j$.

Suppose $N>\deg D$.  
Let $\bar C_j$ are homogeneous polynomials with $A_j-\bar C_j\in K[\xi,\eta]_{\deg A_j-1}$ if  
$N=\deg A_jQ_j$. Put $\bar C_j=0$ if $N>\deg A_jQ_j$. 
Then $\bar C_1\bar Q_1+\bar C_2\bar Q_2+\bar C_3\bar Q_3=0$ and we have $\bar A,\,\bar B\in K[\xi,\eta]$
satisfying 
\begin{equation*}
  \begin{cases}
   \bar C_1=(A\bar D_5+yB)\bar D_2,\\
   \bar C_2=-(A\bar D_4+xB)\bar D_1,\\
   \bar C_3=B\bar D_3.
 \end{cases}
\end{equation*}
Replacing $A_j$ by $A_j-C_j$ for $1\le j\le 3$ with \eqref{eq:solres}, we have \eqref{eq:DQ} with smaller $N$.
Thus, we may assume $N=\ord D$.  Then 
$D+K[\xi,\eta]_{\deg D-1}\subset \bar I + K[\xi,\eta]_{\deg D-1}$.
Here $\bar I=K[\xi,\eta]\bar Q_1+K[\xi,\eta]\bar Q_2)+K[\xi,\eta]\bar Q_3$. 
Since \[(L\cap K[\xi,\eta]_m)/(I\cap K[\xi,\eta]_{m-1})\simeq 
(\bar I\cap K[\xi,\eta]_m)/(\bar I\cap K[\xi,\eta]_{m-1})\]  
and 
$\dim K[\xi,\eta]/I=\sum\limits_{m=0}^\infty \dim (I\cap K[\xi,\eta]_m)/(I\cap K[\xi,\eta]_{m-1})$,
$\dim K[\xi,\eta]/I$ does not depend on the parameters. 

Note that $\dim  K[\xi,\eta]/I$ is $\dim C[\xi,\eta]/\sum_{i=1}^3\mathbb C[\xi,\eta]Q_i|_{(x,y)=(x_0,y_0)}$ for generic $(x_0,y_0)\in\mathbb C^2$.
Hence we assume $(x,y)$ is a generic point of $\mathbb C^2$.

Put $V:=\{(\xi,\eta)\in\mathbb C^2/\mathbb C^\times \mid Q_1=Q_2=Q_3=0\}$ and 
$V_{ij}:=\{(\xi,\eta)\in\mathbb C^2/\mathbb C^\times \mid D_i=D_j=0\}$. Then
\[V=V_{12}\cup V_{13}\cup V_{23}\cup V_{15}\cup V_{24}\cup V_{36} \cup V_{45}.\]
When $\boldsymbol \alpha,\,\boldsymbol \beta,\,\boldsymbol\gamma,\,\boldsymbol\alpha',\,
\boldsymbol\beta'$ and $(x,y)$ are generic, the points in $V$ defined by $Q_1=Q_2=Q_3=0$ is multiplicity free and the numbers of points of 
$V_{12},\,V_{13},\,V_{23},\,V_{15}$, $V_{24},\,V_{36},\,V_{45}$ are 
$pq,\,pr',\,qr',\,pL,\,qL,\,r'L,\,L^2$, respectively.
Since $pq+pr'+qr'+pL+qL+r'L+L^2=
R\left(\begin{smallmatrix}p&q&r\\p'&q'&r'\end{smallmatrix}\right)$, 
we have Lemma~\ref{lem:dim}.
\end{proof}

The proof of Theorem~\ref{thm:rank} shows the following.
\begin{cor}
The symbol ideal of the equation
$\mathcal M^{\boldsymbol \alpha, \boldsymbol \beta, \boldsymbol \gamma}
 _{\boldsymbol \alpha',\boldsymbol \beta',\boldsymbol \gamma'}$
in $W(x,y)$ is generated by
$\sigma(P_1)$, $\sigma(P_2)$ and $\sigma(P_{12})$, namely,
\begin{align*}\begin{split}
  &\{\sigma(P)\mid P\in W(x,y)P_1+W(x,y)P_2+W(x,y)P_{12}\}\\
  &\quad{}=K[\xi,\eta]\sigma(P_1)+K[\xi,\eta]\sigma(P_2)+K[\xi,\eta]\sigma(P_{12}).
\end{split}\end{align*}
\end{cor}

\subsection{Singular set of the system
$\mathcal M
 ^{\boldsymbol \alpha, \boldsymbol \beta, \boldsymbol \gamma}
 _{\boldsymbol \alpha',\boldsymbol \beta',\boldsymbol \gamma'}$}\label{sec:Sing}
In this section we examine the singularities of the solutions to
$\mathcal M
 ^{\boldsymbol \alpha, \boldsymbol \beta, \boldsymbol \gamma}
 _{\boldsymbol \alpha',\boldsymbol \beta',\boldsymbol \gamma'}$.

\begin{defn}
We define the singular set of 
$\mathcal M^{\boldsymbol \alpha, \boldsymbol \beta, \boldsymbol \gamma}
 _{\boldsymbol \alpha',\boldsymbol \beta',\boldsymbol \gamma'}$
by 
\begin{align*}
 \mathrm{Sing}(\mathcal M
 ^{\boldsymbol \alpha, \boldsymbol \beta, \boldsymbol \gamma}
 _{\boldsymbol \alpha',\boldsymbol \beta',\boldsymbol \gamma'})
 &:= \{(x,y)\in\mathbb C^2\mid \exists (\xi,\eta)\in\mathbb C^2\setminus\{(0,0)\}
 \text{ such that }\\
 &\quad \sigma(P_1)(x,y,\xi,\eta)=\sigma(P_2)(x,y,\xi,\eta)=\sigma(P_{12})(x,y,\xi,\eta)=0\}.
\end{align*}
\end{defn}

We note that any local solution to 
$\mathcal M^{\boldsymbol \alpha, \boldsymbol \beta, \boldsymbol \gamma}
 _{\boldsymbol \alpha',\boldsymbol \beta',\boldsymbol \gamma'}$
has an analytic continuation along any path in $\mathbb C^2\setminus
 \mathrm{Sing}(\mathcal M
 ^{\boldsymbol \alpha, \boldsymbol \beta, \boldsymbol \gamma}
 _{\boldsymbol \alpha',\boldsymbol \beta',\boldsymbol \gamma'})$.
 It is more natural to take the closure of the singular set in a compactification $\mathbb P^1\times \mathbb P^1$ of $\C^2$.
The divisors $\{x=\infty\}$ and $\{ y=\infty\}$ are regarded as a subset of the singular set.

\begin{thm}\label{thm:Sing}
The singular set $\mathrm{Sing}(\mathcal M^{\boldsymbol\alpha,\boldsymbol\beta,\boldsymbol\gamma}
_{\boldsymbol\alpha',\boldsymbol\beta',\boldsymbol\gamma'})$ is given by the union of the following irreducible subvarieties:
\begin{itemize}
    \item[\rm (1)]$L\ne0$.
    Then, the subvarieties are as follows:
    \begin{align*}
        &V_L:=\bigl\{\bigl(\tfrac{1}{(1+t)^L},\tfrac{t^L}{(1+t)^L}\bigr)
   \in\C\times\C\mid t\in (\C\setminus\{-1\})\cup\{\infty\}
   \bigr\},\ \{x=0\},\\ &\quad\{ y=0\},\ 
        \{ x=1\}\ (qq'> 0),\ \ \{ y=1\}\ (pp'> 0),\ \ \{ x=\epsilon y\}\ (r'>0).
    \end{align*} 
    \item[\rm (2)]$L=0$. Then, the subvarieties are all linear:
    \begin{align*}
        &\{x=0\},\ \{ y=0\},\ \{ x=1\},\ \{ y=1\},\ \{ x=y\}.
    \end{align*} 
\end{itemize}
\end{thm}
\begin{proof}
We may assume $L=r-r'\geq0$ as in the proof of the last theorem.
A point $(x,y)\in\C^2$ is in the singular set $\mathrm{Sing}(\mathcal M
 ^{\boldsymbol \alpha, \boldsymbol \beta, \boldsymbol \gamma}
 _{\boldsymbol \alpha',\boldsymbol \beta',\boldsymbol \gamma'})$ if there exists $(\xi,\eta)\in\mathbb C^2$ such that
\begin{equation}\begin{split}
\bigl((x\xi)^L-x(x\xi+y\eta)^L\bigr)(x\xi)^p(x\xi+y\eta)^{r'}&=0,\\
\bigl((y\eta)^L-y(x\xi+y\eta)^L\bigr)(y\eta)^q(x\xi+y\eta)^{r'}&=0,\\
\bigl(y(x\xi)^L-x(y\eta)^L\bigr)(x\xi)^p(y\eta)^q&=0.
\end{split}\label{eq:Vk}\end{equation}

Suppose $L>0$. 
Then, it is clear that $\{x=0\}\cup\{y=0\}\subset \mathrm{Sing}(\mathcal M
 ^{\boldsymbol \alpha, \boldsymbol \beta, \boldsymbol \gamma}
 _{\boldsymbol \alpha',\boldsymbol \beta',\boldsymbol \gamma'})$.
Let us take an element $(x,y,\xi,\eta)\in\C^2\times\C^2\setminus\{ (0,0)\}$ such that the equations
\begin{align}
\begin{split}
   \bigl(\xi^L-x(\xi+\eta)^L\bigr)\xi^p(\xi+\eta)^{r'}
   &=\bigl(\eta^L-y(\xi+\eta)^L\bigr)\eta^q(\xi+\eta)^{r'}\\
   &=(y\xi^L-x\eta^L)\xi^p\eta^q=0
\end{split}
\label{eq:Vkeq}
\end{align}
hold.
Putting $(\xi,\eta)=(1,0)$ in \eqref{eq:Vkeq}, we have
\[
 1-x=y0^q=y0^q=0
\]
and therefore we have  $x=1$ if $y\ne0$ and $q>0$.

Suppose $\xi=-\eta\ne0$ satisfies \eqref{eq:Vkeq}. 
We have $r'>0$ and $y\xi^L-x\eta^L = 0$, namely, $x-(-1)^Ly=0$. 

Suppose $\xi\eta(\xi+\eta)\ne0$ and $xy\ne0$ satisfies \eqref{eq:Vkeq}.
Then \eqref{eq:Vkeq} means
\begin{equation}
 \xi^L-x(\xi+\eta)^L=\eta^L-y(\xi+\eta)^L=y\xi^L-x\eta^L=0.
\label{eq:Vksub} 
\end{equation}
We may assume $(\xi,\tau)=(1,t)$ with $t\in\bar{\mathbb C}$
and then we have the theorem.

Suppose $L=0$ and $p$, $q$ and $r'$ are positive, then \eqref{eq:Vkeq} is reduced to
\begin{align*}
x(1-x)\xi(x\xi+y\eta)&=0,\\
y(1-y)\eta(x\xi+y\eta)&=0,\\
xy(x-y)\xi\eta&=0.
\end{align*}
Hence, it is clear that
\begin{align*}
 \mathrm{Sing}(\mathcal M
 ^{\boldsymbol \alpha, \boldsymbol \beta, \boldsymbol \gamma}
 _{\boldsymbol \alpha',\boldsymbol \beta',\boldsymbol \gamma'})=\{x=0\}\cup\{x=1\}\cup\{y=0\}\cup\{y=1\}\cup\{x=y\}.
\end{align*}

\end{proof}
\begin{rem}
If the parameters take special values, some singularities may vanish. 
For example the system satisfied by 
$(1-x-y)^{-\alpha}=\sum\limits_{m\ge0,\,n\ge0}\frac{(\alpha)_{m+n}x^my^m}{m!n!}=F^{0,0,1}_{1,1,0}\left(
\begin{smallmatrix}
    \emptyset&\emptyset&\alpha\\
    0&0&\emptyset\
\end{smallmatrix}
;x,y\right)$ 
does not have singularities $\{x=0\}\cup\{y=0\}$.  This example follows from the equality $\bigl(y(x\xi)^1-x(y\eta)^1\bigr)=xy(\xi-\eta)$ in \eqref{eq:Vkeq} and does not happen when $|r-r'|>1$
(cf.~Remark~\ref{rem:red0}).
\end{rem}
\begin{rem} 
Below, we list some properties of the defining polynomial of $V_L$.

{\rm i)}\ Let $f_L(x,y)$ be the resultant of the polynomials $x(1+t)^L-1$ and
$y(1+t)^L-t^L$ of $t$ for $L>0$ and put $f_0(x,y)=1$ and 
$f_{-L}(x,y)=f_L(\tfrac1x,\tfrac1y)$. Then
\begin{align*}
 V_L&=\{(x,y)\in\mathbb C^2\mid f_L(x,y)=0\},\\
 f_L(x,y)&=f_L(y,x)\quad(\Leftarrow t\mapsto \tfrac1t),\\
 f_L(x,0)=0&\Rightarrow x=1,\\
 f_L(\tfrac 1{2^L},\tfrac 1{2^L})&=0\quad(L\ne0),\\
 f_L(x,x)=0&\Leftrightarrow 
 \frac1x=\sum_{\nu=1}^k\binom k\nu \cos\frac{2m\nu}L\pi\quad(m=1,\dots,L)\ \ \text{if \ }L>0.
\end{align*}
Here $(x,y)=\bigl(\tfrac1{(1+t)^L},\frac{t^L}{(1+t)^L}\bigr)$ 
and the last line in the above follows from 
$t=\cos\frac{2m}L\pi+\sqrt{-1}\sin\frac{2m}L\pi$. 

\medskip
{\rm ii)}\ Direct calculation shows
\begin{align*}
f_1(x,y)&=1-x+y,\\
f_2(x,y)&=(1-x-y)^2-4xy,\\
f_3(x,y)&=(1-x-y)^3-27xy,\\
f_4(x,y)&=(1-x-y)^4-8xy\bigl(x^2+y^2+14(x+y)+17\bigr),\\
f_5(x,y)&=(1-x-y)^5-625xy\bigl(x^2+y^2-3(xy-x-y)+1\bigr),\\
f_6(x,y)&=(1-x-y)^6-xy\cdot g_6(x,y),\\
g_6(x,y)&=12(x^4+y^4)+40x^2y^2+2706(x^3+y^3)+20526xy(x+y)-131637xy\\
		&\quad{}+20526(x+y)+2766.
\end{align*}

{\rm iii)}\ We give a different expression of $f_L(x,y)$ in Theorem~\ref{thm:Singn}.

\medskip
{\rm iv)}\ If $L>0$, $f_{L+1}(x,y)$ is a factor of the discriminant of $f_L(\tfrac xt,\tfrac y{1-t})=0$ 
as a polynomial of $t$, which follows
from the integral transformation $K^{\mu-\alpha-\beta,(\alpha,\beta)}_{x,y}$ in \S\ref{sec:IntegRep},
\end{rem}

For example,  
$f_1(\frac xt,\frac y{1-t})=\frac x{1-t}+\frac y{1-t}=1$ means $t^2-(x-y+1)t+x=0$ and 
the discriminant equals $(x-y+1)^2-4x=(x+y-1)^2-4xy=f_2(x,y)$.
The equation $f_2(\frac xt,\frac y{1-t})
=(\tfrac xt+\tfrac y{1-t}-1)^2-4\tfrac xt\tfrac y{1-t}=0$ means 
\[
 2(x-y+1)t^3-(x^2+2xy+y^2+4x-2y+1)t^2+2x(x+y+1)t-x^2
\]
and the discriminant of this polynomial of $t$ equals $256x^3y^3\bigl((x+y-1)^3+27xy\bigr)$.
\medskip

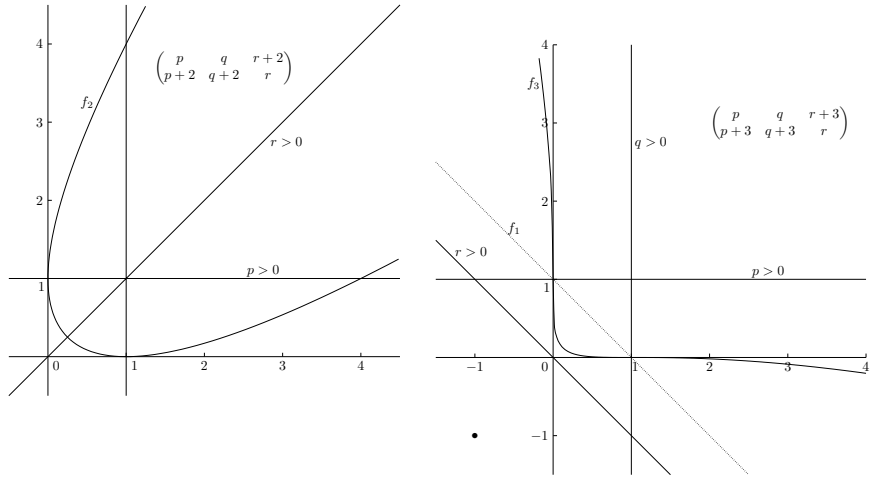
\scalebox{0.52}{\begin{tikzpicture}
\draw(0,-1) -- (0,9);
\draw(8,0) -- (8,0.1);
\node[below] at(8,0){$4$};
\draw(6,0) -- (6,0.1);
\node[below] at(6,0){$3$};
\draw(4,0) -- (4,0.1);
\node[below] at(4,0){$2$};
\draw(2,0) -- (2,0.1);
\node[below right] at(2,0){$1$};
\node[below right] at(0,0){$0$};
\draw(-1,0) -- (9,0);
\draw(0,8) -- (0.1,8);
\node[left] at(0,8){$4$};
\draw(0,6) -- (0.1,6);
\node[left] at(0,6){$3$};
\draw(0,4) -- (0.1,4);
\node[left] at(0,4){$2$};
\draw(0,2) -- (0.1,2);
\node[left] at(0.05,1.8){$1$};
\draw
(-1,-1)--(9,9) (-1,2)--(9,2) (2,-1)--(2,9)
;
\draw
(2.496,8.965) .. controls (2.342,8.673) .. (2.192,8.379) .. controls (2.094,8.188)
 and (1.999,7.999) .. (1.907,7.814) .. controls (1.816,7.628) and (1.728,7.446)
 .. (1.643,7.268) .. controls (1.558,7.089) and (1.476,6.913) .. (1.398,6.741)
 .. controls (1.319,6.569) and (1.244,6.4) .. (1.172,6.235) .. controls (1.101,6.069)
 and (1.032,5.907) .. (0.967,5.748) .. controls (0.902,5.589) and (0.84,5.434)
 .. (0.781,5.281) .. controls (0.723,5.129) and (0.667,4.98) .. (0.615,4.834)
 .. controls (0.563,4.688) and (0.515,4.546) .. (0.469,4.407) .. controls (0.424,4.268)
 and (0.382,4.132) .. (0.343,3.999) .. controls (0.304,3.867) and (0.269,3.737)
 .. (0.236,3.611) .. controls (0.204,3.485) and (0.175,3.363) .. (0.15,3.243)
 .. controls (0.124,3.124) and (0.102,3.008) .. (0.083,2.895) .. controls (0.063,2.782)
 and (0.048,2.673) .. (0.035,2.567) .. controls (0.023,2.46) and (0.014,2.357)
 .. (0.008,2.258) .. controls (0.002,2.158) and (-0.001,2.062) .. (0,1.969)
 .. controls (0.001,1.876) and (0.005,1.786) .. (0.012,1.7) .. controls (0.02,1.613)
 and (0.03,1.53) .. (0.044,1.45) .. controls (0.058,1.371) and (0.075,1.294)
 .. (0.096,1.221) .. controls (0.116,1.148) and (0.14,1.077) .. (0.167,1.011)
 .. controls (0.194,0.944) and (0.225,0.881) .. (0.258,0.821) .. controls (0.292,0.761)
 and (0.329,0.704) .. (0.369,0.651) .. controls (0.409,0.597) and (0.453,0.547)
 .. (0.5,0.5) .. controls (0.547,0.453) and (0.597,0.409) .. (0.651,0.369)
 .. controls (0.704,0.329) and (0.761,0.292) .. (0.821,0.258) .. controls (0.881,0.225)
 and (0.944,0.194) .. (1.011,0.167) .. controls (1.077,0.14) and (1.148,0.116)
 .. (1.221,0.096) .. controls (1.294,0.075) and (1.371,0.058) .. (1.45,0.044)
 .. controls (1.53,0.03) and (1.613,0.02) .. (1.7,0.012) .. controls (1.786,0.005)
 and (1.876,0.001) .. (1.969,0) .. controls (2.062,-0.001) and (2.158,0.002)
 .. (2.258,0.008) .. controls (2.357,0.014) and (2.46,0.023) .. (2.567,0.035)
 .. controls (2.673,0.048) and (2.782,0.063) .. (2.895,0.083) .. controls (3.008,0.102)
 and (3.124,0.124) .. (3.243,0.15) .. controls (3.363,0.175) and (3.485,0.204)
 .. (3.611,0.236) .. controls (3.737,0.269) and (3.867,0.304) .. (3.999,0.343)
 .. controls (4.132,0.382) and (4.268,0.424) .. (4.407,0.469) .. controls (4.546,0.515)
 and (4.688,0.563) .. (4.834,0.615) .. controls (4.98,0.667) and (5.129,0.723)
 .. (5.281,0.781) .. controls (5.434,0.84) and (5.589,0.902) .. (5.748,0.967)
 .. controls (5.907,1.032) and (6.069,1.101) .. (6.235,1.172) .. controls (6.4,1.244)
 and (6.569,1.319) .. (6.741,1.398) .. controls (6.913,1.476) and (7.089,1.558)
 .. (7.268,1.643) .. controls (7.446,1.728) and (7.628,1.816) .. (7.814,1.907)
 .. controls (7.999,1.999) and (8.188,2.094) .. (8.379,2.192) .. controls (8.673,2.342)
 .. (8.965,2.496);
\node at (4.5,7.4){$\begin{pmatrix}p&q&r+2\\ p+2&q+2&r\end{pmatrix}$};
\node at (1,6.5) {$
f_2$};
\node at (5.5,2.2) {$
p>0$};
\node at (2.5,5.5) {
};
\node at (6.1,5.5) {$
r>0$};
\end{tikzpicture}}\quad\ 
\raisebox{-1.05cm}{\scalebox{0.52}{\begin{tikzpicture}
\draw(0,-3) -- (0,8);
\draw 
(2,-3)--(2,8);
\draw(8,0) -- (8,0.1);
\node[below] at(8,0){$4$};
\draw(6,0) -- (6,0.1);
\node[below] at(6,0){$3$};
\draw(4,0) -- (4,0.1);
\node[below] at(4,0){$2$};
\draw(2,0) -- (2,0.1);
\node[below] at(2.1,0){$1$};
\node[below left] at(0,0){$0$};
\draw(-2,0) -- (-2,0.1);
\node[below] at(-2,0){$-1$};
\draw (-3,0) -- (8,0);
\draw
(-3,2)--(8,2);
\draw(0,8) -- (0.1,8);
\node[left] at(0,8){$4$};
\draw(0,6) -- (0.1,6);
\node[left] at(0,6){$3$};
\draw(0,4) -- (0.1,4);
\node[left] at(0,4){$2$};
\draw(0,2) -- (0.1,2);
\node[left] at(0.05,1.8){$1$};
\draw(0,-2) -- (0.1,-2);
\node[left] at(0,-2){$- 1$};
\node at (-2,-2) {
$\bullet$};
\draw
(-3,3)--(3,-3);
\draw[densely dotted
] (-3,5)--(5,-3);
\draw
(-0.359,7.656) .. controls (-0.318,7.34) .. (-0.281,7.024) .. controls (-0.256,6.804)
 and (-0.229,6.566) .. (-0.203,6.309) .. controls (-0.178,6.061) and (-0.15,5.775)
 .. (-0.125,5.451) .. controls (-0.106,5.205) and (-0.066,4.807) .. (-0.047,4.255)
 .. controls (-0.005,3.006) and (-0.011,1.868) .. (0.031,0.844) .. controls (0.04,0.625)
 and (0.1,0.508) .. (0.109,0.478) .. controls (0.131,0.406) and (0.166,0.359)
 .. (0.188,0.325) .. controls (0.212,0.287) and (0.241,0.259) .. (0.266,0.235)
 .. controls (0.291,0.211) and (0.318,0.192) .. (0.344,0.175) .. controls (0.37,0.158)
 and (0.396,0.145) .. (0.422,0.133) .. controls (0.448,0.12) and (0.474,0.11)
 .. (0.5,0.101) .. controls (0.526,0.092) and (0.552,0.085) .. (0.578,0.078)
 .. controls (0.604,0.071) and (0.63,0.065) .. (0.656,0.06) .. controls (0.682,0.054)
 and (0.708,0.05) .. (0.734,0.046) .. controls (0.76,0.042) and (0.786,0.038)
 .. (0.813,0.035) .. controls (0.839,0.032) and (0.865,0.029) .. (0.891,0.026)
 .. controls (0.917,0.024) and (0.943,0.022) .. (0.969,0.02) .. controls (0.995,0.018)
 and (1.021,0.016) .. (1.047,0.015) .. controls (1.073,0.013) and (1.099,0.012)
 .. (1.125,0.011) .. controls (1.151,0.009) and (1.177,0.008) .. (1.203,0.008)
 .. controls (1.229,0.007) and (1.255,0.006) .. (1.281,0.005) .. controls (1.307,0.005)
 and (1.333,0.004) .. (1.359,0.004) .. controls (1.385,0.003) and (1.411,0.003)
 .. (1.438,0.002) .. controls (1.464,0.002) and (1.49,0.002) .. (1.516,0.001)
 .. controls (1.542,0.001) and (1.568,0.001) .. (1.594,0.001) .. controls (1.62,0.001)
 and (1.646,0) .. (1.672,0) .. controls (1.698,0) and (1.724,0)
 .. (1.75,0) .. controls (1.776,0) and (1.802,0) .. (1.828,0)
 .. controls (1.854,0) and (1.88,0) .. (1.906,0) .. controls (1.932,0)
 and (1.958,0) .. (1.984,0) .. controls (2.01,0) and (2.036,0)
 .. (2.063,0) .. controls (2.089,0) and (2.115,0) .. (2.141,0)
 .. controls (2.167,0) and (2.193,0) .. (2.219,0) .. controls (2.245,0)
 and (2.271,0) .. (2.297,0) .. controls (2.323,0) and (2.349,0)
 .. (2.375,0) .. controls (2.401,0) and (2.427,-0.001) .. (2.453,-0.001)
 .. controls (2.479,-0.001) and (2.505,-0.001) .. (2.531,-0.001) .. controls (2.557,-0.001)
 and (2.583,-0.001) .. (2.609,-0.002) .. controls (2.635,-0.002) and (2.661,-0.002)
 .. (2.688,-0.002) .. controls (2.714,-0.002) and (2.74,-0.003) .. (2.766,-0.003)
 .. controls (2.792,-0.003) and (2.818,-0.004) .. (2.844,-0.004) .. controls (2.87,-0.004)
 and (2.896,-0.005) .. (2.922,-0.005) .. controls (2.948,-0.005) and (2.974,-0.006)
 .. (3,-0.006) .. controls (3.026,-0.006) and (3.052,-0.007) .. (3.078,-0.007)
 .. controls (3.104,-0.008) and (3.13,-0.008) .. (3.156,-0.009) .. controls (3.182,-0.009)
 and (3.208,-0.01) .. (3.234,-0.01) .. controls (3.26,-0.011) and (3.286,-0.012)
 .. (3.313,-0.012) .. controls (3.339,-0.013) and (3.365,-0.014) .. (3.391,-0.014)
 .. controls (3.417,-0.015) and (3.443,-0.016) .. (3.469,-0.016) .. controls (3.495,-0.017)
 and (3.521,-0.018) .. (3.547,-0.019) .. controls (3.573,-0.019) and (3.599,-0.02)
 .. (3.625,-0.021) .. controls (3.651,-0.022) and (3.677,-0.023) .. (3.703,-0.024)
 .. controls (3.729,-0.025) and (3.755,-0.026) .. (3.781,-0.026) .. controls (3.807,-0.027)
 and (3.833,-0.028) .. (3.859,-0.029) .. controls (3.885,-0.03) and (3.911,-0.031)
 .. (3.938,-0.033) .. controls (3.964,-0.034) and (3.99,-0.035) .. (4.016,-0.036)
 .. controls (4.042,-0.037) and (4.068,-0.038) .. (4.094,-0.039) .. controls (4.12,-0.04)
 and (4.146,-0.042) .. (4.172,-0.043) .. controls (4.198,-0.044) and (4.224,-0.045)
 .. (4.25,-0.047) .. controls (4.276,-0.048) and (4.302,-0.049) .. (4.328,-0.051)
 .. controls (4.354,-0.052) and (4.38,-0.053) .. (4.406,-0.055) .. controls (4.432,-0.056)
 and (4.458,-0.057) .. (4.484,-0.059) .. controls (4.51,-0.06) and (4.536,-0.062)
 .. (4.563,-0.063) .. controls (4.589,-0.065) and (4.615,-0.066) .. (4.641,-0.068)
 .. controls (4.667,-0.07) and (4.693,-0.071) .. (4.719,-0.073) .. controls (4.745,-0.074)
 and (4.771,-0.076) .. (4.797,-0.078) .. controls (4.823,-0.079) and (4.849,-0.081)
 .. (4.875,-0.083) .. controls (4.901,-0.084) and (4.927,-0.086) .. (4.953,-0.088)
 .. controls (4.979,-0.09) and (5.005,-0.092) .. (5.031,-0.093) .. controls (5.057,-0.095)
 and (5.083,-0.097) .. (5.109,-0.099) .. controls (5.135,-0.101) and (5.161,-0.103)
 .. (5.188,-0.105) .. controls (5.214,-0.107) and (5.24,-0.108) .. (5.266,-0.11)
 .. controls (5.292,-0.112) and (5.318,-0.114) .. (5.344,-0.116) .. controls (5.37,-0.119)
 and (5.396,-0.121) .. (5.422,-0.123) .. controls (5.448,-0.125) and (5.474,-0.127)
 .. (5.5,-0.129) .. controls (5.526,-0.131) and (5.552,-0.133) .. (5.578,-0.135)
 .. controls (5.604,-0.138) and (5.63,-0.14) .. (5.656,-0.142) .. controls (5.682,-0.144)
 and (5.708,-0.147) .. (5.734,-0.149) .. controls (5.76,-0.151) and (5.786,-0.153)
 .. (5.813,-0.156) .. controls (5.839,-0.158) and (5.865,-0.16) .. (5.891,-0.163)
 .. controls (5.917,-0.165) and (5.943,-0.168) .. (5.969,-0.17) .. controls (5.995,-0.172)
 and (6.021,-0.175) .. (6.047,-0.177) .. controls (6.073,-0.18 ) and (6.099,-0.182)
 .. (6.125,-0.185) .. controls (6.151,-0.187) and (6.177,-0.19) .. (6.203,-0.193)
 .. controls (6.229,-0.195) and (6.255,-0.198) .. (6.281,-0.2) .. controls (6.307,-0.203)
 and (6.333,-0.206) .. (6.359,-0.208) .. controls (6.385,-0.211) and (6.411,-0.214)
 .. (6.438,-0.216) .. controls (6.464,-0.219) and (6.49,-0.222) .. (6.516,-0.225)
 .. controls (6.542,-0.227) and (6.568,-0.23) .. (6.594,-0.233) .. controls (6.62,-0.236)
 and (6.646,-0.239) .. (6.672,-0.241) .. controls (6.698,-0.244) and (6.724,-0.247)
 .. (6.75,-0.25) .. controls (6.776,-0.253) and (6.802,-0.256) .. (6.828,-0.259)
 .. controls (6.854,-0.262) and (6.88,-0.265) .. (6.906,-0.268) .. controls (6.932,-0.271)
 and (6.958,-0.274) .. (6.984,-0.277) .. controls (7.01,-0.28) and (7.036,-0.283)
 .. (7.063,-0.286) .. controls (7.089,-0.289) and (7.115,-0.292) .. (7.141,-0.295)
 .. controls (7.167,-0.298) and (7.193,-0.301) .. (7.219,-0.304) .. controls (7.245,-0.308)
 and (7.271,-0.311) .. (7.297,-0.314) .. controls (7.323,-0.317) and (7.349,-0.32)
 .. (7.375,-0.324) .. controls (7.401,-0.327) and (7.427,-0.33) .. (7.453,-0.333)
 .. controls (7.479,-0.337) and (7.505,-0.34) .. (7.531,-0.343) .. controls (7.557,-0.347)
 and (7.583,-0.35) .. (7.609,-0.353) .. controls (7.635,-0.357) and (7.661,-0.36)
 .. (7.688,-0.364) .. controls (7.714,-0.367) and (7.74,-0.37) .. (7.766,-0.374)
 .. controls (7.792,-0.377) and (7.818,-0.381) .. (7.844,-0.384) .. controls (7.87,-0.388)
 and (7.896,-0.391) .. (7.922,-0.395) .. controls (7.961,-0.4) .. (8,-0.405);
\node at (5.8,6){$\begin{pmatrix}p&q&r+3\\ p+3&q+3&r\end{pmatrix}$};
\node at (-0.5,7) {$
f_3$};
\node at (-1,3.3) {$
f_1$};
\node at (5.5,2.2) {$
p>0$};
\node at (2.5,5.5) {$
q>0$};
\node at (-2.1,2.7) {$
r>0$};
\end{tikzpicture}}}


\section{KZ-type equation}
Katz (\cite{katz1996rigid}) defines an operation called middle convolution on local systems.  
Dettweiler and Reiter (\cite{DR}) formulate it as an operation on the Fuchsian systems 
\eqref{eq:ODE} and then Haraoka (\cite{Ha}) extends it to an operation 
on KZ-type equations \eqref{eq:KZ}.
The integral transformations 
$K_x^{\mu-\alpha,\alpha}$ and $\tilde K_{x,y}^{\mu-\lambda,\lambda}$
correspond to some transformations of differential equations generated by  
middle convolutions, additions and coordinate transformations
and hence the middle convolutions are useful to study our hypergeometric functions.
In particular, $F^{p,q,r}_{p',q',r'}(x,y)$ satisfying the condition (F1) or (F2) in \S\ref{sec:Cond}
is realized as a component of a solution to a KZ-type equation. 
Studying this function from the viewpoint of KZ-type equation,  
we obtain several results, which are explained in this section.
Some of the results are announced in \cite{Oi} in the case satisfying (F1).

\subsection{KZ-type equation and middle convolution}\label{sec:KZ}
Any rigid Fuchsian differential equation on $\mathbb P^1$ with more than 3 singular points
can be extended to a KZ-type equation (Knizhnik-Zamolodchikov)  (cf.~\cite{KZorg})
regarding singular points as new variables.
This is proved by using middle convolutions and additions on  KZ-type equations (cf.~\cite{DR,Ha})
and we obtain  many \textit{hypergeometric systems} and functions with several variables.
When the number of singular points is 4, we get hypergeometric equations of two variables.
Then, the classical Appell's $F_1$ system corresponds to the spectral type $21,21,21,21$, and Appell's $F_2$ system corresponds to $211,22,31,31$ (cf.~\cite[\S3]{Oir}).
This extension is also valid, when the original equation has
unramified irregular singularities (cf.~\cite{Okzv}).

The local structure of the ordinary differential equation \eqref{eq:ODE} is given by the Riemann scheme and the equation is called rigid when the Riemann scheme uniquely determines the equation.
For simplicity, we assume that the equation is non-resonant, namely,
the residue matrices $A_{i,j}$ are semisimple and the difference
of eigenvalues of any matrix $A_{i,j}$ is not a non-zero integer.
The dimension of the moduli space describing the global structure with fixed local structure is given by the rigidity index defined by Katz \cite{katz1996rigid} which can be computed by the spectral data of the Riemann scheme.
The spectral data are described by 
a list of multiplicities of eigenvalues of $A_{i,j}$. 
Katz \cite{katz1996rigid} proves that if the equation is rigid, it is constructed from the trivial equation $\frac{du}{dx}(x)=0$ by successive applications
of middle convolutions and additions.

In the case of KZ-type equation, the Riemann scheme is not sufficient to describe 
the local structure of the equation.  
It determines the local structure at a generic point of the singularity defined by $x_i=x_j$. 
The local structure at a generic point defined by $\{x_i=x_j,\,x_k=x_\ell\}$
with $\#\{i,\,j,\,k,\,\ell\}=4$ is also important.
The point is the intersection of the hyperplanes defined by $x_i=x_j$ and $x_k=x_\ell$.  
In this case, the commutator $[A_{i,j},A_{k,\ell}]=A_{i,j}A_{k,\ell}-A_{k,\ell}A_{i,j}$ is zero and the simultaneous eigenspace decomposition of these matrices is important.

Suppose $n =3$.
Then these two kinds of data are sufficient to describe the local structure of the equation.
Moreover 
the second author \cite{Okz} describes the transformation of these data by the middle convolution
as in the case of ordinary differential equation due to \cite{DR}.
Thus we get the Riemann scheme of the differential equation of the hypergeometric equation with (F1) or (F2) and the equation is rigid because these local data uniquely 
determine the equation.
The spectral type of a KZ-type equation corresponding to that of the ordinary differential equation
will be the data consisting of the multiplicities of eigenvalues of $A_{i,j}$ together with
the multiplicities of eigenvalues of $A_{k\ell}$ restricted to the eigenspaces of $A_{ij}$
with $\#\{i,\,j,\,k,\,\ell\}=4$. 
We note that there are 15 pairs of these matrices. 
We denote the eigenvalues of a square matrices $A$ and 
the simultaneous eigenvalues of commuting matrices $A$ and $A'$ by
\begin{align}
\label{eq:simeigen}
\begin{split}
[A]&=\{[\lambda_1]_{m_1},[\lambda_2]_{m_2},\ldots\},\\
[A:A']&=\{[\lambda_1:\lambda'_1]_{m_1},[\lambda_2:\lambda'_2]_{m_2},\ldots\}.
\end{split}
\end{align}
The second line means that the dimension of
the simultaneous eigenspace with the eigenvalues $\lambda_1$ of $A$ and  $\lambda'_1$ of $A'$
equals $m_1$.  
We will simply  denote $[\lambda_1:\lambda'_1]_1$ by  $[\lambda_1:\lambda'_1]$.

From the data $[A_{i,j},A_{k.\ell}]\   (0<i<j\le 4,\ i<k<\ell\le 4,\ j\ne k,\ j\ne\ell)$, we obtain the Riemann scheme of the equation, which is
the data  $[A_{i,j}]\ (0\le i<j\le 4)$ and the Riemann schemes and the spectral type of the ordinary differential equations on a singular 
line defined by $x_i=x_j$ which is satisfied by the boundary value of the solution with respect to an eigenvalue of $A_{ij}$.

\bigskip
A KZ-type equation $\mathcal M$ is 
\begin{align}
 \mathcal M\, &:\,\frac{\partial u}{\partial x_i}=\sum_{\substack{0\le\nu\le n\\ \nu\ne i}}
\frac{A_{i,\nu}}{x_i-x_\nu}u
 \quad(i=0,\dots,n),\label{eq:KZ}\\
 A_{i,j}&=A_{j,i}\in M(N,\mathbb C)\quad(i,\,j\in\{0,1,\dots,n+1\})\notag
\intertext{with}
  A_{i,i}&=A_\emptyset=A_i=0,\quad A_{i,n+1}:=-\sum_{\nu=0}^n A_{i,\nu},\notag\\
 A_{i_1,i_2,\dots,i_k}&:=\sum_{1\le \nu<\nu'\le k}A_{i_\nu,i_{\nu'}}\quad
(\{i_1,\dots,i_k\}\subset\{0,\dots,n+1\}).\notag 
\end{align}
Here, $A_{ij}$ is called the residue matrix along the hypersurface defined by $x_i=x_j$
and $u$ is a column $N$-vector of unknown functions. 
We always assume the compatibility  condition
\begin{equation}
 [A_I,A_J]=0\quad\text{if }I\cap J=\emptyset\text{ or }I\subset J\text{ with }
 I,\,J\subset \{0,\dots,n+1\}.
 \label{eq:compatibility}
\end{equation}

\noindent
Moreover, we assume $\mathcal M$ is homogeneous (cf.~\cite{Okz}), namely,
\begin{equation}\label{eq:homog}
  A_I=0\quad(\#I=n+1).
\end{equation}
The symmetric group $S_{n+3}$ acts on the space of KZ-type equations \eqref{eq:KZ}
as the permutation group of the indices of the matrices $A_{i,j}$.

Any rigid Fuchsian ordinary differential equation
\begin{equation}\label{eq:ODE}
\frac{du}{dx}=\sum_{i=1}^n \frac{A_i}{x-x_i}u
\end{equation}
can be extended to a KZ-type equation \eqref{eq:KZ} by putting $x=x_0$ and $A_i=A_{0,i}$,
which is proved by \cite{Ha,Okz}.
Using the fractional linear transformation of $\mathbb P^1$ which maps to $x_{n-1}$, $x_n$ and 
$x_{n+1}$ to $0$, $1$ and $\infty$, respectively, only $n-1$ variables $x_0,\dots,x_{n-2}$ remain.
These variables are treated as independent variables on the level of KZ-type equation \eqref{eq:KZ}.

In this section, we consider the case $n=3$.
Then, by the relation 
\begin{equation}\label{eq:homogeneous_relation}
    A_{01}+A_{02}+A_{03}+A_{12}+A_{13}+A_{23}=0,
\end{equation}
a tuple $(A_{01},A_{02},A_{03},A_{12},A_{13})$ determines $\mathcal M$.

We use a normalization $(x_0,x_1,x_2,x_3,x_4)=(x,y,1,0,\infty)$ so that the equation $\mathcal M$ is
\begin{align}\label{eq:KZxy}
  \mathcal M : 
  \begin{cases}
   \displaystyle\frac{\p u}{\p x}=
   \frac{A_{01}}{x-y}u+\frac{A_{02}}{x-1} u + \frac{A_{03}}{x} u,\\
   \displaystyle\frac{\p u}{\p y}=
   \frac{A_{01}}{y-x}u+\frac{A_{12}}{y-1} u + \frac{A_{13}}{y} u.
  \end{cases}
\end{align}
\begin{rem}\label{rem:blowup}
The involutive coordinate transformations
\begin{align*}
 (x_0,x_1,x_2,x_3,x_4)\leftrightarrow   (x_2,x_1,x_0,x_4,x_3)  \qquad& \to \qquad
  (x,y)\leftrightarrow (x,\tfrac xy)\\
 (x_0,x_1,x_2,x_3,x_4)\leftrightarrow   (x_0,x_2,x_1,x_4,x_3)  \qquad& \to \qquad
  (x,y)\leftrightarrow (\tfrac yx,y)
\end{align*}
corresponding to elements of $S_5$ give 
the coordinates resolving the singularities of \eqref{eq:KZxy} at the origin (cf. Equation \eqref{eq:dynkin}).
The transformation in the above first line gives the following correspondences in $(x,y)$-space.

\medskip
\hspace{1.95cm} 
$\{|x|<\epsilon,\ |y|<C|x|\}\ \leftrightarrow\  \{|x|<\epsilon,\ |y|>C^{-1}\}$

\hspace{3.6cm}
$x=y=0\ \leftrightarrow\ x=0$

\quad
\begin{tikzpicture}
\path[fill=black!40] (0,0)--(1/3,-1/10)--(1/3,1/2) --cycle;
\path[fill=black!40] (0,2.15)--(1/3,2.15)--(1/3,3/4) --(0,3/4) --cycle;
\draw (0,-0.1)--(0,2.3) (1,-0.1)--(1,2.3)  (2,-0.1)--(2,2.3)
      (-0.1,0)-- (2.3,0)  (-0.1,1)-- (2.3,1)  (-0.1,2)-- (2.3,2)
      (-0.1,-0.1)--(2.3,2.3);
\node at (-0.5,0.1) {$y=0$};
\node at (-0.5,1.1) {$y=1$};
\node at (-0.5,2.1) {$y=\infty$};
\node at (0,-0.3) {$x=0$};
\node at (1,-0.3) {$x=1$};
\node at (2,-0.3) {$x=\infty$};
\node at (1.5,1.4) {$x=y$};
\node at (0.45,1.7) {$\downarrow$};
\node at (0.45,0.2) {$\uparrow$};
\end{tikzpicture}\qquad
\raisebox{1.5cm}{$\mathcal M\  \leftrightarrow\ 
\begin{cases}
 \displaystyle
 \frac{\p u}{\p x}=\frac{A_{12}}{x-y}u+\frac{A_{02}}{x-1}u+\frac{A_{23}}x u,\\
 \displaystyle\frac{\p u}{\p y}=\frac{A_{12}}{y-x}u+\frac{A_{01}}{y-1}u+\frac{A_{13}}y u.
\end{cases}
$}
\nolinebreak
\end{rem}

\medskip
The middle convolution $\mathrm{mc}_\mu$ corresponds to an integral transformation
\begin{align*}
   (I_0^\mu u)(x)&:=\frac1{\Gamma(\mu)}\int_0^x t^{\mu-1}u(x-t)dt\\
                 &=\frac{x^{\mu}}{\Gamma(\mu)}\int_0^1 (1-t)^{\mu-1}u(tx)dt
.
\end{align*}
The middle convolution $\mathrm{mc}_\mu$ of the KZ-type equation \eqref{eq:KZ}
is defined through the convolution $\tilde{\mathrm{mc}}_\mu$ of \eqref{eq:KZ},
which is the equation satisfied by the vector of functions
\[
 (\tilde{\mathrm{mc}}_\mu u)(x,y):=\begin{pmatrix}
  I_{x,0}^{\mu+1} \frac{u(x,y)}{x-y}\\
  I_{x,0}^{\mu+1} \frac{u(x,y)}{y}\\
  I_{x,0}^{\mu+1} \frac{u(x,y)}{x}
\end{pmatrix}=
\begin{pmatrix}
    \tfrac1{\Gamma(\mu+1)}\int_0^x (1-t)^\mu \frac{u(t,y)}{t-y}d t\\
    \tfrac1{\Gamma(\mu+1)}\int_0^x (1-t)^\mu \frac{u(t,y)}{y}d t\\
    \tfrac1{\Gamma(\mu+1)}\int_0^x (1-t)^\mu \frac{u(t,y)}{t}d t
\end{pmatrix}
\]
for the solution $u(x)$ to \eqref{eq:KZ}.
In view of an equality
$K_x^{\mu,\lambda}=x^{-\mu-\lambda}\circ I_{x,0}^\mu\circ x^\lambda$,
we define
\begin{align}\label{eq:IKZ}
 (\tilde K_x^{\mu,\lambda} u)(x,y)
 &:=\begin{pmatrix} K_x^{\mu+1,\lambda}\frac {xu(x,y)}{x-y}\\
   K_x^{\mu+1,\lambda}\frac {xu(x,y)}{x-1}\\
   K_x^{\mu+1,\lambda} u(x,y)
\end{pmatrix}.
\end{align}
For a solution $u$ to the equation \eqref{eq:KZ}, we put 
$\tilde u=\tilde K_x^{-\tau-\lambda,\lambda} u$.
Then we have a system of equations
\begin{align}
 \frac{\p\tilde u}{\p x_i}&=\sum_{\substack{0\le\nu\le 3\\\nu\ne i}}\frac{\tilde A_{i,\nu}}{x_i-x_\nu}\tilde u\quad(0\le i\le 3)\label{eq:KKZ}
\end{align}
satisfied by $\tilde u$.  
These matrices $\tilde A_{i,j}$ are given by 
{\small%
\begin{align*}\label{eq:KxKZ}
 \tilde A_{01}&=\begin{pmatrix}
  A_{01}-\tau-\lambda & A_{02} & A_{03}+\lambda\\
 0 & 0 & 0\\
 0 & 0 & 0
\end{pmatrix},\qquad\ \ 
\tilde A_{02}=\begin{pmatrix}
 0 & 0 & 0\\
 A_{01} &  A_{02}-\tau-\lambda & A_{03}+\lambda\\
 0 & 0 & 0
\end{pmatrix},\allowdisplaybreaks\\ 
\tilde A_{03}&=\begin{pmatrix}
 \tau & 0 & 0\\
 0 & \tau & 0\\
 A_{01} & A_{02} & A_{03}
\end{pmatrix},\qquad\ \ \ \,
\tilde A_{04}=\begin{pmatrix}
 -A_{01}+\lambda & -A_{02} & -A_{03}-\lambda\\
 -A_{01} & -A_{02}+\lambda & -A_{03}-\lambda\\
 -A_{01} & -A_{02} & -A_{03}
\end{pmatrix},\allowdisplaybreaks\\
\tilde A_{12}&=\begin{pmatrix}
 A_{12}+A_{02} & -A_{02} & 0\\
 -A_{01} & A_{12}+A_{01} & 0\\
 0 & 0 & A_{12}
\end{pmatrix},\ \ 
\tilde A_{13}=\begin{pmatrix}
 A_{13}+A_{03}+\lambda & 0 & -A_{03}-\lambda\\
 0 & \!A_{13}\! & 0\\
 -A_{01} & 0 & A_{01}+A_{13}
\end{pmatrix},\\
\tilde A_{14}&=
\begin{pmatrix}
 A_{23}+\tau & 0 & 0\\
 A_{01} & A_{02}+A_{03}+A_{23} & 0\\
 A_{01} & 0 & A_{02}+A_{03}+A_{23}
\end{pmatrix},
\\
\tilde A_{23}&=
\begin{pmatrix}
  A_{23} & 0 & 0\\
 0 & A_{03}+A_{23}+\lambda & -A_{03}-\lambda\\
 0 & -A_{02} & A_{02}+A_{23}
\end{pmatrix},\allowdisplaybreaks\\
 \tilde A_{24}&=\begin{pmatrix}
 A_{01}+A_{13}+A_{03} & A_{02} & 0\\
 0 & A_{13}+\tau&0\\
 0 & A_{02} & A_{01}+A_{13}+A_{03}
\end{pmatrix},\allowdisplaybreaks\\
 \tilde A_{34}&=\begin{pmatrix}
 A_{12}+A_{01}+A_{02}-\tau-\lambda& 0 & A_{03}+\lambda\\
 0 & A_{12}+A_{01}+A_{02}-\tau-\lambda & A_{03}+\lambda\\
 0 & 0 & A_{12}
\end{pmatrix}.%
\end{align*}}
Here we write $A_{ij}$ in place of $A_{i,j}$ for simplicity.

We define a subspace $\mathcal{L}\subset\mathbb C^{3N}$ by
\begin{align}
\begin{split}
 \mathcal L&:=\begin{pmatrix}\ker A_{01}\\ \ker A_{02}\\ \ker (A_{03}+\lambda)\end{pmatrix}
 +\ker(\tilde A_{04}+\tau)\\
 &=\begin{pmatrix}\ker A_{01}\\ \ker A_{02}\\ \ker (A_{03}+\lambda)\end{pmatrix}
 +\ker\begin{pmatrix}A_{01}-\tau-\lambda&A_{02}&A_{03}+\lambda\\ A_{01}&A_{02}-\tau-\lambda &A_{03}+\lambda\\ A_{01}&A_{02}&A_{03}-\tau\end{pmatrix}.
\end{split}\end{align}
It satisfies  $\tilde A_{i,j}\mathcal L\subset \mathcal L$.
The matrices $\tilde A_{i,j} $ induce linear transformations on
the quotient space $\mathbb  C^{3N}/\mathcal L$.
We fix a basis of the quotient space and write 
$A_{i,j}$ for the matrices corresponding to the induced linear transformations.
Then $\bar A_{i,j}$ are square matrices of size $3N-\dim\mathcal L$.
Thus, we obtain another KZ-type equation
\begin{align}
\bar{\mathcal M} : 
 \frac{\p\bar u}{\p x_i}&=\sum_{\substack{0\le\nu\le 3\\\nu\ne i}}\frac{\bar A_{i,\nu}}{x_i-x_\nu}\bar u\quad(1\le i\le 3).\label{eq:RKKZ}
\end{align}

If $\lambda$ and $\tau$ are generic so that 
\begin{equation}
    \ker (A_{03}+\lambda)=0\ \text{and}\ \ker(\tilde{A}_{04}+\tau)=0,
    \label{eq:genericity_condition}
\end{equation}
we have
\begin{equation}
 \mathcal L=\mathcal L_1+\mathcal L_2,\quad
 \mathcal L_1=\begin{pmatrix}
  \ker A_{01}\\
  0\\
  0
 \end{pmatrix},\quad
 \mathcal L_2=\begin{pmatrix}
  0\\
  \ker A_{02}\\
  0
 \end{pmatrix}.
\label{eq:L1L2}
\end{equation}

\begin{rem}\label{rem:KZ}
{\rm i)}\  
Since the KZ-type equation $\mathcal M$ have holomorphic parameters, we define middle 
convolution in a neighborhood of fixed parameters by defining the above kernels for the generic points in a neighborhood of the fixed parameters. Then the resulting 
system obtained by the middle convolution has also holomorphic parameters in the neighborhood. Then we may examine the reducibility of the resulting equation, which is related with the difference of the kernels 
for the fixed parameter.

{\rm ii)} \ 
Since $\mathrm{ker}\,(A_{03}+\lambda)=0$ in the above, 
we may assume that the residue class of the last component of $\tilde u$
is the last component of $\bar u$ and therefore 
if $\phi(x,y)$ is the last component of a solution to \eqref{eq:KZ},
then $\tilde K_x^{-\tau-\lambda,\lambda}\phi$ 
is the last component of a solution to $\bar{\mathcal M}$.
\end{rem}

We write $\Ad\bigl((x_0-x_3\bigr)^\lambda)$ for the transformation defined by the replacement
\[
 A_{03}\mapsto A_{03}+\lambda,\ A_{04}\mapsto A_{04}-\lambda,\ 
 A_{34}\mapsto A_{34}-\lambda
\]
and $A_{ij}\mapsto A_{ij}$ for the other $A_{ij}$.
Note that $\Ad\bigl((x_0-x_3\bigr)^\lambda)$ maps a KZ-type system to a KZ-type system, namely, the compatibility condition \eqref{eq:compatibility} is preserved.
We have
\begin{equation}
\label{eq:Kmc}
 \tilde K_x^{-\tau-\lambda,\lambda}=
 \Ad((x_0-x_3)^{\tau})\circ \mathrm{mc}_{-\tau-\lambda}\circ \Ad((x_0-x_3)^{\lambda}).
\end{equation}
Here $\Ad((x_0-x_3)^{\lambda})$ is the transformation of the system  \eqref{eq:KZ} corresponding 
to the transformation $u\mapsto (x_0-x_3)^{\lambda}u$ and 
the transformation $\tilde K_x^{-\tau-\lambda,\lambda}$ keeps homogeneity \eqref{eq:homog}.

The following theorem is a direct consequence of \cite[Theorem~7.1]{Okz}.
\begin{thm}\label{thm:TKz}
Retain the above notation. 

Consider the transformation $\mathrm{mc}_\mu$
in place of $\tilde K_x^{-\tau-\lambda,\lambda}$. 
Assume $\mu$ is generic.
For\/ $\{i,j,k\}=\{1,2,3\}$ 
we have (cf.~\eqref{eq:simeigen})
\begin{align*}
[\tilde A_{0,k}:\tilde A_{i,j}]&=[A_{0,k}:A_{i,j}]\cup[-\mu:A_{i,j}]\cup[-\mu:A_{k,4}],\\
[\tilde A_{0,k}:\tilde A_{i,j}]|_{\mathcal L_\nu}&=
\begin{cases}
[-\mu:A_{k,4}]|_{\mathrm{Ker}\,A_{0,\nu}}\ \ \ \ &\quad(\nu=i,\,j),\\
[0:A_{i,j}]|_{\mathrm{Ker}\,A_{0,k}}&\quad(\nu=k),
\end{cases}
\allowdisplaybreaks\\
[\tilde A_{0,k}:\tilde A_{i,4}]&=[A_{0,k};A_{i,4}]\cup[0:A_{i,4}]\cup[0:A_{j,k}-\mu],\\
[\tilde A_{0,k}:\tilde A_{i,4}]|_{\mathcal L_\nu}&=
\begin{cases}
[-\mu:A_{j,k}-\mu]_{\mathrm{Ker}\,A_{0,i}}&\quad(\nu=i),\\
[0:A_{i,4}]|_{\mathrm{Ker}\,A_{0,j}}&\quad(\nu=j),\\
[\mu:A_{i,4}]|_{\mathrm{Ker}\,A_{0,k}}&\quad(\nu=k),
\end{cases}
\allowdisplaybreaks\\
[\tilde A_{0,4}:\tilde A_{i,j}]&=[A_{0,4}:A_{i,j}]\cup[0:A_{i,j}]\cup
[0:A_{k,4}],\\
[\tilde A_{0,4}:\tilde A_{i,j}]|_{\mathcal L_\nu}&=
\begin{cases}
[0:A_{k,4}]|_{\mathrm{Ker}\,A_{0,\nu}}\quad \ \ \ &\quad(\nu=i,\,j),\\
[0:A_{i,j}]|_{\mathrm{Ker}\,A_{0,k}}&\quad(\nu=k),
\end{cases}
\allowdisplaybreaks\\
[\tilde A_{k,4}:\tilde A_{i,j}\}]&=[A_{i,j}:A_{i,j}]\cup
[A_{k,4}+\mu:A_{i,j}]\cup[A_{k,4}:A_{k,4}+\mu],\\
[\tilde A_{k,4}:\tilde A_{i,j}]|_{\mathcal L_\nu}&=
\begin{cases}
[A_{k,4}+\mu:A_{k,4}]|_{\mathrm{Ker}\,A_{0,\nu}}&\quad(\nu=i,\,j),\\
[A_{i,j}:A_{i,j}|_{\mathrm{Ker}\,A_{0,k}}&\quad(\nu=k)
.
\end{cases}
\end{align*}
\end{thm}

If a function $u(x,y)$ is a component of a solution to a certain 
KZ-type equation, we simply say that $u(x,y)$ is realized as a solution to a KZ-type equation.
Then Theorem~\ref{thm:IntegralRepresentation} and 
the following remark  assure that the hypergeometric function 
$F^{p,q,r}_{p',q',r'}(x,y)$ with the condition (F1) or (F2) is realized as a solution to a
KZ-type equation.

\begin{rem} {\rm i)} \ 
Suppose $u(x,y)$ is realized as a solution to a KZ-type equation at the origin.
Then $K_x^{\mu,\alpha}u$, $K_y^{\mu,\alpha}u$ and $\tilde K_{x,y}^{\mu,\alpha}u$
have the same property.
Hence  $F^{p,q,r}_{p,q,r}(x,y)$ is a solution to a KZ-type equation.

{\rm ii)} \ 
Suppose $u(x,1-y)$ is a component of a solution to a KZ-type equation at the origin.
Then $K_x^{\mu,\alpha}u$ and  $K_y^{\mu,\alpha}u$ have the same property.
\end{rem}
\subsection{Example: Generalization of Appell's $F_1$}\label{sec:F1}
In this section we examine the hypergeometric function
\begin{equation}
 F_{p,q,r}
 \left(\begin{smallmatrix}
  \alpha&\beta&\gamma\\\alpha'&\beta'&\gamma'
\end{smallmatrix};x,y\right):=
F_{p,q,r}^{p,q,r} \left(\begin{smallmatrix}
  \alpha&\beta&\gamma\\\alpha'&\beta'&\gamma'
\end{smallmatrix};x,y\right)\text{ \ with \ } 
\alpha'_p=\beta'_q=0.
\end{equation}

Then this hypergeometric function is realized as a solution to the KZ-type equation \eqref{eq:KZxy} or \eqref{eq:KZ} 
whose construction is given in the last section.

\begin{thm}\label{thm:F1sim}
The simultaneous eigenvalues with their multiplicities of the pair of commuting residue matrices 
at the 15 normally crossing points (cf.~\cite[\S6]{Okz}) are give by  
\begin{align*}
[A_{01}:A_{23}]
       &=\{[0:\alpha_i+\beta_j],\,[0:\gamma_k]_{p+q-1},\,[-\alpha''-\beta'':\gamma_k]\}
,\allowdisplaybreaks\\
[A_{01}:A_{24}]
       &=\{[0:\alpha'_i+\beta'_j],\,[0:\gamma'_k]_{p+q-1},\,[-\alpha''-\beta'':\gamma'_k]\}
,\allowdisplaybreaks\\
[A_{01}:A_{34}]
        &=\{[0:0]_{R-(p+q+r)+1},\,[0:-\alpha''-\beta''-\gamma''],\\
        &\!\qquad[-\alpha''-\beta'':-\alpha''-\beta''-\gamma''],\,
      [0:-\beta''-\gamma'']_{p-1},\\
      &\!\qquad [0:-\alpha''-\gamma'']_{q-1},\,
      [-\alpha''-\beta'':-\alpha''-\beta'']_{r-1}\}
,\allowdisplaybreaks\\
[A_{02}:A_{13}]
       &=\{[-\alpha''-\gamma'':\beta'_j],\,[0:\beta'_j]_{p+r-1},\,[0:\alpha_i+\gamma'_k]\}
,\allowdisplaybreaks\\
[A_{02}:A_{14}]
   &=\{[-\alpha''-\gamma'':\beta_j],\,[0:\beta_j]_{p+r-1},\,[0:\alpha'_i+\gamma_k]\}
,\allowdisplaybreaks\\
[A_{02}:A_{34}]
 &=\{[0:0]_{R-(p+q+r)+1},\,[-\alpha''-\gamma'':-\alpha''-\beta''-\gamma''],\\
 &\!\qquad [0:-\alpha''-\beta''-\gamma''],\,
  [0:-\beta''-\gamma'']_{p-1},\\
  &\!\qquad [-\alpha''-\gamma'':-\alpha''-\gamma'']_{q-1},\,[0:-\alpha''-\beta'']_{r-1}\},\allowdisplaybreaks\\
[A_{03}:A_{12}]
 &=\{[\alpha'_i:-\beta''-\gamma''],\,[\alpha'_i:0]_{q+r-1},\,[\beta_j+\gamma'_k:0]\}
,\allowdisplaybreaks\\ 
[A_{03}:A_{14}]
  &=\{[\alpha'_i:\beta_j],\,[\beta_j+\gamma'_k:\beta_j],\,[\alpha'_i:\alpha'_i+\gamma_k]\}
,\allowdisplaybreaks\\
[A_{03}:A_{24}]
 &=\{[\alpha'_i:\alpha'_i+\beta'_j],\,[\alpha'_i:\gamma'_k],\,[\beta_j+\gamma'_k:\gamma'_k]\}
,\allowdisplaybreaks\\
[A_{04}:A_{12}]
 &=\{[\alpha_i:-\beta''-\gamma''],\,[\alpha_i:0]_{q+r-1},\,[\beta'_j+\gamma_k:0]\}
,\allowdisplaybreaks\\
[A_{04}:A_{13}]
 &=\{[\alpha_i:\beta'_j],\,[\alpha_i:\alpha_i+\gamma'_k],\,[\beta'_j+\gamma_k:\beta'_j]\}
,\allowdisplaybreaks\\
[A_{04}:A_{23}]
 &=\{[\alpha_i:\alpha_i+\beta_j],\,[\alpha_i:\gamma_k],\,[\beta'_j+\gamma_k:\gamma_k]\}
,\allowdisplaybreaks\\
[A_{12}:A_{34}]
 &=\{[0:0]_{R-(p+q+r)+1},\,[-\beta''-\gamma'':-\alpha''-\beta''-\gamma''],\\
 &\!\qquad [0:-\alpha''-\beta''-\gamma''],\,
[-\beta''-\gamma'':-\beta''-\gamma'']_{p-1},\\
&\!\qquad [0:-\alpha''-\gamma'']_{q-1},\,[0:-\alpha''-\beta'']_{r-1}\},\\
[A_{13}:A_{24}]
&=\{[\beta'_j:\alpha'_i+\beta'_j],\,[\beta'_j:\gamma'_k],\,[\alpha_i+\gamma'_k:\gamma'_k]\}
,\allowdisplaybreaks\\
[A_{14}:A_{23}]
 &=\{[\beta_j:\alpha_i+\beta_j],\,[\beta_j:\gamma_k],\,[\alpha'_i+\gamma_k:\gamma_k]\}
.
\numberthis\label{eq:list_of_F1}
\end{align*}
Here
\begin{align}
R&:=pq+qr+rp,\quad 
1\le i\le p,\ 1\le j\le q,\ 1\le k\le r, \label{eq:rankF1}\\
\alpha''&:=\sum_{i=1}^p(\alpha_i+\alpha'_i),\  \beta'':=\sum_{j=1}^q(\beta_j+\beta'_j),\ 
 \gamma'':=\sum_{k=1}^r(\gamma_k+\gamma'_k).
\label{eq:doubleprime}
\end{align}
\end{thm}
\begin{proof}
When $p=q=1$ and $r=0$, the claim is obvious because the rank of 
the system equals 1 and 
characterized by the solution
\[
F_{1,1,0}\left(\begin{smallmatrix}\alpha&\beta&0\\\alpha'&\beta'&0\end{smallmatrix};
\alpha',\beta';x,y\right)
=x^{\alpha'}(1-x)^{-\alpha-\alpha'}y^{\beta}(1-y)^{-\beta-\beta'}.
\] 
Assuming the above list, we get the resulting list under the convolution and addition corresponding to 
$p\mapsto p+1$
which is realized by 
\begin{equation}\label{eq:ptop1}
\begin{split}
 K_x^{-\alpha_0-\alpha'_0,\alpha_0}
 &=C\Ad(x^{\alpha_0})\circ\mathrm{mc}_{-\alpha''_0}\circ\Ad(x^{\alpha'_0}
  )
,\\
\alpha_0,\ \alpha'_0&\in\mathbb C,\ \alpha''_0=\alpha_0+\alpha'_0\text{ and a constant }C.
\end{split}
\end{equation}
To avoid confusion of the notation, we assume the condition (T) and introduce parameters 
$\alpha_0$ and $\alpha'_0$ which will be $\alpha_{p+1}$ and $\alpha'_{p+1}$, respectively.
Hence $\alpha''=\sum_{i=1}^p(\alpha'_i+\alpha'_i)$ etc. 
Then Theorem~\ref{thm:TKz} and \eqref{eq:ptop1} show
\begin{align*}
[\tilde A_{01}:\tilde A_{23}]&=[A_{01}-\alpha''_0:A_{23}]\cup[0:A_{23}]\cup
[0,A_{14}+\alpha_0],\\
[A_{01}-\alpha''_0:A_{23}]&=\{[-\alpha''_0,\alpha_i+\beta_j],\,
 [-\alpha''_0:\gamma_k]_{p+q-1},\,[-\alpha''-\beta''-\alpha''_0:\gamma_k]\}.\\
[0:A_{23}]&=\{[0:\alpha_i+\beta_j],\,[0:\gamma_k]_{p+q}\},\\
[0,A_{14}+\alpha_0]&=\{[0:\alpha_0+\beta_j]_{p+r},\,[0:\alpha'_i+\gamma_k+\alpha_0]\}
\intertext{and}
[\tilde A_{01}:\tilde A_{23}]|_{\mathcal L_1}&=[A_{01}-\alpha''_0:A_{23}]|_{\ker A_{01}}
\\&=\{[-\alpha''_0:\alpha_i+\beta_j],\,[-\alpha''_0:\gamma_k]_{p+q-1}\},\\
[\tilde A_{01}:\tilde A_{23}]|_{\mathcal L_2}&=[0,A_{14}+\alpha_0]|_{\ker A_{02}}
\\&
=\{[0:\alpha_0+\beta_j]_{p+r-1},\,[0:\alpha_0+\alpha'_i+\gamma_k]\}.
\end{align*} 
Here we note that $\dim\mathcal L_1=pq+qr+rp-r$ and $\dim\mathcal L_2=pq+qr+rp-q$. 
Since $\bar A_{ij}$ are matrices induced from $A_{ij}$ on the quotient space 
$\mathbb C^{3N}/(\mathcal L_1\oplus\mathcal L_2)$, we have the required result
\begin{align*}
[\bar A_{01}:\bar A_{23}]
       &=\{[0:\alpha_i+\beta_j],\,[0:\gamma_k]_{p+q},\,[-\alpha''-\beta'':\gamma_k]\}
\end{align*}
by putting $\alpha_{p+1}=\alpha_0$ and  $\alpha'_{p+1}=\alpha'_0$.

This result can be read from the first part of the following table.
The part without an underline means it is contained in $\mathcal L_1\oplus\mathcal L_2$.
The multiplicity without an underline means that $\mathcal L_1\oplus\mathcal L_2$ contains it with 
a smaller multiplicity, which is obtained comparing it with the original list.

The other 14 cases are similarly obtained and the result is as follows and hence we get the required result under the transformation corresponding to $p\mapsto p+1$.
\scriptsize
\begin{align*}
[\tilde A_{01}:\tilde A_{23}]
&=\{[-\alpha''_0:\alpha_i+\beta_j],\,[-\alpha''_0:\gamma_k]_{p+q-1},\,\underline{[-\alpha''-\beta''-\alpha''_0:\gamma_k]},\,\underline{[0:\gamma_k]_{p+q}},\,
\\&\qquad
\underline{[0:\alpha_i+\beta_j]},\,\underline{[0:\alpha_0+\beta_j]}_{p+r},\,[0:\alpha_0+\gamma_k+\alpha'_i]\},\\
[\tilde A_{01}:\tilde A_{24}]
&=\{[-\alpha''_0:\alpha'_i+\beta'_j],\,[-\alpha''_0:\gamma'_k]_{p+q-1},\,\underline{[-\alpha''-\beta''-\alpha''_0:\gamma'_k]},\,
\underline{[0:\gamma'_k]_{p+q}},\,
\\&\qquad 
\underline{[0:\alpha'_i+\beta'_j]},\,\underline{[0:\alpha'_0+\beta'_j]}_{p+r},\,[0:\alpha_i+\gamma'_k+\alpha'_0]\},
\allowdisplaybreaks\\
[\tilde A_{01}:\tilde A_{34}]
&=\{[-\alpha''_0:-\beta''-\gamma''-\alpha''_0]_{p-1},\,[-\alpha''_0:-\alpha''-\beta''-\alpha''_0],\,
\\&\qquad[-\alpha''_0:-\alpha''-\gamma''-\alpha''_0]_{q-1},\,
[-\alpha''_0:-\alpha''_0]_{R-(p+q+r)+1},\,
\\&\qquad
\underline{[-\alpha''-\beta''-\alpha''_0:-\alpha''-\beta''-\gamma''-\alpha''_0]},\,
\underline{[-\alpha''-\beta''-\alpha''_0:-\alpha''-\beta''-\alpha''_0]_{r-1}},\,
\\&\qquad
\underline{[0:-\alpha''-\beta''-\gamma''-\alpha''_0]}_{2}
,\,
\underline{[0:-\alpha''-\gamma''-\alpha''_0]_{q-1}},\,
\underline{[0:0]_{R-p}},\,
\underline{[0:-\beta''-\gamma'']_{p}},
\\&\qquad[0:-\alpha''_0]_{R-(p+q+r)+1},\,
[0:-\alpha''-\beta''-\alpha''_0]_{r-1},\,[0:-\beta''-\gamma''-\alpha''_0]_{p-1}
\},\allowdisplaybreaks\\
[\tilde A_{02}:\tilde A_{13}]
 &=\{
   \underline{[-\alpha''-\gamma''-\alpha''_0:\beta'_j]},\,[-\alpha''_0:\beta'_j]_{p+r-1},\,[-\alpha''_0:\alpha_i+\gamma'_k],\,\underline{[0:\beta'_j]_{p+r}},\,
  \\&\qquad\underline{[0:\alpha_i+\gamma'_k]},\,
   \underline{[0:\alpha_0+\gamma'_k]}_{p+q},\,[0:\alpha_0+\alpha'_i+\beta'_j]\},
\allowdisplaybreaks\\
[\tilde A_{02}:\tilde A_{14}]
&=\{\underline{[-\alpha''-\gamma''-\alpha''_0:\beta_j]},\,[-\alpha''_0:\beta_j]_{p+r-1},\,[-\alpha''_0:\gamma_k+\alpha'_i],\,\underline{[0:\beta_j]_{p+r}},\,
\\&\qquad 
 \underline{[0:\alpha'_i+\gamma_k]},\,\underline{[0:\alpha'_0+\gamma_k]}_{p+q},\,[0:\alpha_i+\beta_j+\alpha'_0]\},\\
[\tilde A_{02}:\tilde A_{34}]
&=\{\underline{[-\alpha''-\gamma''-\alpha''_0,-\alpha''-\beta''-\gamma''-\alpha''_0]},\,
 \underline{[-\alpha''-\gamma''-\alpha''_0:-\alpha''-\gamma''-\alpha''_0]_{q-1}},\,
\\&\qquad
  [-\alpha''_0:-\alpha''-\beta''-\gamma''-\alpha''_0],\,
  [-\alpha''_0:-\beta''-\gamma''-\alpha''_0]_{p-1},\,
\\&\qquad
  [-\alpha''_0:-\alpha''_0]_{R-(p+q+1)+1},\,
  [-\alpha''_0:-\alpha''-\beta''-\alpha''_0],\,
\\&\qquad
 [0:-\alpha''_0]_{R-(p+q+1)+1},\,
\underline{[0:-\alpha''-\beta''-\gamma''-\alpha''_0]}_{2},\,
  \underline{[0:-\alpha''-\beta''-\alpha''_0]_{r-1}},\,
\\&\qquad\underline
 [0:-\beta''-\gamma''-\alpha''_0]_{p-1},\,
 [0:-\alpha''-\gamma''-\alpha''_0]_{q-1},\,
 {[0:0]_{R-p}},\,
  \underline{[0:-\beta''-\gamma'']_{p}}\}
\allowdisplaybreaks\\
[\tilde A_{03}:\tilde A_{12}]
&=
\{\underline{[\alpha'_i:-\beta''-\gamma'']},\,
 \underline{[\alpha'_i:0]_{q+r-1}},\,
 \underline{[\beta_j+\gamma'_k:0]},\,
 \underline{[\alpha'_0:0]}_{R-p},\,
 \underline{[\alpha'_0:-\beta''-\gamma'']}_{p},\,
\\&\qquad
 [\alpha'_0:0]_{R-(p+q+r)+1},\,
 [\alpha'_0:-\alpha''-\beta''-\gamma'']_{2},\,
 [\alpha'_0:-\alpha''-\beta]_{r-1},\,
\\&\qquad
 [\alpha'_0:-\beta''-\gamma'']_{p-1},\,[\alpha'_0:-\alpha''-\gamma'']_{q-1}\},
\\
[\tilde A_{03}:\tilde A_{14}]&=
\{\underline{[\alpha'_i:\beta_j]},\,\underline{[\beta_j+\gamma'_k:\beta_j]},\,
\underline{[\alpha'_i:\alpha'_i+\gamma_k]},\,[\alpha'_0:\beta_j]_{p+r},\,
[\alpha'_0:\alpha'_i+\gamma_k],\,
\\&\qquad
\underline{[\alpha'_0:\alpha'_0+\gamma_k]}_{p+q},\,
[\alpha'_0:\beta_j+\alpha_i+\alpha'_0]\}
\allowdisplaybreaks\\
[\tilde A_{03}:\tilde A_{24}]
&=\{
 \underline{[\alpha'_i:\alpha'_i+\beta'_j]},\,
 \underline{[\alpha'_i:\gamma'_k]},\,
 \underline{[\beta_j+\gamma'_k:\gamma'_k]},\,
 \underline{[\alpha'_0:\gamma'_k]}_{p+q},\,
 [\alpha'_0:\alpha'_i+\beta'_j],\,
\\&\qquad
  \underline{[\alpha'_0:\alpha'_0+\beta'_j]}_{p+r},\,
[\alpha'_0:\alpha_i+\gamma'_k+\alpha'_0]\}
\\
[\tilde A_{04}:\tilde A_{12}]
&=\{
 \underline{[\alpha_i:-\beta''-\gamma'']},\,
 \underline{[\alpha_i:0]_{q+r-1}},\,
 \underline{[\beta'_j+\gamma_k:0]},\,
 \underline{[\alpha_0:0]}_{R-p},\,
 \underline{[\alpha_0:-\beta''-\gamma'']}_{p},\,
\\&\qquad
 [\alpha_0:0]_{R-(p+q+r)+1},\,
 [\alpha_0:-\alpha''-\beta''-\gamma'']_{2},\,
 [\alpha_0:-\alpha''-\beta]_{r-1},\,
\\&\qquad
 [\alpha_0:-\beta''-\gamma'']_{p-1},\,
 [\alpha_0:-\alpha''-\gamma'']_{q-1}\}.
\allowdisplaybreaks\\
[\tilde A_{04}:\tilde A_{13}]
 &=
\{\underline{[\alpha_i:\beta'_j]},\,
 \underline{[\alpha_i:\alpha_i+\gamma'_k]},\,
 \underline{[\beta'_j+\gamma_k:\beta'_j]},\,
 \underline{[\alpha_0:\beta'_j]}_{p+r},\,
 [\alpha_0:\alpha_i+\gamma'_k],\,
\\&\qquad
  \underline{[\alpha_0:\alpha_0+\gamma'_k]}_{p+q},\,
 [\alpha_0:\alpha_0+\alpha'_i+\beta'_j]\}
\allowdisplaybreaks\\
[\tilde A_{04}:\tilde A_{23}]
 &=\{
 \underline{[\alpha_i:\alpha_i+\beta_j]},\,
 \underline{[\alpha_i:\gamma_k]},\,
 \underline{[\beta'_j+\gamma_k:\gamma_k]},\,
 \underline{[\alpha_0:\gamma_k]}_{p+q},\,
 \underline{[\alpha_0:\alpha_i+\beta_j]},\,
\\&\qquad
 [\alpha_0:\beta_j+\alpha_0]_{p+r},\,
 [\alpha_0:\alpha_0+\gamma_k+\alpha'_i]\}
\allowdisplaybreaks\\
[\tilde A_{12}:\tilde A_{34}]
&=\{[-\beta''-\gamma:-\alpha''-\beta''-\gamma''-\alpha''_0],\,
[-\beta''-\gamma:-\beta''-\gamma''-\alpha_0-\alpha'_0]_{p-1},\,
\\&\qquad
\underline{[0:-\alpha''-\beta''-\gamma''-\alpha''_0]},\,
\underline{[0:-\alpha''-\gamma''-\alpha''_0]_{q-1}},\,
[0:-\alpha_0-\alpha'_0]_{R-(p+q+r)+1},\,
\\&\qquad
\underline{[0:-\alpha''-\beta''-\alpha''_0]_{r-1}},\,
[0:0]_{R-p},\,
\underline{[-\beta''-\gamma:-\beta''-\gamma'']_{p}},\,
\\&\qquad
[0:-\alpha_0-\alpha'_0]_{R-(p+q+r)+1},\,
[-\alpha''-\beta''-\gamma:-\alpha''-\beta''-\gamma''-\alpha''_0]_{2},\,
\\&\qquad
[-\alpha''-\beta: -\alpha''-\beta''-\alpha''_0]_{r-1},\,
\underline{[-\beta''-\gamma: -\alpha''-\beta''-\gamma''-\alpha''_0]}_{p-1},\,
\\&\qquad
[-\alpha''-\gamma: -\alpha''-\gamma''-\alpha''_0]_{q-1}\}
\allowdisplaybreaks\\
[\tilde A_{13}:\tilde A_{24}]
 &=\{[0:\alpha'_i+\beta'_j],\,[\beta'_j+\gamma'_k:\gamma'_k],\,\underline{[\beta'_j:\alpha'_i+\beta'_j]},\,\underline{[\alpha_i+\gamma'_k:\gamma'_k]},\,\underline{[\beta'_j:\gamma'_k]},\,
\\&\qquad
\underline{[\beta'_j:\alpha'_0+\beta'_j]}_{p+r},\,
[\alpha_i+\gamma'_k:\alpha_i+\gamma'_k+\alpha'_0],\,
\\&\qquad
\underline{[\alpha_0+\gamma'_k:\gamma'_k]}_{p+q},\,
[\alpha'_i+\alpha_0+\beta'_j:\alpha'_i+\beta'_j]\},\\
[\tilde A_{14}:\tilde A_{23}]&=\{
 \underline{[\beta_j:\alpha_i+\beta_j]},\,
 \underline{[\beta_j:\gamma_k]},\,
 \underline{[\alpha'_i+\gamma_k:\gamma_k]},\,
 \underline{[\alpha'_0+\gamma_k:\gamma_k]}_{p+q},\,
 [\alpha_i+\beta_j+\alpha'_0:\alpha_i+\beta_j],\,\\
&\qquad
 \underline{[\beta_j:\alpha_0+\beta_j]}_{p+r},\,
 [\alpha'_i+\gamma_k:\alpha'_i+\alpha_0+\gamma_k]\}.
\end{align*}
\normalsize
By the symmetry of the numbers $\{p,q,r\}$ of the points $\{x_0,x_1,x_2\}$, we have the list given in Theorem~\ref{thm:F1sim} by the induction with respect to $(p,q,r)$ starting from 
$(p,q,r)=(1,1,0)$.
\end{proof}

Note that the Riemann scheme is easily obtained from the list in Theorem~\ref{thm:F1sim}.
Moreover 
the list above already contains the spectral type of the boundary equation.
For example, the boundary equation for the diagonal line $x_0=x_1$ with respect to the characteristic exponent $0$ is
\begin{equation}\label{eq:boundary_equation}
    \frac{d u}{d x}
    =\frac{A_{013}|_{{\rm Ker}A_{01}}}{x}u+\frac{A_{012}|_{{\rm Ker}A_{01}}}{x-1}u
    =-\frac{A_{24}|_{{\rm Ker}A_{01}}}{x}u-\frac{A_{34}|_{{\rm Ker}A_{01}}}{x-1}u
\end{equation}
The residue matrix at infinity is
\begin{equation}
    -A_{013}|_{{\rm Ker}A_{01}}-A_{012}|_{{\rm Ker}A_{01}}=-A_{23}|_{{\rm Ker}A_{01}}=A_{014}|_{{\rm Ker}A_{01}}
\end{equation}
in view of the relation \eqref{eq:homogeneous_relation}.
The first three lines tell us the spectral type of the boundary equation.
It is
\begin{equation*}
    (p+q-1)^{r}1^{pq},(R-p-q-r+1)(p-1)(q-1)1,(p+q-1)^r1^{pq}
\end{equation*}
with the rigidity index $- 2 \left(p^2 q-p^2+p q^2-3 p q+2 p-q^2+2 q-2\right)$, which is strictly smaller than $2$ if and only if $p,q>1$.
Thus, the equation \eqref{eq:boundary_equation} may not be rigid.
However, the accessory parameters are suitably specialized so that it is globally analyzable in the following sense:
a basis of local solutions near the origin of \eqref{eq:boundary_equation} can be identified with a subspace of the space of local solutions at $(0,0^2)$ which is invariant by the monodromy along a loop $\gamma_{01}$ going around the divisor $\{x=y\}$ once with the base point $b$ near $(0,0^2)$.
The monodromy of \eqref{eq:boundary_equation} around $x=1$ can be computed from the monodromy of the system $\mathcal{M}^{\boldsymbol{\alpha},\boldsymbol{\beta},\boldsymbol{\gamma}}_{\boldsymbol{\alpha}',\boldsymbol{\beta}',\boldsymbol{\gamma}'}$ along a loop $\Gamma$ going around the divisors $\{ x=0\},\{y=1\},\{x=y\}$ all at once with the base point $b$.
Under the notation of \S\ref{sec:braid}, $\Gamma=\gamma_{01}*\gamma_{02}*\gamma_{12}$ where $*$ is the product in the fundamental group.
The discussion of \S\ref{sec:braid} shows that we can compute the monodromy along $\Gamma$.

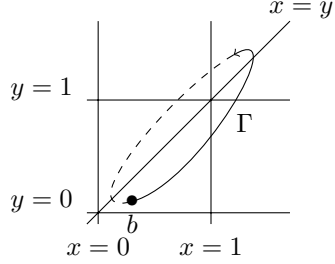
\begin{center}
\begin{tikzpicture}[scale=1.5]
\draw (0,-0.1)--(0,1.7) (1,-0.1)--(1,1.7) (-0.1,0)-- (1.7,0)  (-0.1,1)-- (1.7,1)  (-0.1,-0.1)--(1.7,1.7);
\node at (-0.5,0.1) {$y=0$};
\node at (-0.5,1.1) {$y=1$};
\node at (0,-0.3) {$x=0$};
\node at (1,-0.3) {$x=1$};
\node at (1.8,1.8) {$x=y$};
\draw[->] (0.3,0.1) to [out=15, in=30] (1.2,1.4);
\draw[dashed] (1.2,1.4) to [out=210, in=195] (0.3,0.1);
\node at (0.3,0.1) {$\bullet$};
\node at (0.3,-0.1){$b$};
\node at (1.3,0.8){$\Gamma$};
\end{tikzpicture}
\end{center}

On the other hand, the boundary equation with respect to the characteristic exponent $-\alpha''-\beta''$ is given as above after the application of the homogenized addition $\Ad\bigl((\frac{x_0-x_1}{x_2-x_3})^{\alpha''+\beta''}\bigr)$ which keeps the condition \eqref{eq:homog} and the equation has the rigid spectral type $1^r,1^r,(r-1)1$.

\begin{rem}\label{ewm:F1sim}
The calculation of the list for given integers $p$, $q$ and $r$ is 
supported by the Library \cite{Or} of the computer algebra system {\rm\texttt{Risa/Asir}} (\cite{noro1992risa}).
If $(p,q,r)=(4,3,2)$, it can be done as follows.
\begin{verbatim}
 [0] K432=os_md.mc2grs(0,["K",[4,3,2]]);
 [1] os_md.mc2grs(K432,"show"|dviout=-1");
 [2] os_md.mc2grs(K432,"show0"|dviout=-1");
 [3] os_md.mc2grs(K432,"get"|dviout=-1");
 [4] os_md.mc2grs(K432,"spct"|dviout=-1,div=5);
 [5] os_md.mc2grs(K432,"rest"|dviout=-1);
 [6] os_md.mc2grs(K432,"rest0"|dviout=-1);
\end{verbatim}
Here we get the list {\rm\texttt{K432}} in {\rm\texttt{[0]}}.  By {\rm\texttt{[1]}}, {\rm\texttt{[2]}}, {\rm\texttt{[3]}}, 
{\rm \texttt{[4]}}, {\rm \texttt{[5]}} and  {\rm \texttt{[6]}},  
we get the list as in Theorem~\ref{thm:F1sim}, the multiplicities of eigenvalues for 
15 normally crossing points, the Riemann scheme (cf.~\cite{Okz}), 
the spectral type of 10 residue matrices (cf.~\cite{Okz}), 
the Riemann scheme of the boundary values with respect to 10 boundaries and  
the spectral types of the boundary values with respect to 10 boundaries, 
respectively, in a \TeX\ format.
Here the boundary values are defined with respect to the characteristic exponents of the regular 
singularities along the boundaries (cf.~\cite{KO}).
The option {\rm\texttt{div=5}} indicate that the Riemann scheme is divided into two parts so that 
the first part contains 5 columns as is given in this section. 
If {\rm\texttt{dviout=-1}} is replaced by  {\rm\texttt{dviout=1}} in the above, 
the result is transformed into a \texttt{PDF} file and displayed.

\end{rem}

\begin{rem}\label{rem:F1_01}
The boundary equation  of 
$\mathcal M
 ^{\boldsymbol \alpha, \boldsymbol \beta, \boldsymbol \gamma}
 _{\boldsymbol \alpha',\boldsymbol \beta',\boldsymbol \gamma'}$
on $x=0$ 
with respect to the characteristic exponent $\alpha'_i$ 
equals
$\mathcal M
 ^{\boldsymbol \beta,\boldsymbol \gamma+\alpha'_i}
 _{\boldsymbol \beta',\boldsymbol \gamma'-\alpha'}$, 
which has
the characteristic exponent $q+r-\beta''-\gamma''-1$ with the free 
multiplicity at $y=1$ and the connection coefficients between 
the local solutions at $y=0$ and the local solution with the characteristic exponent
$q+r-\beta''-\gamma''-1$ at $y=1$.
They equal
the connection coefficients of $\mathcal M
 ^{\boldsymbol \alpha, \boldsymbol \beta, \boldsymbol \gamma}
 _{\boldsymbol \alpha',\boldsymbol \beta',\boldsymbol \gamma'}$
between the local solutions  at $(0^2,0)$ and the local solutions at $(0,1)$
with the characteristic exponents $(\alpha'_i,q+r-\beta''-\gamma''-1)$ and the corresponding
connection problem of $\mathcal M
 ^{\boldsymbol \alpha, \boldsymbol \beta, \boldsymbol \gamma}
 _{\boldsymbol \alpha',\boldsymbol \beta',\boldsymbol \gamma'}$ is solved.
By the symmetry of the system we can solve the similar connection problems for 6 edges in \textbf{Fig.1} in \S\ref{sec:connection}.
\end{rem}

The KZ-type system $\mathcal M$ defined in \eqref{eq:KZxy} corresponds to a $W(x,y)$-module.
The coefficient matrices $A_{i,j}$ are $R\times R$ matrices acting on $\C^R$ by left multiplication.
For any $i=1,\dots,R$, let $v_i=(0,\dots,1,\dots,0)\in\C^R$ be the $i$-th unit column vector and we regard it as an element of $\C(x,y)^R$.
Then, we define an action of partial differential operators on an element of the form $a(x,y)v_i$ by
\begin{align}\begin{split}
    \frac{\partial}{\partial x}\left(a(x,y)v_i\right)&:=\frac{\partial a}{\partial x}(x,y)v_i+a(x,y)\left( \frac{A_{01}}{x-y}+\frac{A_{02}}{x-1} + \frac{A_{03}}{x}\right)v_i,\\
    \frac{\partial}{\partial y}\left(a(x,y)v_i\right)&:=\frac{\partial a}{\partial y}(x,y)v_i+
   a(x,y)\left(
   \frac{A_{01}}{y-x}+\frac{A_{12}}{y-1}  + \frac{A_{13}}{y} \right)v_i
\end{split}\label{eq:action}\end{align}
for any $a(x,y)\in \C(x,y)$ and $i=1,\dots, R$.
We extend \eqref{eq:action} to an action of $W(x,y)$ on $\C(x,y)^R$ by linearity. 
By abuse of notation, we write $\mathcal M$ for the $W(x,y)$-module $\C(x,y)^R$.
On the other hand, we write $\mathcal M
 ^{\boldsymbol \alpha, \boldsymbol \beta, \boldsymbol \gamma}
 _{\boldsymbol \alpha',\boldsymbol \beta',\boldsymbol \gamma'}$ for the $W(x,y)$-module  $W(x,y)/(W(x,y)P_1+W(x,y)P_2+W(x,y)P_{12})$.

\begin{thm}
Suppose that none of 
$\alpha_i+\alpha'_{i'},\ \beta_j+\beta'_{j'},\ \gamma_k+\gamma'_{k'},\   \alpha_i+\beta_j+\gamma'_k,\ \alpha'_i+\beta'_j+\gamma_k$
is integral.
Then, 
$\mathcal M
 ^{\boldsymbol \alpha, \boldsymbol \beta, \boldsymbol \gamma}
 _{\boldsymbol \alpha',\boldsymbol \beta',\boldsymbol \gamma'}$ is isomorphic to $\mathcal M$ as a $W(x,y)$-module through the correspondence $\mathcal M
 ^{\boldsymbol \alpha, \boldsymbol \beta, \boldsymbol \gamma}
 _{\boldsymbol \alpha',\boldsymbol \beta',\boldsymbol \gamma'}\ni[1]\mapsto v_R\in\mathcal{M}$.    
\end{thm}

\begin{proof}
    The assumption ensures that $\mathcal M
 ^{\boldsymbol \alpha, \boldsymbol \beta, \boldsymbol \gamma}
 _{\boldsymbol \alpha',\boldsymbol \beta',\boldsymbol \gamma'}$ is irreducible (cf. Theorem \ref{thm:irreducibility}).
    We consider a partial derivative $\p_{gen}:=a\frac{\p}{\p x}+b\frac{\p}{\p y}$ for generic complex numbers $a,b$.
    Let $u(x,y)=(u_1(x,y),\dots,u_R(x,y))$ be a non-trivial solution vector to the system $\mathcal M$.
    By the construction of $\mathcal M$, the last entry $u_R(x,y)$ is a solution to $\mathcal M
 ^{\boldsymbol \alpha, \boldsymbol \beta, \boldsymbol \gamma}
 _{\boldsymbol \alpha',\boldsymbol \beta',\boldsymbol \gamma'}$ (see Remark \ref{rem:KZ}).
    There exist $R^2$ rational functions $a_{ij}(x,y)$ $(1\leq i\leq j\leq R)$ such that the following equations hold true:
    \begin{equation}
        \begin{cases}
            \p_{gen}u_1(x,y)&=a_{11}(x,y)u_1(x,y)+\cdots+a_{1R}(x,y)u_R(x,y)\\
            &\vdots\\
            \p_{gen}u_R(x,y)&=a_{R1}(x,y)u_1(x,y)+\cdots+a_{RR}(x,y)u_R(x,y).
        \end{cases}
    \end{equation}
    Suppose that $a_{R1}(x,y),\dots,a_{R(R-1)}(x,y)$ are all equal to zero.
    Then, the restriction of $u_R(x,y)$ to a line $\{ bx-ay+c=0\}$ for some generic $c$ is subject to a first order ODE.
    This contradicts the irreducibility of $\mathcal M
 ^{\boldsymbol \alpha, \boldsymbol \beta, \boldsymbol \gamma}
 _{\boldsymbol \alpha',\boldsymbol \beta',\boldsymbol \gamma'}$ as $u_R(x,y)$ is a solution to the system $\mathcal M
 ^{\boldsymbol \alpha, \boldsymbol \beta, \boldsymbol \gamma}
 _{\boldsymbol \alpha',\boldsymbol \beta',\boldsymbol \gamma'}$.
 Without loss of generality, we may assume that $a_{R(R-1)}(x,y)\neq 0$.
 We can express $u_{R-1}(x,y)$ as a linear combination of $u_1,u_2,\dots,u_{R-2},u_{R},\p_{gen}u_R$ and we obtain a system of equations of the following form
 \begin{equation}
    \begin{cases}
    \p_{gen}u_1(x,y)&=b_{11}(x,y)u_1(x,y)+\cdots+b_{1(R-2)}u_{R-2}(x,y)\\
    &\quad+b_{1(R-1)}(x,y)\p_{gen} u_{R}(x,y)+b_{1R}(x,y)u_R(x,y)\\
    &\vdots\\
    \p_{gen}u_{R-2}(x,y)&=b_{(R-2)1}(x,y)u_1(x,y)+\cdots+b_{(R-2)(R-2)}u_{R-2}(x,y)\\
    &\quad+b_{(R-2)(R-1)}(x,y)\p_{gen} u_{R}(x,y)+b_{(R-2)R}(x,y)u_R(x,y)\\
    \p_{gen}^2u_R(x,y)&=b_{(R-1)1}(x,y)u_1(x,y)+\cdots+b_{(R-1)(R-2)}(x,y)u_{R-2}(x,y)\\
    &\quad+b_{(R-1)(R-1)}(x,y)\p_{gen}u_R(x,y)+b_{(R-1)R}(x,y)u_R(x,y).
    \end{cases}
\end{equation}
Again by the irreducibility of $\mathcal M
 ^{\boldsymbol \alpha, \boldsymbol \beta, \boldsymbol \gamma}
 _{\boldsymbol \alpha',\boldsymbol \beta',\boldsymbol \gamma'}$, we may assume that $b_{(R-1)(R-2)}(x,y)\neq 0$.
Repeating this argument, we can conclude that $u_1(x,y),\dots, u_{R-1}(x,y)$ are linear combinations of  partial derivatives of $u_R(x,y)$.
It is readily seen that $\mathcal M$ is isomorphic to the $\C(x,y)$-span of $u_1(x,y),\dots,u_R(x,y)$ in the space of local meromorphic functions through the correspondence $v_i\mapsto u_i(x,y)$.
Now the correspondence $\mathcal M
 ^{\boldsymbol \alpha, \boldsymbol \beta, \boldsymbol \gamma}
 _{\boldsymbol \alpha',\boldsymbol \beta',\boldsymbol \gamma'}\ni [1]\mapsto u_R(x,y)\in\mathcal M$ is well-defined and surjective.
 Since $\dim_{\C(x,y)}\mathcal M
 ^{\boldsymbol \alpha, \boldsymbol \beta, \boldsymbol \gamma}
 _{\boldsymbol \alpha',\boldsymbol \beta',\boldsymbol \gamma'}=\dim_{\C(x,y)}\mathcal M=R$, we are done.
\end{proof}

\begin{rem} {\rm i) }\ 
We constructed a $W(x,y)$ homomorphism between $\mathcal M
 ^{\boldsymbol \alpha, \boldsymbol \beta, \boldsymbol \gamma}
 _{\boldsymbol \alpha',\boldsymbol \beta',\boldsymbol \gamma'}$ and $\mathcal M$, which gives an isomorphism 
when the parameters are generic.  Since $\mathcal M$ and $\mathcal M
 ^{\boldsymbol \alpha, \boldsymbol \beta, \boldsymbol \gamma}
 _{\boldsymbol \alpha',\boldsymbol \beta',\boldsymbol \gamma'}$ have the same rank, 
it follows from \cite[Lemma~2.1]{Oir} that the condition for the irreducibility of $\mathcal M$ 
equals that of $\mathcal M
 ^{\boldsymbol \alpha, \boldsymbol \beta, \boldsymbol \gamma}
 _{\boldsymbol \alpha',\boldsymbol \beta',\boldsymbol \gamma'}$.
 
 {\rm ii)\ }
 Suppose, for example, $\mathcal M
 ^{\boldsymbol \alpha, \boldsymbol \beta, \boldsymbol \gamma}
 _{\boldsymbol \alpha',\boldsymbol \beta',\boldsymbol \gamma'}$ is reducible when $\alpha_{i}+\alpha'_{i'}=c$ with a suitable $c\in\mathbb C$. Then 
 it also follows from \cite[Lemma~2.1]{Oir} and Proposition~\ref{prop:shift} that the system 
 is reducible when $\alpha_{i}+\alpha'_{i'}-c\in\mathbb Z$.
\end{rem}

\medskip
Owing to Theorem~\ref{thm:F1sim}, we get the Riemann scheme and the spectral type of the corresponding KZ-type equation $\mathcal M$  (cf.~\cite{Oi}) as follows.

\begin{align}
\begin{split}
&\left\{\begin{matrix}
A_{01}&A_{02}&A_{03}&A_{04}&A_{12}\\
[0]_{pq+(p+q-1)r} & [0]_{pr+(p+r-1)q} & [\alpha'_i]_{q+r} & [\alpha_i]_{q+r} & [0]_{qr+(q+r-1)p} \\
[-\alpha''-\beta'']_r & [-\alpha''-\gamma'']_q & \beta_j+\gamma'_k &  \beta'_j+\gamma_k& [-\beta''-\gamma'']_p 
\end{matrix}\right.\\
&\qquad\left.\begin{matrix}
A_{13}&A_{23}&A_{14}&A_{24}&A_{34}\\
 [\beta'_j]_{p+r} & [\gamma_k]_{p+q} & [\beta_j]_{p+r} & [\gamma'_k]_{p+q} & [0]_{pq+qr+rp-(p+q+r)+1}\\
\alpha_i+\gamma'_k & \alpha_i+\beta_j & \alpha_i'+\gamma_k & \alpha'_i+\beta'_j& [-\alpha''-\beta''-\gamma'']_2\\
 & & & & [-\alpha'' - \beta'']_{r-1}\\
 & & & &[-\beta''-\gamma'']_{p-1}\\
 & & & & [-\alpha''-\gamma'']_{q-1}\\
\end{matrix}\right\}\\
\end{split}
\end{align}

\smallskip
\begin{tabular}{|r|rrrrr|} 
\multicolumn{6}{c}{Spectral type of the KZ-type equation satisfied by $F_{p,q,r}(x,y)$}\\
\hline
& $x_0=x$& $x_1=y$& $x_2=1$& $x_3=0$& $x_4=\infty$\\ \hline
$x_0$& & $(R-r)r$& $(R-q)q$& $(q+r)^p1^{qr}$& $(q+r)^p1^{qr}$\\
$x_1$& $(R-r)r$& &$(R-p)p$ & $(p+r)^q1^{pr}$& $(p+r)^q1^{pr}$\\
$x_2$& $(R-q)q$& $(R-p)p $& & $(p+q)^r1^{pq}$ & $(p+q)^r1^{pq}$\\
$x_3$& $(q+r)^p1^{qr}$&$(p+r)^q1^{pr}$& $(p+q)^r1^{pq} $ & & $S$\\
$x_4$& $(q+r)^p1^{qr}$& $(p+r)^q1^{pr}$ & $(p+q)^r1^{pq}$& $S$& \\ \hline
\end{tabular}

rank $:R:=pq+qr+rp,\ S:=(R-p-q-r+1)(p-1)(q-1)(r-1)2$
\medskip

This table shows that, for example, the spectral type of the residue matrix $A_{0,3}$ is $(q+r)^p1^{qr}$, which represents the partition of $R: 
\overbrace{q+r,\cdots,q+r}^p,\overbrace{1,\dots,1}^{qr}$,
namely, $A_{0,3}$ has $p$ eigenvalues with multiplicity $q+r$ and $qr$ multiplicity free eigenvalues.
The symmetry $S_{x,y,0}\times S_{0,\infty}$ acts on this table.
For example, the spectral type at $x_1=x_4$ is obtained from that at $x_0=x_4$ by exchanging $p$ and $q$.
This symmetry corresponds to $x_0\leftrightarrow x_1$.

The next table shows the spectral type of the equation when $p=q=r=2$ together with the index of rigidity for each variable.

\smallskip
\centerline{
\begin{tabular}{|r|rrrrr|r|} 
\multicolumn{7}{c}{$p=q=r=2\ ($rank:$\,12$)}\\
\hline
& $x_0$& $x_1$& $x_2$& $x_3$& $x_4$& idx\\ \hline
$x_0$& & 2& (10)2& $4^21^4$ & $4^21^4$ &$-8$\\
$x_1$& (10)2& & (10)2& $4^21^4$& $4^21^4$ & $-8$\\
$x_2$& (10)2& (10)2& & $4^21^4$& $4^21^4$& $-8$\\
$x_3$& $4^21^4$ & $4^21^4$ & $4^21^4$& & $721^3$ & $-124$\\
$x_4$& $4^21^4$& $4^21^4$& $4^21^4$& $721^3$& & $-124$\\ \hline
\end{tabular}}

\bigskip
Focusing on one variable, for example $x_0$, we have an ordinary differential equation \eqref{eq:ODE} of rank $R$ which has $n+1=4$ singular points including $\infty$ in $\mathbb P^1$.  
 The rigidity index of the equation equals the sum of squares of multiplicities of the eigenvalues of the residue matrices minus $(n-1)R^2$. 
It is
\begin{align}
\begin{split}
\idx_{x_0}\mathcal M&=(R-q)^2+q^2+(R-r)^2+r^2+2(p(q+r)^2+qr)-2R^2\\
&=2-2(q-1)(r-1)(q+r+1).
\end{split}
\end{align}
Then the number $2-
\idx_{x_0}\mathcal M$ coincides with the maximal possible number of the accessory parameters of the equation which has the same local structure  as the equation \eqref{eq:ODE}.
An ordinary differential equation is rigid if and only if the rigidity index equals 2.
Then the KZ-type equation satisfied by $F_{p,q,r}$ has a rigid variable if and only if
\begin{equation}
 (p-1)(q-1)(r-1)=0.
\end{equation}

If the KZ-type equation has a rigid variable, the condition of the irreducibility of the monodromy group of its solution space is obtained by \cite{Oir} and a precise analysis is possible for the reducibility.
In general,  
the necessary and sufficient condition for the irreducibility is given by Theorem~\ref{thm:irreducibility}.

The transformation given by additions and middle convolutions are invertible. 
$d(\mathbf m_x)$ gives the maximal number of the change of the rank of the equation by one middle convolution with additions for the variable $x$. Namely, the rank decreases by $d(\mathbf m_x)$, which equals the sum of maximal multiplicities of eigenvalues of residue matrices minus $(n-1)R$ (see $d_{max}(\mathbf m)$ in \cite[(5.36)]{Ow}).  
\begin{equation}
    d(\mathbf m_x)\,=(R-r)+(R-q)+(q+r)+(q+r)-2R=q+r.
\end{equation}
The number $d(\mathbf m_x)$ is always positive if $\mathbf m_x$ is a rigid spectral type. Moreover these transformations do not change the rigidity index. Hence there is an invertible algorithm due to Katz, Dettweiler-Reiter and Haraoka transforming the rigid system to the trivial system (cf.~\cite{DR, Ha, katz1996rigid}).
In fact, the change of the spectral type is described by decreasing the maximal multiplicities of the residue matrices by $d(\mathbf m_x)$.
\[\mathbf m_x:\underline{(R-r)}r,\underline{(R-q)}r,\underline{(q+r)^p}1^{qr},\underline{(q+r)^p}1^{qr}\]
In the above, the maximal multiplicities are indicated by underlines and they decrease by $d(\mathbf m_x)$, which leads to the change $(p,q,r)\mapsto(p-1,q,r)$ since $R=(q+r)p+qr$. We can repeat this algorithm until $(p,q,r)$ changes into $(1,p,q)$.
Note that $\idx_x(\mathcal M)$ does not depend on $p$.
If $p=1$, the equation is rigid with respect to $y$ variable and we have the above reduction until we get the trivial equation.  Inverting these algorithm, we get the original equation $\mathbf m$ from the trivial equation. 

Owing to \cite[Theorem~7.1]{Okz}, the above procedure can be constructed only by the data describing the simultaneous eigenvalues with their multiplicities for commuting pairs of residue matrices of the equation.

\subsection{Appell's $F_1$}\label{sec:AppellF1}
Appell's $F_1$ series is defined by
\begin{align*}
  &F_1(a;b,c;d;x,y):=F^{1,1,1}_{1,1,1}
 \left(\begin{smallmatrix} \alpha & \beta& \gamma\\ 0 & 0 &\gamma'\end{smallmatrix};x,y\right)\\
 &\qquad \text{with }
  \alpha=b,\ \beta=c,\ \gamma= a,\ \gamma'=1-d.
\end{align*}
The Riemann scheme and the spectral type of the corresponding KZ-type equation  
are
\begin{align}\begin{split}
 &\left\{\begin{matrix}
A_{01}&A_{02}&A_{03}&A_{04}&A_{12}\\
[0]_2 & [0]_2 & [0]_2 & [\alpha]_2 & [0]_2 \\
-\alpha-\beta&-\alpha-\gamma-\gamma' & \beta+\gamma'&  \gamma & -\beta-\gamma-\gamma'  
\end{matrix}\right.\\
&\qquad\left.\begin{matrix}
A_{13}&A_{23}&A_{14}&A_{24}&A_{34}\\
 [0]_2 & [\gamma]_{2} & [\beta]_{2} & [\gamma']_{2} &[-\alpha-\beta-\gamma-\gamma']_2\\
\alpha+\gamma' & \alpha+\beta & \gamma & 0 & 0
\end{matrix}\right\},
\end{split}
\end{align}
\centerline{\begin{tabular}{|r|rrrrr|r|}
\multicolumn{7}{c}{$F_1:p=q=r=1\ ($rank:$\,3$)}
\\\hline
& $x_0$& $x_1$& $x_2$& $x_3$& $x_4$& idx\\ \hline
$x_0$& & 21& 21& 21& 21& $2$\\
$x_1$& 21& & 21& 21& 21& $2$\\
$x_2$& 21& 21& & 21& 21& $2$\\
$x_3$& 21& 21& 21& & 21& $2$\\
$x_4$& 21& 21& 21& 21& & $2$\\ \hline
\end{tabular}}

\bigskip
It follows from Remark.~\ref{rem:red0} that 
$F_1(\gamma;\alpha,\beta;1-\gamma';x,y)$ is a solution to
\begin{align*}
\bar{\mathcal M}^{\boldsymbol\alpha,\boldsymbol\beta,\boldsymbol\gamma}
 _{0,0,\boldsymbol\gamma'}:
\begin{cases}
   \p_x(\vartheta_x+\vartheta_y-\gamma')u&=
   (\vartheta_x+\alpha)(\vartheta_x+\vartheta_y+\gamma)u,\\
   \p_y(\vartheta_x+\vartheta_y-\gamma')u&=
   (\vartheta_y+\beta)(\vartheta_x+\vartheta_y+\gamma)u,\\
   (\vartheta_x+\alpha)\p_yu&=(\vartheta_y+\beta)\p_xu.
\end{cases}
\end{align*}
Put $(X,Y)=(1-x,1-y)$. Then  
\[(\vartheta_X+\alpha)\p_Yu=(\vartheta_Y+\beta)\p_Xu\]
and the first line of the equations $\bar{\mathcal M}^{\boldsymbol\alpha,\boldsymbol\beta,\boldsymbol\gamma}
 _{0,0,\boldsymbol\gamma'}$ with this equation says
\begin{align*}
0&=-\p_X(\vartheta_X+\vartheta_Y-\p_X-\p_Y-\gamma')u\\
  &\quad{}+(\vartheta_X-\p_X+\alpha)(\vartheta_X+\vartheta_Y-\p_X-\p_Y+\gamma)u\\
 &=(\vartheta_X+\alpha+\gamma+\gamma')\p_Xu+(\vartheta_X+\alpha)\p_Yu
   -(\vartheta_X+\alpha)(\vartheta_X+\vartheta_Y+\gamma)u\\
 &=\p_X(\vartheta_X+\alpha+\beta+\gamma+\gamma'-1)u -
   (\vartheta_X+\alpha)(\vartheta_X+\vartheta_Y+\gamma)u.
\end{align*}
Hence we have
\begin{equation*}
T_{(x,y)\to(1-x,1-y)}{\mathcal M}^{\alpha,\beta,\gamma}
 _{0,0,\gamma'}={\mathcal M}^{\alpha,\beta,\gamma}
 _{0,0,1-\alpha-\beta-\gamma-\gamma'}.
\end{equation*}
and 
\begin{align*}
&{\mathcal M}^{\alpha,\beta,\gamma}
 _{\alpha',\beta',\gamma'}
    \xrightarrow{\Ad(x^{-\alpha'}y^{-\beta'})}
{\mathcal M}^{\alpha+\alpha',\beta+\beta',\gamma+\alpha'+\beta'}
 _{0,0,\gamma'-\alpha'-\beta'}
  \xrightarrow{T_{(x,y)\to(1-x,1-y)}}\\
&\quad
{\mathcal M}^{\alpha+\alpha',\beta+\beta',\gamma+\alpha'+\beta'}
 _{0,0,1-\alpha-\alpha'-\beta-\beta'-\gamma-\gamma'}
 \xrightarrow{\Ad(x^{\alpha'}y^{\beta'})}
{\mathcal M}^{\alpha,\beta,\gamma}
 _{\alpha',\beta',1-\alpha-\beta-\gamma-\gamma'}.
\end{align*}

\begin{prop}\label{prop:F101}
One has
\[
T^{F_1}_{0\leftrightarrow1}{\mathcal M}^{\alpha,\beta,\gamma}
 _{\alpha',\beta',\gamma'}={\mathcal M}^{\alpha,\beta,\gamma}
 _{\alpha',\beta',1-\alpha-\beta-\gamma-\gamma'}
 \]
by putting $T^{F_1}_{0\leftrightarrow1}=
 \Ad(x^{\alpha'}y^{\beta'})\circ
 T_{(x,y)\to(1-x,1-y)}\circ 
 \Ad(x^{-\alpha'}y^{-\beta'})$. 
\end{prop}
\begin{rem} [Symmetry of $F_1$]\ {\rm i)}\ 
If $u(x.y)$ is a solution to ${\mathcal M}^{\alpha,\beta,\gamma}
 _{\alpha',\beta',\gamma'}$, then  
$(\frac{x}{1-x})^{\alpha'}(\frac{y}{1-y})^{\beta'}u(1-x,1-y)$ is a solution to 
$T^{F_1}_{0\leftrightarrow1}{\mathcal M}^{\alpha,\beta,\gamma}
 _{\alpha',\beta',\gamma'}$.

{\rm ii)} \ In the above, 
we may replace $\gamma$ and $\gamma'$ by operators which commute with $x,\,y,\,\p_x$ and $\p_y$.

{\rm iii)} \ 
The Riemann scheme of the ordinary differential equation of the variable $x$ satisfied 
by $F^{1,1,1}_{1,1,1}
 \left(\begin{smallmatrix} \alpha & \beta& \gamma\\ 0 & 0 &\gamma'\end{smallmatrix};x,y\right)$
is
\[
\begin{Bmatrix}
x=y & x=1 & x=0 & x=\infty\\
[0]_{(2)} & [0]_{(2)} & [0]_{(2)} & [\alpha]_{(2)}\\
-\alpha-\beta+1&-\alpha-\gamma-\gamma'+1 & \beta+\gamma'&  \gamma 
\end{Bmatrix},
\]
which is rigid.
Hence, we can recover Proposition \ref{prop:F101} by using this.
\end{rem}

As is shown in \S\ref{sec:KZ}, the group $S_{\{x,y,0,1,\infty\}} \simeq 
S_{\{x_0,x_1,x_2,x_3,x_4\}}\simeq S_{\{0,1,2,3,4\}}\simeq S_5$ 
acts on KZ-type equation \eqref{eq:KZ} with $n=2$ as a group of coordinate transformations, 
which
corresponds to the group of coordinate transformations under the $(x,y)$ coordinate 
given in \eqref{eq:dynkin}.
The transformations of $\mathcal M^{\alpha,\beta,\gamma}_{\alpha',\beta,\gamma'}$
,$T_{(x,y)\to(y,x)}$, $T_{(x,y)\to(\tfrac yx,\tfrac 1y)}$,
$T_{(x,y)\to(\tfrac 1x,\tfrac 1y)}$ and $T^{F_1}_{0\leftrightarrow1}$
correspond to the transposition $(0,1)$, $(1,2)$, $(3,4)$ and $(2,3)$ in $S_{\{0,1,2,3,4\}}$,
respectively. They generate $S_{\{0,1,2,3,4\}}$. 

In this section, $\tilde X$ denotes the space obtained by resolving 
the singularities 
$\mathrm{Sing}(\mathcal M^{\alpha,\beta,\gamma}_{\alpha',\beta,\gamma'})\subset
\mathbb P^1\times\mathbb P^1$ at 
$(0,0)$, $(1,1)$ and $(\infty,\infty)$.
In  $\tilde X$ there are 10 singular lines parametrized by $\{i,j\}$ $(0\le i<j\le 4)$
which corresponds to $x_i=x_j$. Here $\{2,4\}$ corresponds to the exceptional fiber
at $(0,0)$. There are 
15 normally crossing singular points parametrized by
$\bigl\{\{i,j\},\{k.\ell\}\bigr\}$ with $\{i,j,k,\ell\}\subset\{0,1,2,3,4\}$ and
$\#\{i,j,k,\ell\}=4$.
In each singular line $\{i,j\}$ there are three normally crossing singular points 
by the intersection of lines $\{k,\ell\}$ with $\{k,\ell\}\subset \{0,1,2,3,4\}\setminus 
\{i,j\}$. 

The connection formula from $(0^2,0)$ to $(0,\infty)$ given 
in Theorem~\ref{thm:conenct2} corresponds to the ordered pair
$\bigl(\{0,3\},\{2,4\}\},\{\{0,3\},\{1,4\}\} \bigr)$.
The group $S_{\{0,1,2,3,4\}}$ simply and transitively acts on the
set of singular lines with ordered two singular points in the lines.
Hence applying suitable coordinate transformations corresponding to elements
in $S_{\{0,1,2,3,4\}}$, 
we get the connection formula between normally crossing points
belonging to any singular line in $\tilde X$.

The subgroup of $S_{\{0,1,2,3,4\}}$ which keeps the exceptional line at $(0,0)$ corresponding
to $x_2=x_4$ is generated by transpositions $(0,1)$, $(1,3)$ and $(2,4)$, 
which correspond to $(x,y)\mapsto (y,x)$, 
$(\frac {y-x}{y-1},\frac y{y-1})$ and  
$(\frac x{x-1},\frac y{y-1})$, respectively. 
The subgroup is isomorphic to $S_3\times S_2$, 
Applying this subgroup to 
$F
 \Bigl(\begin{smallmatrix}\alpha&\beta&\gamma\\0&0&\gamma'
\end{smallmatrix};x,y\bigr)$
and $F_1(a;b,c;d;x,y)$, we obtain the identity between 
transformed 12 functions as in the case of Kummer's formula of Gauss hypergeometric
function.

For example, 
since $(2\ 4)=(3\ 4)(2\ 3)(3\ 4)$, the corresponding transformation of the system is given by
$T_{(x,y)\to(\frac1x,\frac1y)}\circ T^{F_1}_{0\leftrightarrow 1}\circ T_{(x,y)\to(\frac1x,\frac1y)}$ and therefore it follows from Theorem~\ref{prop:transformations} and 
Proposition~\ref{prop:F101} that
$\mathcal M^{\alpha,\beta,\gamma}_{\alpha',\gamma',\beta'}$ and its solution $u(x,y)$ 
are transformed into 
$\mathcal M^{\alpha,\beta,1-\alpha'-\beta'-\gamma-\gamma'}_{\alpha',\gamma',\beta'}$
and $(1-x)^{-\alpha}(1-y)^{-\beta}u(\tfrac x{x-1},\tfrac y{y-1})$.
Hence we have
\begin{equation}
 F\Bigl(
\begin{smallmatrix}
 \alpha & \beta & \gamma\\
 0 & 0 & \gamma'
\end{smallmatrix}
;x,y\Bigr)=(1-x)^{-\alpha}(1-y)^{-\beta}
 F\Bigl(
\begin{smallmatrix}
 \alpha & \beta & 1-\gamma-\gamma'\\
 0 & 0 & \gamma'
\end{smallmatrix}
;\frac x{x-1},\frac y{y-1}\Bigr).
\end{equation}
Using the transformation $T_{(x,y)\to(\frac xy,\frac 1y)}\circ T^{F_1}_{0\leftrightarrow1}\circ T_{(x,y)\to(\frac xy,\frac1y)}$ corresponding to $(1\ 3)=(1\ 2)(2\ 3)(1\ 2)$, we have 
\begin{equation}
 F\Bigl(
\begin{smallmatrix}
 \alpha & \beta & \gamma\\
 0 & 0 & \gamma'
\end{smallmatrix}
;x,y\Bigr)=(1-y)^{-\alpha}
F\Bigl(
\begin{smallmatrix}
 \alpha & 1-\alpha-
 \beta-\gamma'& \gamma\\
 0 & 0 & \gamma'
\end{smallmatrix}
;\frac{y-x}{y-1},\frac y{y-1}\Bigr).
\end{equation}

The subgroup of $S_{\{0,1,2,3,4\}}$ fixing the point $(0,\infty)$ in $\tilde X$,
which corresponds to $\bigl\{\{0,3\},\{1,4\}\bigr\}$, is generated by
$(0\ 3)$, $(1\ 4)$ and $(0\ 1)(3\ 4)$.  
The subgroup is
isomorphic to $W\!_{B_2}$, the reflection group of the root space of type $B_2$.
Hence we have identities among 8 functions obtained by transforming 
$G_2(a,a';b,b';x,y)$ by the corresponding coordinate transformations
(the formula is in \cite[\S1.3]{Mi2}).

\subsection{Example: Generalization of Appell's $F_2$ and $F_3$}\label{sec:F2}
In this section we study the hypergeometric functions under the condition (F2) or (F3). 
We put
\begin{align*}
I_{p,q}\left(\begin{smallmatrix}
\boldsymbol \alpha&\boldsymbol \beta&\gamma\\
 \boldsymbol \alpha'&\boldsymbol \beta'
\end{smallmatrix};x,y\right)&:=
F^{p-1,q-1,1}_{p,q,0}\left(
\begin{smallmatrix}
\boldsymbol \alpha&\boldsymbol \beta&\gamma\\
 \boldsymbol \alpha'&\boldsymbol \beta'
\end{smallmatrix};x,y\right),\\
J_{p,q}\left(\begin{smallmatrix}
\boldsymbol \alpha&\boldsymbol \beta\\ 
 \boldsymbol \alpha'&\boldsymbol \beta'&\gamma
\end{smallmatrix};x,y\right)&:=
F^{p,q,0}_{p-1,q-1,1}\left(
\begin{smallmatrix}
\boldsymbol \alpha&\boldsymbol \beta&\\
 \boldsymbol \alpha'&\boldsymbol \beta'&\gamma
\end{smallmatrix};x,y\right). 
\end{align*}
Note that the transformation $(x,y)\mapsto (\frac1x,\frac1y)$
induces 
$\mathcal M^{\boldsymbol \alpha,\boldsymbol \beta,\boldsymbol \gamma}_{\boldsymbol \alpha',\boldsymbol \beta',\boldsymbol \gamma'}\mapsto
\mathcal M_{\boldsymbol \alpha,\boldsymbol \beta,\boldsymbol \gamma}^{\boldsymbol \alpha',\boldsymbol \beta',\boldsymbol \gamma'}$.

A KZ-type equation satisfied by 
$I_{p,q}\left(\begin{smallmatrix}
\boldsymbol\alpha&\boldsymbol\beta&\gamma\\\boldsymbol\alpha',0 &\boldsymbol\beta',0&\emptyset 
\end{smallmatrix}
;x,1-y\right)$
is constructed by transformations of KZ-type equations 
according to the integral transformation of $(y-x)^{-\gamma_{q-1}}$ in a neighborhood of $(0,1)$.
Then the residue matrices $A_{i,j}$ of the KZ-type equation have the following property.

\begin{thm}[F2]\label{thm:F2sim}
The simultaneous eigenvalues with their multiplicities of the pair of commuting residue matrices 
at the 15 normally crossing points are give by \begin{align*}
[A_{01}:A_{23}]&=\left\{[-\alpha''-\beta''-\gamma:\gamma],\,[0:\gamma]_{p+q-1},\,[0:\alpha_i+\beta_j]\right\},\allowdisplaybreaks\\
[A_{01}:A_{24}]&=\left\{[-\alpha''-\beta''-\gamma:-\beta''-\gamma],\,[0:-\beta''-\gamma]_{p-1},\,[0:\alpha'_i]_{q-1}\right\},\allowdisplaybreaks\\
[A_{01}:A_{34}]&=\left\{[-\alpha''-\beta''-\gamma:-\alpha''-\gamma],\,[0:\beta'_j]_{p-1},\,[0:-\alpha''-\gamma]_{q-1}\right\},\allowdisplaybreaks\\
[A_{02}:A_{13}]&=\left\{[0:\alpha_i-\beta''-\gamma],\,[0:0]_{pq-p-q+2},\,[\beta_j-\alpha''-\gamma:0]\right\},\allowdisplaybreaks\\
[A_{02}:A_{14}]&=\left\{[0:\alpha'_i+\gamma],\,[0:\beta_j]_{p-1},\,[\beta_j-\alpha''-\gamma:\beta_j]\right\},\allowdisplaybreaks\\
[A_{02}:A_{34}]&=\left\{[0:-\alpha''-\gamma],\,[0:\beta'_j]_{p-1},\,[\beta_j-\alpha''-\gamma:-\alpha''-\gamma]\right\},\allowdisplaybreaks\\
[A_{03}:A_{12}]&=\left\{[\alpha'_i:\beta'_j]\right\},\allowdisplaybreaks\\
[A_{03}:A_{14}]&=\left\{[\alpha'_i:\alpha'_i+\gamma],\,[\alpha'_i:\beta_j]\right\},\allowdisplaybreaks\\
[A_{03}:A_{24}]&=\left\{[\alpha'_i:-\beta''-\gamma],\,[\alpha'_i:\alpha'_i]_{q-1}\right\},\allowdisplaybreaks\\
[A_{04}:A_{12}]&=\left\{[[\alpha_i:\beta'_j],\,[\beta'_j+\gamma:\beta'_j]\right\},\allowdisplaybreaks\\
[A_{04}:A_{13}]&=\left\{[\alpha_i:\alpha_i-\beta''-\gamma],\,[\alpha_i:0]_{q-1},\,[\beta'_j+\gamma:0]\right\},\allowdisplaybreaks\\
[A_{04}:A_{23}]&=\left\{[\alpha_i:\gamma],\,[\beta'_j+\gamma:\gamma],\,[\alpha_i:\alpha_i+\beta_j]\right\},\allowdisplaybreaks\\
[A_{12}:A_{34}]&=\left\{[\beta'_j:-\alpha''-\gamma],\,[\beta'_j:\beta'_j]_{p-1}\right\},\allowdisplaybreaks\\
[A_{13}:A_{24}]&=\left\{[\alpha_i-\beta''-\gamma:-\beta''-\gamma],\,[0:-\beta''-\gamma],\,[0:\alpha'_i]_{q-1}\right\},\allowdisplaybreaks\\
[A_{14}:A_{23}]&=\left\{[\alpha'_i+\gamma:\gamma],\,[\beta_j:\gamma],\,[\beta_j:\alpha_i+\beta_j]\right\}.%
\end{align*}
\end{thm}

The proof is omitted since it is an induction on $(p,q)$ 
which is similar to the proof of Theorem~\ref{thm:F1sim}. 
We only remark that in one step of the induction we increase $p$ by one by applying
Theorem~\ref{thm:TKz} to our setting with $\dim\mathcal L_1=pq-1$ and $\dim\mathcal L_2=pq-q+1$.

The following Riemann scheme and spectral type of the corresponding KZ-type equation
are obtained from Theorem~\ref{thm:F2sim}.

\bigskip
\noindent
\begin{tabular}{|r|rrrrr|c|} 
\multicolumn{7}{c}{Spectral type of $I_{p,q}$ (rank: $pq$)}\\
\hline
& $x_0=x$& $x_1=1-y$& $x_2=1$& $x_3=0$& $x_4=\infty$& idx\\ \hline
$x_0$& & $(pq-1)1$& $S_2$& $q^p$& $q^{p-1}1^q$& $2$\\
$x_1$& $(pq-1)1$& & $p^q$& $S_1$ & $p^{q-1}1^p$& $2$\\
$x_2$& $S_2$& $p^q$& & $S$ & $(q-1)^pp$ & \\
$x_3$& $q^p$&$S_1$& $S$ & & $(p-1)^qq$& \\
$x_4$& $q^{p-1}1^q$& $q^{p-1}1^q$ & $(q-1)^pp$& $(p-1)^qq$& & \\ \hline
\multicolumn{7}{c}{$S=(p+q-1)1^{(p-1)(q-1)},\ S_1=(pq-p+1)1^{p-1},\ 
S_2=(pq-q+1)1^{q-1}$}
\end{tabular}
\begin{equation*}\label{eq:RSF2ext}\begin{split}
&\left\{\begin{matrix}
A_{01}&A_{02}&A_{03}&A_{04}&A_{12}\\
[0]_{pq-1} & [0]_{pq-q+1} & [\alpha'_i]_q & [\alpha_i]_q & [\beta'_j]_p \\
-\alpha''-\beta''-\gamma& 
\beta_j  -\alpha''-\gamma&  &\beta'_j+\gamma   & 
\end{matrix}\right.\\
&\quad\left.\begin{matrix}
A_{13}&A_{23}&A_{14}&A_{24}&A_{34}\\
 [0]_{pq-p+1} & [\gamma]_{p+q-1} & [\beta_j]_p & [\alpha'_i]_{q-1} & [\beta'_j]_{p-1}\\
 \alpha_i-\beta''-\gamma& \alpha_i+\beta_j & \alpha'_i+\gamma & [-\beta'' -\gamma]_p & [-\alpha''-\gamma]_q\\
\end{matrix}\right\}
\end{split}\end{equation*}
\centerline{\begin{tabular}{|r|rrrrr|r|}
\multicolumn{7}{c}{$F_2:p=q=2\,($rank$:4)$}\\\hline
& $x_0$& $x_1$& $x_2$& $x_3$& $x_4$& idx\\ \hline
$x_0$& & 31& 31& 22& 211& $2$\\
$x_1$& 31& & 22& 31& 211& $2$\\
$x_2$& 31& 22& & 31& 211& $2$\\
$x_3$& 22& 31& 31& & 211& $2$\\
$x_4$& 211& 211& 211& 211& & $- 8$\\ \hline
\end{tabular}}

\medskip
\centerline{\begin{tabular}{|r|rrrrr|r|} 
\multicolumn{7}{c}{$I_{4,3}:p=4,\,q=3\,($rank$:12)$}\\\hline
& \!{\scriptsize$x_0=x$}\!& \!\!\!{\scriptsize$x_1=1-y$\!\!\!\!}
& \!{\scriptsize$x_2=1$}\!& \!\!\!{\scriptsize$x_3=0$}\!& 
\scriptsize{$x_4=\infty$}& idx\\ \hline
$x_0$& & $(11)1$& $(10)1^2$ & $3^4$& $3^31^3$& $2$\\
$x_1$& $(11)1$& & $4^3$ & $91^3$ & $4^21^3$& $2$\\
$x_2$& $(10)1^2$& $4^3$ & & $61^6$ &$42^4$& $-64$ \\
$x_3$&$3^4$ & $91^3$& $61^6$ & & $3^4$ &  $-90$ \\
$x_4$& $3^31^3$ & $4^21^4$ & $42^4$ & $3^4$& & $-154$\\ \hline
\end{tabular}}

\medskip

We note that the the transformation $(x,y)\mapsto (y,x)$ of the coordinate of
$I_{p,q}\left(\begin{smallmatrix}
\boldsymbol \alpha&\boldsymbol \beta&\gamma\\
 \boldsymbol \alpha'&\boldsymbol \beta'
\end{smallmatrix};x,y\right)
$ corresponds to the element $(0\ 1)(2\ 3)\in S_{0,1,2,3,4}$
because of the 
correspondence $(x_0,x_1,x_2,x_3,x_4)=(x,1-y,1,0,\infty)$.

The ordinary differential equation for the variable $x_0$ is rigid and its rigid spectral type with direct decompositions are
\begin{align*}
 &(pq-1)1,((p-1)q+1)1^{q-1},q^p,q^{p-1}1^q\\
 &\quad=(p(q-1)-1)1,((p-1)(q-1)+1)1^{q-1},(q-1)^p,(q-1)^{p-1}1^{q-1}\\
 &\qquad \oplus p0,(p-1)1,1^p,1^{p-1}1\\
 &\quad=10,10,1,01\oplus (pq-2)1,((p-1)q)1^{q-1},(q-1)q^{p-1},q^{p-2}1^{q-1}\\
 &\quad=q(10,10,1,10)\oplus((p-1)q-1)1,((p-2)q+1)1^{q-1},q^{p-1},q^{p-2}1^q
\end{align*}
The direct decomposition is a certain decomposition of spectral type into two parts (cf.~\cite[\S10.1]{Ow}) which corresponds to $\Delta(\mathbf m)$ introduced in Preface of \cite{Ow}.
From the decomposition we get the connection formula and the reducible condition decomposing into the corresponding subsystems etc.
These are explained in \cite{Ow} and \cite{Oir}.  In particular,
the algorithm to get the direct decomposition and the condition for the irreducibility of the corresponding rigid KZ-type equation are explained in \cite{Oir}.

For example, 
the second part of the first decomposition in the above is
$p0,(p-1)1,1^p,1^{p-1}1$
and it corresponds to the following $q(q-1)$ subschemes of the original Riemann scheme:
\begin{align*}
\begin{Bmatrix}
x_0=x_1 & x_0=x_2 & x_0=x_3 & x_0=x_4\\
[0]_p & [0]_{p-1} &\alpha'_{i'} & \alpha_i\\
\emptyset &\!\!\beta_\nu-\alpha''-\gamma\!\!&&\!\!\beta_{\nu'}+\gamma
\end{Bmatrix}\ \ (1\le\nu\le q,\ 1\le\nu'\le q-1).
\end{align*} 
Then, in the subscheme,  if the sum of all the exponents counting with their multiplicities is an integer, the system is reducible.  Namely, the condition
\begin{align*}
&p\cdot 0+(p-1)\cdot 0+(\beta_\nu-\alpha''-\gamma)+\sum\alpha'_{i'}+\sum\alpha_i+\beta_{\nu'}+\gamma\\
&\quad=\beta_\nu+\beta'_{\nu'}\in\mathbb Z
\end{align*}
means the reducibility of the system.
We get the same condition from the first part of the first decomposition. 
Then the necessary and sufficient condition for the reducibility is obtained from the direct 
decompositions.  Here we choose a part without multiplicity such as $10,10,1,10$ in the above
$q(10,10,1,10)$.
Thus we have the condition for the irreducibility:
\[ 
\alpha_i+\alpha'_{i'},\ 
\beta_j+\beta'_{j'},\ \alpha'_i+\beta'_j+\gamma\  \notin\mathbb Z,
\]
which coincides with Theorem~\ref{thm:irreducibility}.

Since the spectral type of the system for the variable $x$ is 
\[
 \mathbf m_x:(pq-1)1, (pq-q+1)1^{q-1},q^p,q^{p-1}1^q,
\]
one step of Katz-Haraoka's reduction algorithm by an addition and a middle convolution reduces the rank by   
\[d(\mathbf m_x)=(pq-1)+(pq-q+1)+q+q-2pq=q,\]
which means $p\mapsto p-1$.
When $p=1$, the spectral type of the 
differential equation for $x$-variable is
\[
 \underline{(q-1)}1,\underline{1}1^{q-1},\underline{q},\underline{0}1^q
\]
which equals the spectral type of the equation satisfied by ${}_qF_{q-1}(x)$
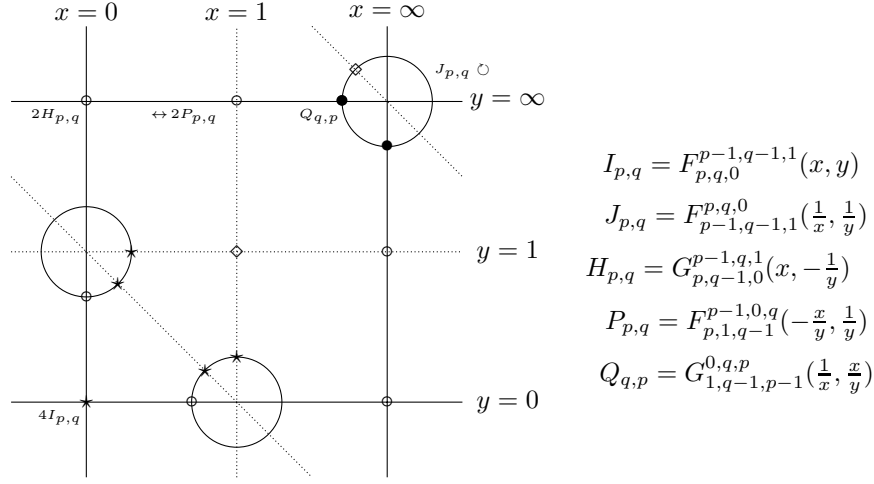
\bigskip

\begin{tikzpicture}
\draw
(0,2) circle[radius=0.6]
(2,0) circle[radius=0.6]
(4,4) circle[radius=0.6];
\draw[densely dotted]
(-1,2)--(5,2)  (2,-1)--(2,5) (-1,3)--(3,-1) (3,5)--(5,3);
\draw
(-1,0)--(5,0) (-1,4)--(5,4) (0,-1)--(0,5) (4,-1)--(4,5)
;
\draw
node at (0,5.2) {$x=0$}
node at (2,5.2) {$x=1$}
node at (4,5.2) {$x=\infty$}
node at (5.6,0) {$y=0$}
node at (5.6,2) {$y=1$}
node at (5.6,4) {$y=\infty$}
node at (-0.35,-0.2) {\tiny$4I_{p,q}$}
node at (0,4) {$\circ$}
node at (4,0) {$\circ$}
node at (0,1.4) {$\circ$}
node at (1.4,0) {$\circ$}
node at (2,4) {$\circ$}
node at (4,2) {$\circ$}
node at (3.4,4) {$\bullet$}
node at (4,3.4) {$\bullet$}
node at (0.6,2) {$\star$}
node at (2,0.6) {$\star$}
node at (0.42,1.58) {$\star$}
node at (1.58,0.42) {$\star$}
node at (2,2) {$\diamond$}
node at (3.4,4) {$\bullet$}
node at (3.58,4.42) {$\diamond$}
node at (3.4,4) {$\bullet$}
node at (0,0){$\star$}
node at (1.3,3.8){\tiny$\leftrightarrow\!2P_{p,q}$}
node at (3.1,3.8) {\tiny$Q_{q,p}$}
node at (5,4.4){\tiny${J_{p,q}\circlearrowright}$}
node at (-0.4,3.8) {\tiny$2H_{p,q}$}
;
\end{tikzpicture}
\ \ 
\raisebox{1cm}{\begin{tikzpicture}
\node at (0.23,0) {$I_{p,q}=F^{p-1,q-1,1}_{p,q,0}(x,y)$};
\node at (0.29,-0.7) {$J_{p,q}=F_{p-1,q-1,1}^{p,q,0}(\tfrac 1x,\tfrac 1y)$}; 
\node at (0.05,-1.4) {$H_{p,q}=G^{p-1,q,1}_{p,q-1,0}(x,-\tfrac1y)$};
\node at (0.3,-2.1) {$P_{p,q}=F^{p-1,0,q}_{p,1,q-1}(-\tfrac xy,\tfrac 1y)$};
\node at (0.3,-2.8) {$Q_{q,p}=G^{0,q,p}_{1,q-1,p-1}(\tfrac 1x,\tfrac xy)$};
\end{tikzpicture}}

\subsection{Appell's $F_2$ and $F_3$}
\label{sec:AppellF2}
Appell's $F_2$ is given by
\begin{gather*}
F_2(a;b,b';c,c';x,y)=F^{1,1,1}_{2,2,0}\left(
\begin{smallmatrix}\alpha&\beta&\gamma\\ \alpha',0 & \beta',0 & \emptyset\end{smallmatrix};x,y
\right),\\
 \quad \alpha=b,\ \beta=b',\ \gamma=a,\ \alpha'=1-c,\ \beta'=1-c'
\end{gather*}
and it satisfies $\mathcal M^{\alpha,\beta,\gamma}_{(\alpha',0),(\beta',0),\emptyset}$.
The solution 
$F^{2,2,0}_{1,1,1}\left(
\begin{smallmatrix}
\alpha',0 & \beta',0 & \emptyset\\
\alpha&\beta&\gamma
\end{smallmatrix};\alpha,\beta;\frac1x,\frac1y
\right)
$
to this system
at $(\infty,\infty)$ 
corresponds to Appell's $F_3$.
Note that 
\begin{equation}\label{eq:F2yinf}
  T_{(x,y)\to(-\frac xy,\frac 1y)}
  \mathcal M^{\alpha,\beta,\gamma}_{\boldsymbol\alpha',\boldsymbol\beta',\emptyset}
 = \mathcal M^{\alpha,\emptyset,\boldsymbol\beta'}_{\boldsymbol\alpha',\gamma,\beta}.
\end{equation}
\begin{lem}\label{lem:specialF2coord}
 $T_{(x,y)\to(-x,1-y)}\bar{\mathcal M}^{\alpha,\emptyset,(\gamma_1,\gamma_2)}
 _{(\alpha',0),0,\gamma'}
 = \bar{\mathcal M}^{\alpha,\emptyset,(\gamma_1,\gamma_2)}_{(\alpha',0),0,1-\gamma_1-\gamma_2-\gamma'}$.
\end{lem}
\begin{proof} Note that 
$\bar{\mathcal M}^{\alpha,\emptyset,(\gamma_1,\gamma_2)}_{(\alpha',0),0,\gamma'}$ is 
\begin{align*}
\begin{cases}
 \p_x(\vartheta_x-\alpha')(\vartheta_x+\vartheta_y-\gamma')u
  =(\vartheta_x+\alpha)(\vartheta_x+\vartheta_y+\gamma_1)(\vartheta_x+\vartheta_y+\gamma_2)u,\\
 \p_y(\vartheta_x+\vartheta_y-\gamma')u
  =(\vartheta_x+\vartheta_y+\gamma_1)(\vartheta_x+\vartheta_y+\gamma_2)u,\\
  (\vartheta_x+\alpha)\p_yu=\p_x(\vartheta_x-\alpha')u.
\end{cases}
\end{align*}
Putting $(x,y)=(-X, 1-Y)$, we have $(\vartheta_X+\alpha)\p_Yu=\p_X(\vartheta_X-\alpha')u$ and
\begin{align*}
&(\vartheta_X+\alpha)(\vartheta_X+\vartheta_Y+\gamma_1)
(\vartheta_X+\vartheta_Y+\gamma_2)u+\p_X(\vartheta_X-\alpha')(\vartheta_X+\vartheta_Y-\gamma')u
\\&\quad=(\vartheta_X+\alpha)\bigl((\vartheta_X+\vartheta_Y+\gamma_2+1)\p_Y
 +(\vartheta_X+\vartheta_Y+\gamma_1)\p_Y \bigr)u\\
&\quad=(2\vartheta_X+2\vartheta_Y+\gamma_1+\gamma_2+1)\p_X(\vartheta_X-\alpha')u\\
&\quad=\p_X(\vartheta_X-\alpha')(2\vartheta_X+2\vartheta_Y+\gamma_1+\gamma_2-1)u,\\
&(\vartheta_X+\vartheta_Y+\gamma_1)(\vartheta_X+\vartheta_Y+\gamma_2)u
+\p_Y(\vartheta_X+\vartheta_Y-\gamma')u\\
&\quad=(\vartheta_X+\vartheta_Y+\gamma_1)\p_Yu+\p_Y(\vartheta_X+\vartheta_Y+\gamma_2)u\\
&\quad=\p_Y(2\vartheta_X+2\vartheta_Y+\gamma_1+\gamma_2-1)u,
\end{align*}
which implies the lemma.
\end{proof}

\begin{rem}
The above lemma is
valid in the case $\alpha,\,\alpha'\in\mathbb C^{p-1}$ for $p\ge 2$.
Moreover  
$\gamma$, $\gamma_1$ and $\gamma_2$ can be operators 
which commute with $x$, $\p_x$, $y$ and  $\p_y$.
\end{rem}

\begin{prop}\label{prop:specialF2}
One has a relation
\[
 T^{F_2}_{y:0\to1}\mathcal M^{\alpha,\beta,\gamma}_{\boldsymbol\alpha',\boldsymbol\beta',\emptyset}
 = \mathcal M^{\alpha,\emptyset,\boldsymbol\beta'}
 _{\boldsymbol\alpha',\gamma,1-\beta-\beta'_1-\beta'_2},
\]
where
\[
 T^{F_2}_{y:0\to1}=\Ad(y^{\gamma})\circ T_{(x,y)\to(-x,1-y)}\circ \Ad(y^{-\gamma})\circ
 T_{(x,y)\to(-\frac xy,\frac 1y)}.
\]
\end{prop}

\begin{proof}
The proposition follows from the following sequence of transformations.
    \begin{align*}
& \mathcal M^{\alpha,\beta,\gamma}_{\boldsymbol\alpha',\boldsymbol\beta',\emptyset}
\xrightarrow{T_{(x,y)\to(-\frac xy,\frac 1y)}}
 \mathcal M^{\alpha,\emptyset,\boldsymbol\beta'}_{\boldsymbol\alpha',\gamma,\beta}
\xrightarrow{\Ad(x^{-\alpha'_1}y^{-\gamma})}
 \mathcal M^{\alpha+\alpha'_1,\emptyset,\beta'+\alpha'_1+\gamma}
 _{\boldsymbol\alpha'-\alpha'_1,0,\beta-\alpha'_1-\gamma}\\
&\quad \xrightarrow{T_{(x,y)\to(-x,1-y)}}
 \mathcal M^{\alpha+\alpha'_1,\emptyset,\boldsymbol\beta'+\alpha'_1+\gamma}
 _{\boldsymbol\alpha'-\alpha'_1,0,1-\alpha'_1-\beta-\beta'_1-\beta'_2-\gamma}\\
&\quad\xrightarrow{\Ad(x^{\alpha'_1}y^{\gamma})}
 \mathcal M^{\alpha,\emptyset,\boldsymbol\beta'}
 _{\boldsymbol\alpha',\gamma,1-\beta-\beta'_1-\beta'_2}.
\end{align*}
\end{proof}

\begin{rem} {\rm i)} \ 
Note that 
$T^{F_2}_{y:0\to1}$ acts only on the system
$\mathcal M
 ^{\boldsymbol \alpha, \boldsymbol \beta, \boldsymbol \gamma}
 _{\boldsymbol \alpha',\boldsymbol \beta',\boldsymbol \gamma'} 
$
satisfied by $F^{p-1,1,1}_{p,2,0}\Bigl(
\begin{smallmatrix}
 \boldsymbol \alpha & \boldsymbol \beta & \boldsymbol \gamma\\
 \boldsymbol \alpha'& \boldsymbol \beta'& \boldsymbol \gamma'
\end{smallmatrix}
 x,y\Bigr)$. 

{\rm ii)}
$T_{(x,y)\to(-x,1-y)}\circ T_{(x,y)\to(-\frac xy,\frac 1y)}=T_{(x,y)\to(\frac xy,1-\frac 1y)}$.
Hence  
$y^{-\gamma}u(\frac xy,1-\frac1y)$ is a solution to 
$
\mathcal M
 ^{\alpha, \beta, \gamma}
 _{(\alpha',0),(\beta',0),\emptyset}
$
if $u(x,y)$ is a solution to 
$
\mathcal M
 ^{\alpha, \emptyset, (\beta'+\gamma,\gamma)}
 _{(\alpha',0),0,1-\beta-\beta'-\gamma}
$.
In particular, putting $u(x,y)=F^{1,0,2}_{2,0,1}
\left(\begin{smallmatrix}
\alpha & \emptyset & \beta'+\gamma,\gamma\\
\alpha',0& 0 & 1-\beta-\beta'-\gamma
\end{smallmatrix};x,y\right)$, 
we have a solution to
$
\mathcal M
 ^{\alpha, \beta, \gamma}
 _{(\alpha',0),(\beta',0),\emptyset}
$ around $(0,1)$.
\end{rem}

Olsson \cite{Ol} gives the solution to $\mathcal M
 ^{\alpha, \beta, \gamma}
 _{(\alpha',0),(\beta',0),\emptyset}
$ around $(0,1)$ :
\begin{equation}
F_P(a,b,b',c,c';x,y):=\sum\frac{(a)_{m+n}(a-c'+1)_m(b)_m(b')_n}
 {(a+b'-c'+1)_{m+n}(c)_mm!n!}x^m(1-y)^n.
\end{equation}
Owing to the above remark, we have the identity
\begin{equation}
\begin{split}
F_P(a,b,b',c,c';x,y)&:=
y^{-a}F^{1,0,2}_{2,1,1}\left(\begin{smallmatrix}
b_1& \emptyset & a-c'+1,b'\\ 1-c,0&0&c'-a-b'
\end{smallmatrix};\tfrac xy,1-\tfrac 1y\right),
\end{split}
\end{equation}
which is also proved by the following direct calculation
\begin{align*}
&\sum\frac{(a)_{m+n}(f)_m(g)_n}{(f+g)_{m+n}m!n!}x^my^n
 =\sum_m \frac{(a)_m(f)_m}{(f+g)_mm!}x^m
 \sum_n\frac {(a+m)_n(g)_n}{(f+g+m)_nn!}y^n\\
 &=\sum_m \frac{(a)_m(f)_m}{(f+g)_mm!}x^m(1-y)^{-a-m}\sum_n\frac {(a+m)_n(f+m)_n}{(f+g+m)_nn!}
   \Bigl(\frac y{y-1}\Bigr)^n\\
 &=(1-y)^{-a}\sum \frac{(a)_{m+n}(f)_{m+n}}{(f+g)_{m+n}m!n!}\Bigl(\frac x{1-y}\Bigr)^m \Bigl(\frac y{y-1}\Bigr)^n.
\end{align*}
Here the second equality in the above follows from Kummer's formula for Gauss' hypergeometric 
series.
Applying the integral transformation $K_x^{1-\mu-\lambda,\lambda}$, we have
\begin{align}
&F^{p,1,1}_{p,1,1}\left(\begin{smallmatrix}
 \boldsymbol\alpha,a & \beta & \gamma\\
 \boldsymbol\alpha',0 & 0 & 1-a-\beta
\end{smallmatrix};x,y
\right)=(1-y)^{-\gamma}F^{p-1,0,2}_{p,1,1}\left(\begin{smallmatrix}
 \boldsymbol\alpha & \emptyset &a
,\gamma\\
  \boldsymbol\alpha',0 & 0& 1-a-\beta
\end{smallmatrix};\frac x{1-y},\frac y{y-1}\right).
\end{align}
\begin{rem}\label{rem:symF2}\ {\rm i) [Symmetry of $I_{p,2}$]}
\ 
Since \[T_{(x,y)\to(-\frac xy,\frac 1y)} \circ T_{(x,y)\to(-x,1-y)}\circ T_{(x,y)\to(-\frac xy,\frac 1y)}=T_{(x,y)\to(\frac x{1-y},\frac y{y-1})}\]
and 
$T_{(x,y)\to(-\frac xy,\frac 1y)}\circ T^{F_2}_{y:0\leftrightarrow1}\circ
T_{(x,y)\to(-\frac xy,\frac 1y)}
\mathcal M
 ^{\boldsymbol\alpha, \beta, \gamma}
 _{(\boldsymbol\alpha',0),(\beta',0),\emptyset}=
\mathcal M
 ^{\boldsymbol\alpha, 1-\beta-\gamma,\gamma}
 _{(\boldsymbol\alpha',0),(\beta',0),\emptyset}$, we have
\[F^{p-1,1,1}_{p,2,0}\Bigl(
\begin{smallmatrix}
 \boldsymbol \alpha & \beta & \gamma\\
 \boldsymbol \alpha',0& \beta',0& \emptyset
\end{smallmatrix}
 x,y\Bigr)=(1-y)^{-\gamma}
F^{p-1,1,1}_{p,2,0}\Bigl(
\begin{smallmatrix}
 \boldsymbol \alpha & 1-\beta-\gamma
 & \gamma\\
 \boldsymbol \alpha',0& \beta',0&\emptyset
\end{smallmatrix}
 ;\frac x{1-y},\frac y{y-1}\Bigr).
\]

{\rm ii)\ [Symmetry of $I_{2,2}$]}\ When $p=q=2$, $T_{(x,y)\to(y,x)}$ naturally acts on $\mathcal M
 ^{\alpha, \beta, \gamma}
 _{(\alpha',0),(\beta',0),\emptyset}$ and the function 
$F_2(a;b,b';c.c';x,y)$ has a symmetry of the reflection group  $W\!_{B_2}$ of type $B_2$ since the coordinate transformations 
$T_{(x,y)\to(\frac x{1-y},\frac y{y-1})}$ and $T_{(x,y)\to(y,x)}$ generate $W\!_{B_2}$.
For example,
\begin{align*}
&T_{(x,y)\to(y,x)}\circ T_{(x,y)\to(\frac x{1-y},\frac y{y-1})}\circ T_{(x,y)\to(x,y)}\circ
T_{(x,y)\to(\frac x{1-y},\frac y{y-1})}\\&\qquad =T_{(x,y)\to(\frac x{x+y-1},\frac y{x+y-1})}
\end{align*}
and we have (cf.~\cite[\S5.11]{Er})
\begin{align*}
&F_2(a;b,b';c,c';x,y)=(1-y)^{-a}F_2(a;b,c'-b';c,c';\frac x{1-y},\frac y{y-1})\\
&\qquad=(1-x)^{-a}F_2(a;c-b,b,b';c,c';\frac x{x-1},\frac y{1-x})\\
&\qquad=(1-x-y)^{-a}F_2(a;c-b;c-b';c,c';\frac x{x+y-1},\frac y{x+y-1}).
\end{align*}
\end{rem}
Resolving the singularities $(0,1)$, $(1,0)$ and $(\infty,\infty)$ of 
$\mathcal M^{\alpha,\beta,\gamma}_{(\alpha',0),(\beta',0),\emptyset}$ 
in $\mathbb P^1\times\mathbb P^1$, we have a manifold $\tilde X$, which is isomorphic to 
the manifold constructed in \S\ref{sec:F1} by the correspondence 
$(x,y)\mapsto(x,1-y)$.
We denote  
the exceptional lines in $\tilde X$
corresponding to $(0,1)$, $(1,0)$ and $(\infty,\infty)$ by $L_{(0,1)}$, $L_{(1,0)}$ and $L_{(\infty,\infty)}$, respectively. 

Under the coordinate $(X,Y)=(\frac xy, 1-\frac 1y)$, the original system has a solution
$F^{p-1,0,2}_{p,1,1}(X,Y)$ and we apply the result in \S\ref{sec:connection}.
Then  we have solutions at the vertices of the pentagon with the edges 
$L_{(0,1)}$, $L_{y=1}$, 
$L_{x=\infty}$, $L_{y=0}$ and  $L_{x=0}$ and solve the connection problem between them. We have already solved the connection problem between the 
vertices of the pentagon with the edges $L_{y=0}$, $L_{x=\infty}$, $L_{(\infty,\infty)}$,
$L_{y=\infty}$ and $L_{x=\infty}$. 
Here $L_{x=0}$ means the line in $\tilde X$ and the pentagon is obtained by shrinking the edge $L_{(0,0)}$ of a hexagon to a vertex.
Hence we can solve the connection problem among 8 points when $p>2$ and $q=2$.

If $p=q=2$, $T_{(x,y)\to(y,x)}$ can be applied and we can add other three points,
which include $(1,\infty)$.
The transformation $T_{(x,y)\to(\frac x{1-y},\frac y{y-1})}$ maps $y=\infty$ to $(0,1)$ and $(1,\infty)$ to $L_{(0,1)}\cap L_{x+y=1}$.
Using the isomorphism corresponding to this transformation, we can analyze the local solutions at $L_{(0,1)}\cap L_{x+y=1}$.
 Thus, we have local solutions and connection formula 
among them at 13 normally crossing singular points in $\tilde X$.
The remaining 2 points are $(1,1)$ and $L_{(\infty,\infty)}\cap L_{x+y=1}$.

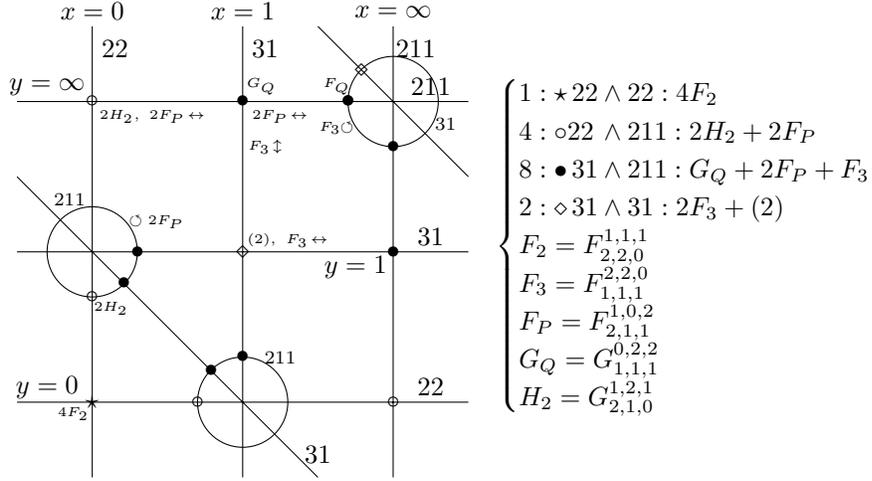
\begin{tikzpicture}
\draw
(-1,0)--(5,0) (-1,2)--(5,2) (-1,4)--(5,4)
(0,-1)--(0,5) (2,-1)--(2,5) (4,-1)--(4,5)
(-1,3)--(3,-1) (3,5)--(5,3)
;
\draw
node at (0,5.2) {$x=0$}
node at (2,5.2) {$x=1$}
node at (4,5.2) {$x=\infty$}
node at (-0.6,0.2) {$y=0$}
node at (3.5,1.8)  {$y=1$}
node at (-0.6,4.2) {$y=\infty$}
node at (0.3,4.7) {22}
node at (2.3,4.7) {31}
node at (4.3,4.7) {211}
node at (4.5,0.2) {22}
node at (4.5,2.2) {31}
node at (4.5,4.2) {211}
node at (3,-0.7) {31}
node at (-0.25,-0.15) {\tiny$4F_2$}
node at (0.25,1.25) {\tiny$2H_2$}
node at (0.85,2.4) {\tiny$\circlearrowleft 2F_P$}
node at (2.6,2.15) {\tiny$(2),\ F_3\!\leftrightarrow$}
node at (0.8,3.8) {\tiny$2H_2,\ 2F_P\!\leftrightarrow$}
node at (2.25,4.2) {\tiny$G_Q$}
node at (3.25,4.2) {\tiny$F_Q$}
node at (3.25,3.65) {\tiny$F_3\!\!\circlearrowleft$}
node at (2.5,3.8) {\tiny$2F_P\!\leftrightarrow$}
node at (2.3,3.4) {\tiny$F_3\!\updownarrow$}
node at (0,4) {$\circ$}
node at (4,0) {$\circ$}
node at (0,1.4) {$\circ$}
node at (1.4,0) {$\circ$}
node at (2,4) {$\bullet$}
node at (4,2) {$\bullet$}
node at (3.4,4) {$\bullet$}
node at (4,3.4) {$\bullet$}
node at (0.6,2) {$\bullet$}
node at (2,0.6) {$\bullet$}
node at (0.42,1.58) {$\bullet$}
node at (1.58,0.42) {$\bullet$}
node at (2,2) {$\diamond$}
node at (3.4,4) {$\bullet$}
node at (3.58,4.42) {$\diamond$}
node at (0,0) {$\star$}
;
\draw
(0,2) circle[radius=0.6]
(2,0) circle[radius=0.6]
(4,4) circle[radius=0.6]
node at (-0.3,2.7) {\scriptsize$211$}
node at (2.5,0.6) {\scriptsize$211$}
node at (4.7,3.7) {\scriptsize$31$}
;
\end{tikzpicture}\quad
\raisebox{3cm}{$\begin{cases}
  1:{\star\,}22\wedge 22 : 4F_2\\
  4:{\circ}22\,\wedge 211: 2H_2+2F_P\\
  8:{\bullet\,}31\wedge 211: G_Q+2F_P+F_3\\
  2:{\diamond\,}31\wedge 31 : 2F_3+(2)\\
  F_2=F^{1,1,1}_{2,2,0}\\
  F_3=F^{2,2,0}_{1,1,1}\\
  F_P=F^{1,0,2}_{2,1,1}\\
  G_Q=G^{0,2,2}_{1,1,1}\\
  H_2=G^{1,2,1}_{2,1,0}
\end{cases}$}

\begin{rem}
As a consequence we have a base of local solutions expressed by our hypergeometric functions
at every normally crossing singular point in 
$\tilde X$ whose multiplicities of simultaneous eigenvalues are free.
Moreover the calculation of the  connection coefficients between two such points along 
a singular line is given.  
Note that all such points are connected by singular lines.

Some connection formulas for $F_2$ are obtained in \cite{Mi}.
\end{rem}

\section{Many variables and GKZ system}\label{sec:GKZ}
In this section, we introduce a multivariate extension $\mathcal{M}^{\boldsymbol{\alpha}_1,\dots,\boldsymbol{\alpha}_n,\boldsymbol{\gamma}}_{\boldsymbol{\alpha}'_1,\dots,\boldsymbol{\alpha}'_n,\boldsymbol{\gamma}'}$ of the system $\mathcal M^{\boldsymbol\alpha,\boldsymbol\beta,\boldsymbol\gamma}
_{\boldsymbol\alpha',\boldsymbol\beta',\boldsymbol\gamma'}$.
To analyze this system, the viewpoint of GKZ system will be useful.

\subsection{The system $\mathcal{M}^{\boldsymbol{\alpha}_1,\dots,\boldsymbol{\alpha}_n,\boldsymbol{\gamma}}_{\boldsymbol{\alpha}'_1,\dots,\boldsymbol{\alpha}'_n,\boldsymbol{\gamma}'}$}\label{sec:SystemM}
Let $n$ be a positive integer, $x=(x_1,\dots,x_n)$ be $n$ complex variables.
We consider non-negative integers $p_1,\dots,p_n,r,p'_1,\dots,p'_n,r'$ and parameters $\boldsymbol\alpha_i\in\C^{p_i},\boldsymbol\gamma\in\C^{r},\boldsymbol\alpha'_i\in\C^{p'_i},$ and $\boldsymbol\gamma'\in\C^{r'}$.
Then, for $i=1,\dots,n$ and $1\leq i<j\leq n$, we set
\begin{align}
    P_i&:=(\vartheta_{x_i}-\boldsymbol\alpha'_{i})(\vartheta_{x_1}+\cdots+\vartheta_{x_n}-\boldsymbol\gamma')-x_i(\vartheta_{x_i}+\boldsymbol\alpha_{i})(\vartheta_{x_1}+\cdots+\vartheta_{x_n}+\boldsymbol\gamma),\nonumber\\
    P_{ij}&:=x_i(\vartheta_{x_i}+\boldsymbol\alpha_{i})(\vartheta_{x_j}-\boldsymbol\alpha'_{j})-x_j(\vartheta_{x_j}+\boldsymbol\alpha_{j})(\vartheta_{x_i}-\boldsymbol\alpha'_{i}),
    \label{eq:P_iP_ij}
\end{align}
where we use the convention as \eqref{eq:abbrev} and $\theta_{x_i}:=x_i\frac{\partial}{\partial x_i}$.
For an unknown function $u(x)$, the system $\mathcal{M}^{\boldsymbol{\alpha}_1,\dots,\boldsymbol{\alpha}_n,\boldsymbol{\gamma}}_{\boldsymbol{\alpha}'_1,\dots,\boldsymbol{\alpha}'_n,\boldsymbol{\gamma}'}$ is defined by
\begin{equation}
\begin{cases}
    P_i\cdot u(x)=0&(i=1,\dots,n),\\
    P_{ij}\cdot u(x)=0&(1\leq i<j\leq n).
\end{cases}
\end{equation}

 For a multi-index $m=(m_1,\dots,m_n)\in\Z^n_{\geq 0}$ and $\epsilon=(\epsilon_1,\dots,\epsilon_n)\in\{-1,+1\}^n$, we set $|m|:=m_1+\cdots+m_n$ and $\epsilon\cdot m:=\epsilon_1m_1+\cdots+\epsilon_nm_n$. 
 We write $F^{p_1,\dots,p_n,r}_{p'_1,\dots,p_n',r'}\left(\substack{{\boldsymbol \alpha}_1,\dots,{\boldsymbol \alpha}_n,{\boldsymbol \gamma}\\ {\boldsymbol \alpha}'_1,\dots,{\boldsymbol \alpha}'_n,{\boldsymbol \gamma}'};x\right)$ for the following multi-variate series:
\begin{equation}
   \sum_{m=(m_1,\dots,m_n)\in\Z^n_{\geq 0}}\frac{({\boldsymbol \alpha}_1)_{m_1}\cdots ({\boldsymbol \alpha}_n)_{m_n}({\boldsymbol \gamma})_{|m|}}{(\boldsymbol{1-\alpha}'_1)_{m_1}\cdots (\boldsymbol{1-\alpha}'_n)_{m_n}(\boldsymbol{1-\gamma}')_{|m|}}x^m.
   \label{eq:multivariate_F}
\end{equation}
Instead of $F^{p_1,\dots,p_n,r}_{p'_1,\dots,p_n',r'}\left(\substack{{\boldsymbol \alpha}_1,\dots,{\boldsymbol \alpha}_n,{\boldsymbol \gamma}\\ {\boldsymbol \alpha}'_1,\dots,{\boldsymbol \alpha}'_n,{\boldsymbol \gamma}'};x\right)$
we often use a symbol $F\left(\substack{{\boldsymbol \alpha}_1,\dots,{\boldsymbol \alpha}_n,{\boldsymbol \gamma}\\ {\boldsymbol \alpha}'_1,\dots,{\boldsymbol \alpha}'_n,{\boldsymbol \gamma}'};x\right)$
in the following.

Proposition \ref{prop:transformations} is generalized as follows.

\begin{prop}\label{prop:transformations2}
One has the following transformations of the system $\mathcal M
 ^{\boldsymbol \alpha, \boldsymbol \beta, \boldsymbol \gamma}
 _{\boldsymbol \alpha',\boldsymbol \beta',\boldsymbol \gamma'}$.
\begin{itemize}
\item[\rm (1)] $T_{(x_1,\dots,x_n)\mapsto(x_{\s(1)},\dots,x_{\s(n)})}
\mathcal{M}^{\boldsymbol{\alpha}_1,\dots,\boldsymbol{\alpha}_n,\boldsymbol{\gamma}}_{\boldsymbol{\alpha}'_1,\dots,\boldsymbol{\alpha}'_n,\boldsymbol{\gamma}'}
=
\mathcal{M}^{\boldsymbol{\alpha}_{\s(1)},\dots,\boldsymbol{\alpha}_{\s(n)},\boldsymbol{\gamma}}_{
\boldsymbol{\alpha}'_{\s(1)},\dots,\boldsymbol{\alpha}'_{\s(n)},\boldsymbol{\gamma}'}
.
$
\item[\rm (2)] $T_{(x_1,\dots,x_n)\to(\tfrac1x_1,\dots,\tfrac1x_n)}
\mathcal{M}^{\boldsymbol{\alpha}_1,\dots,\boldsymbol{\alpha}_n,\boldsymbol{\gamma}}_{\boldsymbol{\alpha}'_1,\dots,\boldsymbol{\alpha}'_n,\boldsymbol{\gamma}'}
=
\mathcal{M}^{\boldsymbol{\alpha}'_1,\dots,\boldsymbol{\alpha}'_n,\boldsymbol{\gamma}'}_{
\boldsymbol{\alpha}_{1},\dots,\boldsymbol{\alpha}_n,\boldsymbol{\gamma}}.$
 \item[\rm (3)] $T_{(x_1,\dots,x_n)\to\left(\epsilon\frac{x_1}{x_n},\dots,\epsilon\frac{x_{n-1}}{x_n},\frac{1}{x_n}\right)}
\mathcal{M}^{\boldsymbol{\alpha}_1,\dots,\boldsymbol{\alpha}_n,\boldsymbol{\gamma}}_{\boldsymbol{\alpha}'_1,\dots,\boldsymbol{\alpha}'_n,\boldsymbol{\gamma}'}
=
\mathcal{M}^{\boldsymbol{\alpha}_1,\dots,\boldsymbol{\alpha}_{n-1},\boldsymbol{\gamma}',\boldsymbol{\alpha}'_n}_{
\boldsymbol{\alpha}'_{1},\dots,\boldsymbol{\alpha}'_{n-1},\boldsymbol{\gamma},\boldsymbol{\alpha}_{n}}$.
\end{itemize}
\end{prop}

\noindent
The proof of Proposition \ref{prop:transformations2} is parallel to that of Proposition \ref{prop:transformations} and is omitted.

As in \S\ref{sec:Cond}, we always assume the following condition
\begin{gather}
p_i+r=p'_i+r'\quad (i=1,\dots,n)
\tag{F}
\end{gather}
for the system $\mathcal{M}^{\boldsymbol{\alpha}_1,\dots,\boldsymbol{\alpha}_n,\boldsymbol{\gamma}}_{\boldsymbol{\alpha}'_1,\dots,\boldsymbol{\alpha}'_n,\boldsymbol{\gamma}'}$.
We call the number
\begin{equation}
    L:=r-r'
\end{equation}
the level of the system $\mathcal{M}^{\boldsymbol{\alpha}_1,\dots,\boldsymbol{\alpha}_n,\boldsymbol{\gamma}}_{\boldsymbol{\alpha}'_1,\dots,\boldsymbol{\alpha}'_n,\boldsymbol{\gamma}'}$ and set
$$
\epsilon:=(-1)^L.
$$
Moreover, we assume 
\begin{equation}
(\boldsymbol\alpha'_1)_1=(\boldsymbol\alpha'_2)_1=\cdots=(\boldsymbol\alpha'_n)_1=0,\tag{T}
\end{equation}
$p_1+p_1',\cdots,p_n+p_n'>1$ and $r+r'>0$
in the following discussion.
It is easy to see that $F^{p_1,\dots,p_n,r}_{p'_1,\dots,p_n',r'}\left(\substack{{\boldsymbol \alpha}_1,\dots,{\boldsymbol \alpha}_n,{\boldsymbol \gamma}\\ {\boldsymbol \alpha}'_1,\dots,{\boldsymbol \alpha}'_n,{\boldsymbol \gamma}'};x\right)$ is a solution to the system $\mathcal{M}^{\boldsymbol{\alpha}_1,\dots,\boldsymbol{\alpha}_n,\boldsymbol{\gamma}}_{\boldsymbol{\alpha}'_1,\dots,\boldsymbol{\alpha}'_n,\boldsymbol{\gamma}'}$ under the condition (T).

As in \S\ref{sec:IntegRep}, 
 the series \eqref{eq:multivariate_F} admits an integral representation as follows.

\begin{prop}
Assume $\alpha'_{ip'_i}=0$ for any $i=1,\dots,n$.
Then, the following identity holds.
\begin{align*}
&F^{p_1,\dots,p_n,r}_{p'_1,\dots,p_n',r'}\left(\substack{{\boldsymbol \alpha}_1,\dots,{\boldsymbol \alpha}_n,{\boldsymbol \gamma}\\ {\boldsymbol \alpha}'_1,\dots,{\boldsymbol \alpha}'_n,{\boldsymbol \gamma}'};x\right)\\
&=\prod_{i=1}^n \prod_{\nu=1}^{p'_i-1}K_{x_i}^{1-\alpha'_{i,\nu}-\alpha_{i,\nu},\alpha_{i,\nu}}
\prod_{k=1}^r \tilde K_{\mathbf x}^{1-\gamma'_k-\gamma_k,\gamma_k}\\
&\quad\prod_{k=1}^{r'-r}K_{\mathbf x}^{1-\gamma'_{r+k}-\sum_{i=1}^n\alpha_{p'_i+k-1},
(\alpha_{1,p'_1+k-1},\dots,\alpha_{n,p'_n+k-1})}\prod_{i=1}^{n}\bigl(1-x_i\bigr)^{-\alpha_{p_i}}\\
&\qquad \text{if }r\le r',
\\
&=\prod_{i=1}^n\prod_{\nu=1}^{p_i}K_{x_i}^{1-\alpha'_{i,\nu}-\alpha_{i,\nu},\alpha_{i,\nu}}
\prod_{k=1}^{r'} \tilde K_{\mathbf x}^{1-\gamma'_k-\gamma_k,\gamma_k}\\
&\quad\prod_{k=1}^{r-r'-1}L_{\mathbf x}^{1-\gamma_{r+k}-\sum_{i=1}^n\alpha'_{i,p+k},
(\alpha'_{1,p+k},\dots,\alpha'_{n,q+k})}\bigl(1-\sum\limits_{i=1}^nx_i\bigr)^{-\gamma_r}\quad \text{ if }r>r'.
\end{align*}
\end{prop}

\subsection{The singular set}\label{sec:Singularset}
We describe the singular set of the system $\mathcal{M}^{\boldsymbol{\alpha}_1,\dots,\boldsymbol{\alpha}_n,\boldsymbol{\gamma}}_{\boldsymbol{\alpha}'_1,\dots,\boldsymbol{\alpha}'_n,\boldsymbol{\gamma}'}$.
As in \S\ref{sec:Sing}, we define the singular set $\mathrm{Sing}(\mathcal{M}^{\boldsymbol{\alpha}_1,\dots,\boldsymbol{\alpha}_n,\boldsymbol{\gamma}}_{\boldsymbol{\alpha}'_1,\dots,\boldsymbol{\alpha}'_n,\boldsymbol{\gamma}'})$ of the system $\mathcal{M}^{\boldsymbol{\alpha}_1,\dots,\boldsymbol{\alpha}_n,\boldsymbol{\gamma}}_{\boldsymbol{\alpha}'_1,\dots,\boldsymbol{\alpha}'_n,\boldsymbol{\gamma}'}$ by 
\begin{align*}
 &\{x\in\mathbb C^n\mid \exists \xi\in\mathbb C^n\setminus\{0\}
 \text{ such that }\\
 &\quad \sigma(P_i)(x,\xi)=0\ (i=1,\dots,n)\ \ \text{and}  \ \sigma(P_{ij})(x,\xi)=0\ (1\leq i<j\leq n)\}.
\end{align*}
Note that $\mathrm{Sing}(\mathcal{M}^{\boldsymbol{\alpha}_1,\dots,\boldsymbol{\alpha}_n,\boldsymbol{\gamma}}_{\boldsymbol{\alpha}'_1,\dots,\boldsymbol{\alpha}'_n,\boldsymbol{\gamma}'})$ is a closed subvariety of $\C^n$.
To state the theorem, we prepare a notation.
Let $f(x)$ be a Laurent polynomial in the variables $x_1,\dots,x_n$.
When $f(x)$ contains a negative power in $x_i$, we define $m_{f,i}$ as the smallest integer such that $x_i^{m_{f,i}}f(x)$ is a polynomial in $x_i$.
When $f(x)$ is a polynomial in $x_i$, we set $m_{f,i}:=0$.
We define the polynomial part of $f$ by
$$
{\rm p.p.}f(x)
:=\prod_{i=1}^nx_i^{m_{f,i}}f(x).
$$

\begin{thm}\label{thm:Singn}
The singular set $\mathrm{Sing}(\mathcal{M}^{\boldsymbol{\alpha}_1,\dots,\boldsymbol{\alpha}_n,\boldsymbol{\gamma}}_{\boldsymbol{\alpha}'_1,\dots,\boldsymbol{\alpha}'_n,\boldsymbol{\gamma}'})$ is given by the union of the following irreducible subvarieties:
\begin{itemize}
 \item[\rm (1)] Suppose $L\neq 0$. Then, the subvarieties are as follows
 \begin{equation}
     x_i=0\ \ \ \ (i=1,\dots,n),
\end{equation}
\begin{equation}
         x_i=1\ \ \ \ \left(i=1,\dots,n,\ \displaystyle\prod_{j\neq i}p_j> 0\right),
         \label{eq:singular_set1}
\end{equation}
\begin{equation}
     {\rm p.p.}\displaystyle\prod_{i\in I}\prod_{\omega_i\in U_L}\left(\sum_{i\in I}\omega_ix_i^{\frac{1}{L}}\right)=0 \ \ \ \left(\text{ if }r'>0,\ \ \prod\limits_{i\notin I}p_i>0\right),\label{eq:singular_set2}    
\end{equation}
\begin{equation}
     {\rm p.p.}\displaystyle\prod_{i\in I}\prod_{\omega_i\in U_L}\left(1-\sum_{i\in I}\omega_ix_i^{\frac{1}{L}}\right)=0\ \ \ \ \left(\displaystyle\prod_{i\notin I}p_i>0\right).
 \label{eq:singular_set3}    
\end{equation}

\noindent
Here, $I\subset\{1,\dots,n\}$ in \eqref{eq:singular_set2} and \eqref{eq:singular_set3} runs over subsets with cardinality greater than or equal to two and $U_L$ is the set of $L$-th roots of unity.
Note that the symbol ${\rm p.p.}$ can be omitted when $L>0$. 
 \item[\rm (2)] Suppose $L=0$. Then, the subvarieties are all linear:
 \begin{equation}
     x_i=0,\ \ x_i=1,\ \ x_i=x_j\ \ \ (i=1,\dots,n,\ i<j). 
     \label{eq:braid_arrangement}
 \end{equation}
\end{itemize}
\end{thm}

\begin{proof}
In view of Proposition \ref{prop:transformations2}, we may assume $L\geq 0$.
The singular set $\mathrm{Sing}(\mathcal{M}^{\boldsymbol{\alpha}_1,\dots,\boldsymbol{\alpha}_n,\boldsymbol{\gamma}}_{\boldsymbol{\alpha}'_1,\dots,\boldsymbol{\alpha}'_n,\boldsymbol{\gamma}'})$ is contained in the vanishing locus of the symbols of the operators \eqref{eq:P_iP_ij}.
The equations are given by
\begin{align}
    (x_i\xi_i)^{p_i}(x_1\xi_1+\cdots+x_n\xi_n)^{r'}\left((x_i\xi_i)^L-x_i(x_1\xi_1+\cdots+x_n\xi_n)^L\right)=0\label{eq:i}
\end{align}
for $i=1,\dots,n$ and
\begin{align}
    (x_i\xi_i)^{p_i}(x_j\xi_j)^{p_j}\left( x_i(x_j\xi_j)^L-x_j(x_i\xi_i)^L\right)
    \label{eq:ij}
\end{align}
for $1\leq i<j\leq n$.
A point $x\in\C^n$ lies on the singular locus if there exists a vector $\xi=(\xi_1,\dots,\xi_n)\in\C^n\setminus\{0\}$ such that $(x,\xi)$ satisfies \eqref{eq:i} and \eqref{eq:ij}.
Since $x_i=0$ satisfies  \eqref{eq:i} and \eqref{eq:ij} with $\xi_i=1$ and $\xi_j=0$ for any $j\neq i$, we assume $x_1\cdots x_n\neq 0$ below.
Let $(x,\xi)$ satisfy \eqref{eq:i} and \eqref{eq:ij}.
We may assume that there is an index set $\emptyset\subsetneq I\subset \{1,\dots,n\}$ so that $\xi_i\neq 0$ if and only if $i\in I$.
Such a condition is consistent if and only if $p_j>0$ for any $j\notin I$ because otherwise \eqref{eq:ij} implies that $\xi_i=0$ for any $i\in I$.
To simplify the notation, we assume that $1\in I$.
Since the conditions \eqref{eq:i} and \eqref{eq:ij} are homogeneous with respect to $\xi$, we may assume $\xi_1=1$.

\noindent
(1)We consider the case $L>0$.
The condition \eqref{eq:ij} is equivalent to
\begin{equation}
    \xi_i^L=\left(\frac{x_1}{x_i}\right)^{L-1}\ \ (i\in I)
    \label{eq:xi_i^L}
\end{equation}
or equivalently
\begin{equation}
    \xi_i=\left(\frac{x_1}{x_i}\right)^{\frac{L-1}{L}}\ \ (i\in I)
    \label{eq:xi_i}
\end{equation}
for a choice of an $L$-th root.
When $I=\{1\}$, \eqref{eq:i} is reduced to $x_1=1$.
Thus, the variety \eqref{eq:singular_set1} appears.

Now, we consider the case when $|I|\geq 2$.
Assume $r'=0$ and $\sum\limits_{i\in I}x_i\xi_i=0$.
Then, it is inconsistent with \eqref{eq:i}.
Assume $r'>0$ and $\sum\limits_{i\in I}x_i\xi_i=0$.
Then, it follows from \eqref{eq:xi_i} that
\begin{equation}
    x_1+\sum_{i\in I,i\neq 1}x_i\left(\frac{x_1}{x_i}\right)^{\frac{L-1}{L}}=0
\end{equation}
or equivalently
\begin{equation}
    \sum_{i\in I}x_i^{\frac{1}{L}}=0.
\end{equation}
Thus, the variety \eqref{eq:singular_set2} appears.

Finally, we assume $\sum\limits_{i\in I}x_i\xi_i\neq 0$.
By \eqref{eq:i} and \eqref{eq:xi_i^L}, we obtain
\begin{equation}
    (x_1)^{L-1}=x_i\left(\sum_{i\in I}x_i\xi_i\right)^L\ \ (i\in I).
    \label{eq:6.13}
\end{equation}
By taking $L$-th root of \eqref{eq:6.13} and substituting \eqref{eq:xi_i}, we obtain a relation
\begin{equation}
    \sum_{i\in I}x_i^{\frac{1}{L}}=1.
    \label{eq:the_curved_divisor}
\end{equation}
Thus, the variety \eqref{eq:singular_set3} appears.

\noindent
(2)We consider the case $L=0$.
The condition \eqref{eq:ij} implies that 
\begin{equation}
    x_i=x_j\ \ \ \ (i,\,j\in I).
    \label{eq:diagonal}
\end{equation}
Assume $I=\{1\}$.
Then, the condition \eqref{eq:diagonal} is an empty condition and \eqref{eq:i} implies that $x_1=1$.
Assume $|I|= 2$ and set $I=\{ 1,j\}$.
Then, the condition \eqref{eq:diagonal} is just $x_1=x_j$.
Setting $\xi_j=-\frac{x_1}{x_j}$, we see that the condition \eqref{eq:i} is also satisfied.
Finally, assume $|I|>2$.
Then, the condition \eqref{eq:diagonal} defines a subvariety of codimension higher or equal to $2$, which does not contribute to the singular locus.

\end{proof}

\begin{rem}
    As in Theorem \ref{thm:Sing}, the varieties \eqref{eq:singular_set2} and \eqref{eq:singular_set3} admit rational parametrizations.
    Let $I\subset\{1,\dots,n\}$ be a subset such that $\prod\limits_{i\notin I}p_i>0$ and $|I|\geq 2$.
    The varieties \eqref{eq:singular_set2} and \eqref{eq:singular_set3} are regarded as subvarieties of $\C^I$.
    Let us take an element $i_0\in I$.
    When $r'>0$, a rational parametrization of \eqref{eq:singular_set2} is given by 
    \begin{equation}
        \C^{I\setminus\{ i_0\}}\ni (t_i)_{i\in I\setminus\{ i_0\}}\mapsto \Biggl( \epsilon\biggl( \sum_{i\in I,i\neq i_0}\!t_i\biggr)^{\!L},(t_i^L)_{i\in I\setminus\{ i_0\}}\Biggr)\in\C^I.
    \end{equation}
    On the other hand, a rational parametrization of \eqref{eq:singular_set3} is given by
    \begin{equation}
        \C^{I\setminus\{ i_0\}}\ni (t_i)_{i\in I\setminus\{ i_0\}}\mapsto \left( \frac{1}{\Bigl(1+\!\sum\limits_{i\in I, i\neq i_0}\!t_i\Bigr)^{\!L}},\Biggl(\biggl(\frac{t_i}{1+\!\sum\limits_{i\in I, i\neq i_0}\!t_i}\biggr)^{\!L}\Biggr)_{\! i\in I\setminus\{ i_0\}}\!\right)\in\C^I.\label{eq:parametrization}
    \end{equation}
    Note that the parametrization \eqref{eq:parametrization} takes a symmetric form 
    \begin{equation}
        [t_i:i\in I]\mapsto \Biggl(\biggl( \frac{t_i}{\sum\limits_
        {i\in I}t_i}\biggr)^L\Biggr)_{i\in I}
    \end{equation}
    where $[t_i:i\in I]$ is a homogeneous coordinate.
\end{rem}

\subsection{GKZ system}\label{sec:GKZsystem}

We first recall the definition of GKZ system (\cite{gel1989hypergeometric}).
Our description follows that of \cite{gel1992general}.
We fix a natural number $N$ and set $M_{\geq 0}:=\Z_{\geq 0}^N$ viewed as a free abelian monoid.
Let $M$ be the associated abelian group $\Z^N$.
Let $L$ be an abelian subgroup of $M$ such that the quotient $M/L$ has no torsion.
We write $M^\vee$ for the dual lattice of $M$.
If $\{e_1,\dots,e_N\}$ is a free generator of $M_{\geq 0}$, the coordinate ring of $M^\vee_{\C}:=M^\vee\otimes_{\Z}\C$ is a polynomial ring $\C[M_{\geq 0}]$ and each element $e_i$ defines a differential operator: for a pair of elements $v\in M$ and $\psi\in M$, we set $\langle\psi,v\rangle:=\psi(v)$. 
For a differentiable function $f$ on $M^\vee_{\C}$, we set
\begin{equation}
    (\partial_{e_i}f)(z):=\frac{d}{dt}f(e^{t\langle e_i,z\rangle}\cdot z)|_{t=0}.
\end{equation}
In general, an element $v=\sum_{i=1}^Nv_ie_i\in M_{\geq 0}$ defines a differential operator $\partial^v:=\prod_{i=1}^N\partial_{e_i}^{v_i}$.
In the same way, we define the Euler operator: for an element $z\in M^\vee_{\C}$, $\psi\in M^\vee$, and $t\in\C^\times$, we define $t^\psi\cdot z$ by
\begin{equation}
    \langle t^\psi\cdot z,v\rangle=t^{\langle \psi,v\rangle}\langle z,v\rangle\ \ \ \ (v\in M).
\end{equation}
Then, for a differentiable function $f$ on $M^\vee_{\C}$, we set
\begin{equation}
    (\vartheta_\psi f)(z):=\frac{d}{dt}f(t^\psi\cdot z)|_{t=1}.
\end{equation}
Let $\varphi:M_{\geq 0}\to M/L$ be the composition of the natural inclusion $M_{\geq 0}\to M$ and the quotient map $A:M\to M/L$.
The map $\varphi$ induces a ring homomorphism $\C[M_{\geq 0}]\to\C[M/L]$ which is denoted by the same symbol $\varphi$.
We set $I_A:={\rm Ker}\varphi$ and $L^\perp:=\{ \psi\in M^\vee\mid\psi\equiv 0\text{ on }L\}$.
For any complex vector $v\in M_{\C}$, we set $c:=-Av$.
The GKZ system $M_A(c)$ is defined by
\begin{equation}\label{eqn:Oshima_system}
M_A(c):
    \begin{cases}
        P\cdot u(z)=0 & (P\in I_A),\\
        \vartheta_{\psi}u(z)+\langle\psi,c\rangle=0 & (\psi\in L^\perp).
    \end{cases}
\end{equation}

We identify the projection $A:M\to M/L$ with the image of a set $\{ e_1,\dots,e_N\}$ by $A$.
We set $a_i:=Ae_i$ for each $i$.
Let ${\rm Cone}(A)$ be the cone spanned by $A$.
Given a subset $I\subset\{ 1,\dots,N\}$, the cone $C(I)$ spanned by $\{ a_i\mid i\in I\}$ defines a face of ${\rm Cone}(A)$ if there is a dual vector $\phi\in(M/L)^\vee$ such that
\begin{equation}
\begin{cases}
    \langle\phi,a_i\rangle=0& (i\in I),\\
    \langle\phi,a_i\rangle<0& (i\notin I).
\end{cases}  
\label{eq:face}
\end{equation}
Conversely, any face $F$ of ${\rm Cone}(A)$ has a form $C(I)$ for some $I\subset\{ 1,\dots,N\}$.
We say $c\in (M/L)\otimes_{\Z}\C$ is non-resonant if it does not lie on a set ${\rm span}_{\C}(F)+\Z^n$ for any face $F$ of ${\rm Cone}(A)$.
Here, ${\rm span}_{\C}$ is the spanning subspace over $\C$.
Note that $c\in (M/L)\otimes_{\Z}\C$ does not lie on a set ${\rm Cone}(I)+\Z^n$ if and only if the condition
\begin{equation}
    \langle\phi,c\rangle\notin\Z
    \label{eq:non-resonance}
\end{equation}
holds for any $\phi\in(M/L)^\vee$ with \eqref{eq:face}.

We briefly recall the definition of a regular polyhedral subdivision {(\cite[Chapter 7]{GKZbook}, \cite[Chapter 8]{SturmfelsLecture})}.
A collection {$S$} of subsets of $\{1,\dots,N\}$ is called a {polyhedral subdivision} if $\{ {\rm Cone}(I)\mid I\in S\}$ is a set of cones in a {polyhedral} fan whose support equals ${\rm Cone}(A)$.
For any choice of a vector $\omega\in\R^N$ we define a polyhedral subdivision $S(\omega)$ as follows: a subset $I\subset\{1,\dots,N\}$ belongs to $S(\omega)$ if there exists a dual vector $\phi\in (M/L)^\vee_{\R}$ such that $\langle\phi,{a}_i\rangle=\omega_i$ if $i\in I$ and $\langle\phi,{a}_i\rangle<\omega_j$ if $ j\notin I$.
A polyhedral subdivision $S$ is called a {regular polyhedral subdivision} if $S=S(\omega)$ for some $\omega.$
Given a regular polyhedral subdivision $S$, we write $C_S\subset \R^N$ for the cone consisting of vectors $\omega$ such that $S(\omega)=S$.
If any maximal (with respect to inclusion) element $I$ of a regular polyhedral subdivision $T$ is a simplex, we call $T$ a {regular triangulation}.
A collection of polyhedral cones $\{ C_S\mid S\text{ is a regular subdivision}\}$ is a complete fan in $\R^N$, which we call the secondary fan.
By identifying $\R^N$ with $M^\vee_{\R}$ via dot product with respect to $\{e_1,\dots,e_N\}$, it is readily seen that each cone $C_S$ contains a linear subspace $L^\perp_{\R}:=L^\perp\otimes_{\Z}\R$.
Therefore, we often regard secondary fan as a fan in $M^\vee_{\R}/L^\perp_{\R}\simeq L^\vee_{\R}$.

\subsection{The system $\mathcal{M}^{\boldsymbol{\alpha}_1,\dots,\boldsymbol{\alpha}_n,\boldsymbol{\gamma}}_{\boldsymbol{\alpha}'_1,\dots,\boldsymbol{\alpha}'_n,\boldsymbol{\gamma}'}$ from the view point of GKZ system} \label{sec:GKZM}

Next, we consider GKZ extension of the system $\mathcal{M}^{\boldsymbol{\alpha}_1,\dots,\boldsymbol{\alpha}_n,\boldsymbol{\gamma}}_{\boldsymbol{\alpha}'_1,\dots,\boldsymbol{\alpha}'_n,\boldsymbol{\gamma}'}$.
We set
$$M_{\geq 0}:=\Z^{p_1}_{\geq 0}\oplus\cdots\oplus\Z^{p_n}_{\geq 0}\oplus\Z^r_{\geq 0}\oplus\Z^{p'_1}_{\geq 0}\oplus\cdots\oplus\Z^{p'_n}_{\geq 0}\oplus\Z^{r'}_{\geq 0}. $$
An element of $M_{\geq 0}$ has $p_1+\cdots+p_n+r+p_1'+\cdots+p_n'+r'$ entries.
As an index set, we use a set
\begin{equation}
    \{\alpha_{ij}\}_{\substack{i=1,\dots,n\\ j=1,\dots,p_i}}\cup\{ \gamma_k\}_{k=1,\dots,r}\cup\{\alpha'_{ij}\}_{\substack{i=1,\dots,n\\ j=1,\dots,p'_i}}\cup\{ \gamma'_k\}_{k=1,\dots,r'}.
    \label{eq:indices}
\end{equation}
Namely, any element $u\in M$ is specified by its entries 
\begin{align*}
    u(\alpha_{ij})&\ \ \ \  (i=1,\dots,n,\ j=1,\dots,p_i),\\
    u(\gamma_k)&\ \ \ \ (k=1,\dots,r),\\
    u(\alpha'_{ij})&\ \ \ \  (i=1,\dots,n,\ j=1,\dots,p'_i),\\
    u(\gamma'_k)&\ \ \ \ (k=1,\dots,r').
\end{align*}
We also set
$$
u(\boldsymbol\alpha_i):=\bigl(u(\alpha_{i1}),\dots,u(\alpha_{ip_i})\bigr)
$$
and use analogous notation for $u(\boldsymbol{\gamma}),u(\boldsymbol{\alpha}_i'),u(\boldsymbol{\gamma}')$.
The index set \eqref{eq:indices} is denoted by $\{\boldsymbol\alpha,\boldsymbol\gamma,\boldsymbol\alpha',\boldsymbol\gamma'\}$.
We define a sub-lattice $L\subset M$ by
\begin{equation}
    L=\left\{ u\in M\ {\Biggm |} \ \scalebox{1.3}{$\begin{smallmatrix}
    u(\alpha_{i1})=\cdots=u(\alpha_{ip_i})=-u(\alpha'_{i1})=\cdots=-u(\alpha'_{ip'_i})\ (1\leq i\leq n),\\ u(\gamma_1)=\cdots=u(\gamma_r)=-u(\gamma'_1)=\cdots=-u(\gamma'_{r'})=\sum\limits_{i=1}^nu(\alpha_{i1})
    \end{smallmatrix}$}
    \right\}.
    \label{eq:L}
\end{equation}
Let $v$ be an element of $M_{\C}$ and write it as $v=(\boldsymbol{\alpha}_1,\dots,\boldsymbol{\alpha}_n,\boldsymbol{\gamma},\boldsymbol{\alpha}'_1,\dots,\boldsymbol{\alpha}'_n,\boldsymbol{\gamma}')$.
Then, the GKZ system with these data $M,L,c:=-Av$ corresponds to the system $\mathcal{M}^{\boldsymbol{\alpha}_1,\dots,\boldsymbol{\alpha}_n,\boldsymbol{\gamma}}_{\boldsymbol{\alpha}'_1,\dots,\boldsymbol{\alpha}'_n,\boldsymbol{\gamma}'}$ as we will see below.

Let us first find a set of generators of the toric ideal $I_A$.
For any $1\leq i\leq n$ and $1\leq j\leq p_i$, we consider a vector $u$ of which all the entries are zero but $u(\alpha_{ij})=1$.
We write $z(\alpha_{ij})$ for the corresponding element of $\C[M_{\geq 0}]$.
Similarly, we define elements $z(\gamma_k),z(\alpha'_{ij}),z(\gamma'_k)\in\C[M_{\geq 0}]$.
For $i=1,\dots,n$, we set
\begin{equation}\label{eq:g_i}
g_i(z):=\underline{\left(\prod_{j=1}^{p'_i}z(\alpha'_{ij})\right)\left(\prod_{j=1}^{r'}z(\gamma'_{j})\right)}-        \left(\prod_{j=1}^{p_i}z(\alpha_{ij})\right)\left(\prod_{j=1}^{r}z(\gamma_{j})\right).
\end{equation}
For $1\leq i<j\leq n$, we set
\begin{equation}\label{eq:g_ij}
g_{ij}(z):=\underline{\left(\prod_{k=1}^{p'_i}z(\alpha'_{ik})\right)        \left(\prod_{k=1}^{p_j}z(\alpha_{jk})\right)}-\left(\prod_{k=1}^{p'_j}z(\alpha'_{jk})\right)        \left(\prod_{k=1}^{p_i}z(\alpha_{ik})\right).
\end{equation}
Let $<$ be any term order on the polynomial ring $\C[M_{\geq 0}]$ such that the monomials with underlines in the equations \eqref{eq:g_i} and \eqref{eq:g_ij} are leading terms of $g_i$ and $g_{ij}$ with respect to $<$.
Then, we obtain the following description of the toric ideal.

\begin{prop}
The followings are true.
\begin{itemize}
    \item[\rm (1)] $\mathcal{G}=\{ g_i\}_{i=1}^n\cup\{ g_{ij}\}_{1\leq i<j\leq n}$ is a reduced $<$-Gr\"obner basis of $I_A$.
    \item[\rm (2)] Suppose $p_i=p_i'$ for any $i=1,\dots,n$ and $r=r'$. Then, $\mathcal{G}$ is a universal Gr\"obner basis of $I_A$. 
\end{itemize}    
\end{prop}

\begin{proof}
The second claim (2) follows from (1), the fact that $A$ is a Lawrence lifting by \eqref{eq:L} and \cite[Theorem 7.1]{SturmfelsLecture}.
In the following, we prove (1).
In view of \cite[Lemma 1.1]{aramova2000finite}, it is enough to prove the following claim: if $u_1,u_2\in M_{\geq 0}$ satisfy conditions $z^{u_1},z^{u_2}\notin {\rm in}_<(\mathcal{G})$ and $Au_1=Au_2$, then, $u_1=u_2$.

The condition that $u\notin {\rm in}_<(\mathcal{G})$ implies that for any indices $i<j$, either $u(\boldsymbol{\alpha}'_i)$ or $u(\boldsymbol{\alpha}_j)$ has at least one zero-entry.
Without loss of generality, we may assume that there is an index $1\leq i_0\leq n+1$ such that $u_1(\boldsymbol{\alpha}'_{1}),\dots,u_1(\boldsymbol{\alpha}'_{(i_0-1)})$ all have a zero entry and $u_1({\alpha}'_{i_0k})>0$ for $k=1,\dots,p'_{i_0}$.
It follows that $u_1(\boldsymbol{\alpha}_{(i_0+1)}),\dots,u_1(\boldsymbol{\alpha}_{n})$ all have at least one zero-entry because otherwise $z^{u_1}$ is divisible by some ${\rm in}_<(g_{i_0j})$.
Similarly, we may assume that there is an index $i_0\leq j_0\leq n+1$ such that $u_2(\boldsymbol{\alpha}'_{1}),\dots,u_2(\boldsymbol{\alpha}'_{(j_0-1)})$ all have at least one zero-entry and $u_2({\alpha}'_{j_0k})>0$ for $k=1,\dots,p'_{j_0}$.
It follows that $u_1(\boldsymbol{\alpha}_{(j_0+1)}),\dots,u_1(\boldsymbol{\alpha}_{n})$ all have at least one zero-entry.

We assume $u_1\neq u_2$ and derive a contradiction.
Let us first consider the case $i_0<j_0$.
For any $j<j_0$, there is an index $k=1,\dots,p'_j$ such that $u_2(\alpha'_{jk})=0$.
Since $Au_1=Au_2$ and $u_1\in M_{\geq 0}$, we obtain $u_1(\alpha'_{j1})-u_2(\alpha'_{j1})=u_1(\alpha'_{jk})-u_2(\alpha'_{jk})\geq 0$.
In the same way, for any $i>i_0$, there is an index $k=1,\dots,p_i$ such that $u_1(\alpha_{ik})=0$ and $u_1(\alpha'_{i1})-u_2(\alpha'_{i1})=u_2(\alpha_{ik})-u_1(\alpha_{ik})\geq 0$.
We conclude that $u_1(\gamma'_k)-u_2(\gamma'_k)=\sum_{i=1}^n(u_1(\alpha'_{i1})-u_2(\alpha'_{i1}))\geq 0$ for any $k=1,\dots,r'$.
Since $u_1\neq u_2$, there is an index $i=1,\dots,n$ such that $u_1(\alpha'_{i1})-u_2(\alpha'_{i1})>0$, which implies that $z^{u_1}$ is divisible by ${\rm in}_<(g_i)$.
This is a contradiction.

We consider the case $i_0=j_0$.
We may assume that $u_1(\alpha'_{i_01})-u_2(\alpha'_{i_01})\geq 0$.
As in the previous case, we can prove that $u_1(\alpha'_{i1})-u_2(\alpha'_{i1})\geq 0$ for any $i=1,\dots,n$, which implies the same contradiction.
\end{proof}

For any subset $I\subset\{\boldsymbol{\alpha},\boldsymbol{\gamma},\boldsymbol{\alpha}',\boldsymbol{\gamma}'\}$, we write $I^c$ for its complement.
The primary decomposition of the initial ideal ${\rm in}_<(I_A)$ determines a regular triangulation of $A$ (cf. \cite[Corollary 8.4]{SturmfelsLecture}).
For the system $\mathcal{M}^{\boldsymbol{\alpha}_1,\dots,\boldsymbol{\alpha}_n,\boldsymbol{\gamma}}_{\boldsymbol{\alpha}'_1,\dots,\boldsymbol{\alpha}'_n,\boldsymbol{\gamma}'}$, we obtain an explicit description.

\begin{cor}\label{cor:regular_triangulation}
    The term order $<$ defines a regular triangulation
    \begin{equation}
    \begin{split}
        &\bigl\{ \{\alpha'_{1i_1},\dots,\alpha'_{ni_n}\}^c\bigr\}_{\substack{i_j=1,\dots,p'_j\\ j=1,\dots,n}}\cup\\
        &\quad
        \bigcup_{\ell=1}^n
        \bigl\{ \{\alpha_{1i_1},\dots,\alpha_{(\ell-1)i_{\ell-1}},\gamma'_{k},\alpha'_{(\ell+1)i_{\ell+1}},\dots,\alpha'_{ni_n}\}^c\bigr\}_{\substack{k=1,\dots ,r'\\ i_j=1,\dots,p'_j\\ j=1,\dots,n, j\neq \ell}}
    \end{split}
    \end{equation}
\end{cor}

\begin{proof}
    The corollary follows from the primary decomposition
    \begin{align}
        {\rm in}_<(\mathcal{G})&=
        \bigcap_{\substack{i_j=1,\dots,p'_j\\ j=1,\dots,n}}\!\bigl\langle z(\alpha'_{1i_1}),\dots,z(\alpha'_{ni_n})\bigr\rangle\cap
        \bigcap_{\substack{i_1=1,\dots ,r'\\ i_j=1,\dots,p'_j\\ j=2,\dots,n}}\!\bigl\langle z(\gamma'_{i_1}),z(\alpha'_{2i_2}),\dots,z(\alpha'_{ni_n})\bigr\rangle\nonumber\\
        &\quad\cap\dots\cap\bigcap_{\substack{i_j=1,\dots,p_j\\ j=1,\dots,n-1\\ i_n=1,\dots,r'}} \bigl\langle z(\alpha_{1i_1}),\dots,z(\alpha_{n-1i_{n-1}}),z(\gamma'_{i_n})\bigr\rangle.
    \end{align}
\end{proof}

The second equation of the GKZ system \eqref{eqn:Oshima_system} can be used to reduce the number of variables.
Namely, it is equivalent to an identity
\begin{equation}\label{eq:reduction}
    u(z)=t^{-\langle\psi,c\rangle}u(t^\psi\cdot z)\quad (t\in\C^\times,\ \psi\in L^\perp).
\end{equation}
The identity \eqref{eq:reduction} shows that the value of a solution $u$ to \eqref{eqn:Oshima_system} on $T:={\rm Spec}\ \C[M]\subset M^\vee_{\C}$ is determined by that on a suitable subspace of $T$.
Let us assume $\alpha'_{i1}=0$ for any $i=1,\dots,n$.
For a function $u(z)$ on $T$, we define a function $u_{\rm red}$ as a function depending only on $n$-complex variables $z(\alpha'_{11}),\dots,z(\alpha'_{n1})$ obtained from $u$ by its restriction to
\begin{equation}
\begin{cases}
        z(\alpha'_{ij})=1&(i=1,\dots,n,\ j=2,\dots,p_i'),\\
        z(\gamma'_k)=1&(k=1,\dots,r'),\\
        z(\alpha_{ij})=1&(i=1,\dots,n,\ j=1,\dots,p_i),\\
        z(\gamma_k)=1&(k=1,\dots,r).
\end{cases}
\label{eq:restriction}
\end{equation}
Then, \eqref{eq:reduction} is equivalent to
\begin{equation}
u(z)=z(\boldsymbol{\alpha})^{\boldsymbol{\alpha}}z(\boldsymbol{\gamma}')^{\boldsymbol{\gamma}}z(\boldsymbol{\alpha}')^{\boldsymbol{\alpha}'}z(\boldsymbol{\gamma}')^{\boldsymbol{\gamma}'}
    u_{\rm red}\left( \frac{z(\boldsymbol{\alpha}'_1)z(\boldsymbol{\gamma}')}{z(\boldsymbol{\alpha}_1)z(\boldsymbol{\gamma})},\dots,\frac{z(\boldsymbol{\alpha}'_n)z(\boldsymbol{\gamma}')}{z(\boldsymbol{\alpha}_n)z(\boldsymbol{\gamma})}\right),
    \label{eq:u_red}
\end{equation}
where we use notation
\begin{equation}
z(\boldsymbol{\alpha})^{\boldsymbol{\alpha}}:=\prod_{i=1}^n\prod_{j=1}^{p_i}z({\alpha}_{ij})^{\alpha_{ij}},\quad\quad z(\boldsymbol{\alpha}'_i):=\prod_{j=1}^{p_i}z({\alpha}_{ij})^{\alpha_{ij}}
\end{equation}
and so on.
Using the identity \eqref{eq:u_red}, it can be shown that a holomorphic function $u(z)$ on $T$ is annihilated by $g_i(\partial),g_{ij}(\partial)$ if and only if its restriction to \eqref{eq:restriction} and 
\begin{equation}
            z(\alpha'_{i1})=(-1)^{p_i+r}x_i \quad i=1,\dots,n
\end{equation}
is a solution to the system $\mathcal{M}^{\boldsymbol{\alpha}_1,\dots,\boldsymbol{\alpha}_n,\boldsymbol{\gamma}}_{\boldsymbol{\alpha}'_1,\dots,\boldsymbol{\alpha}'_n,\boldsymbol{\gamma}'}$.
Thus, we obtain the following proposition
\begin{prop}\label{prop:solution_correspondence}
Let $u(x_1,\dots,x_n)$ be a holomorphic function defined on a domain in $\C^n$.
Then, $u$ is a solution to the system $\mathcal{M}^{\boldsymbol{\alpha}_1,\dots,\boldsymbol{\alpha}_n,\boldsymbol{\gamma}}_{\boldsymbol{\alpha}'_1,\dots,\boldsymbol{\alpha}'_n,\boldsymbol{\gamma}'}$ if and only if 
$$
z(\boldsymbol{\alpha})^{\boldsymbol{\alpha}}z(\boldsymbol{\gamma}')^{\boldsymbol{\gamma}}z(\boldsymbol{\alpha}')^{\boldsymbol{\alpha}'}z(\boldsymbol{\gamma}')^{\boldsymbol{\gamma}'}
    u\left( (-1)^{p_1+r}\frac{z(\boldsymbol{\alpha}'_1)z(\boldsymbol{\gamma}')}{z(\boldsymbol{\alpha}_1)z(\boldsymbol{\gamma})},\dots,(-1)^{p_n+r}\frac{z(\boldsymbol{\alpha}'_n)z(\boldsymbol{\gamma}')}{z(\boldsymbol{\alpha}_n)z(\boldsymbol{\gamma})}\right)
$$
is a solution to the GKZ system $M_A(c)$.
\end{prop}

\begin{prop}
    The rank of the system $\mathcal{M}^{\boldsymbol{\alpha}_1,\dots,\boldsymbol{\alpha}_n,\boldsymbol{\gamma}}_{\boldsymbol{\alpha}'_1,\dots,\boldsymbol{\alpha}'_n,\boldsymbol{\gamma}'}$ is 
\begin{align}
   & p'_1\cdots p'_n+r'\sum_{k=1}^n
p_1\cdots p_{k-1}p'_{k+1}\cdots p'_n\label{eq:rankGKZ}\\
   &\quad
    =\begin{cases}
 p_1\cdots p_n+r\displaystyle\sum_{k=1}^np_1\cdots p_{k-1}p_{k+1}\cdots p_n&(L=0),\\
 \displaystyle\frac1{L}(p'_1\cdots p'_n r-p_1\cdots p_nr')&(L\ne 0).
 \end{cases}
\notag
\end{align}
\end{prop}

\begin{proof}
    By \eqref{eq:L} and \cite[Theorem 8.1.5]{de2010triangulations}, $A$ is totally unimodular, i.e., any simplex has a normalized volume one.
    In view of \cite[Proposition 13.5]{SturmfelsLecture}, the toric ring $\C[M_{\geq 0}]/I_A$ is normal and hence Cohen-Macaulay by \cite[Theorem 1]{hochster1972rings}.
    Thus, \cite[Corollary 5.21]{adolphson1994hypergeometric} and Corollary \ref{cor:regular_triangulation} of this paper proves that the holonomic rank of the GKZ system $M_A(c)$ is given by \eqref{eq:rankGKZ}.
    The Proposition follows from Proposition \ref{prop:solution_correspondence}.
\end{proof}

\subsection{The structure of the secondary fan and irreducibility}\label{sec:irreducible}

In this section, we provide a combinatorial description of the secondary fan of the GKZ system $M_A(c)$.
We define the following elements of $\Z^n$:
\begin{equation}
\begin{cases}
    \boldsymbol{b}(\alpha'_{ij})=(0,\dots,0,\overset{i}{\breve{1}},0,\dots,0) &(j=1,\dots,p'_i),\\
    \boldsymbol{b}(\gamma'_{j})\ =(1,\dots,1) &(j=1,\dots,r' ),\\
    \boldsymbol{b}(\alpha_{ij})=(0,\dots,0,\overset{i}{\breve{-1}},0,\dots,0)& (j=1,\dots,p_i),\\
    \boldsymbol{b}(\gamma_{j})\ =(-1,\dots,-1) &(j=1,\dots,r).
\end{cases}
\label{eq:b-vectors}
\end{equation}
The vectors \eqref{eq:b-vectors} define a Gale dual of the lattice configuration $A$.
It determines the secondary fan.
Below, we only consider the case $p_ip'_i\neq 0$ for any $i=1,\dots,n$.

\medskip

\noindent
\underline{Case 1: $rr'\neq 0$}

The secondary fan consists of $(n+1)!$ simplicial cones.
We set $y_0:=0$.
Then, for a permutation $\s$ of $\{ 0,\dots,n\}$, we set
\begin{equation}
    C_{\s}:=\{ y=(y_1,\dots,y_n)\in\R^n\mid y_{\s(0)}\leq y_{\s(1)}\leq \cdots \leq y_{\s(n)}\}.
\end{equation}
Each $C_{\s}$ defines an affine space $U_{\s}$ near a torus fixed point.
We set $x_0:=1$.
Then, the local coordinate ring of $U_{\s}$ specified by $\s$ is given by 
\begin{equation}
    \C\left[\frac{x_{\s(1)}}{x_{\s(0)}},\frac{x_{\s(2)}}{x_{\s(1)}},\dots,\frac{x_{\s(n)}}{x_{\s(n-1)}}\right].
\end{equation}
The secondary fan is given by $\{ C_{\s}\}_{\s\in S_{n+1}}$.
We introduce a convention $[0^0:0^1:\dots:0^n]$ as the unique torus fixed point contained in $U_{\rm id}$.
We allow to scale it by $0$ as $[0^0:0^1:\dots:0^n]=[0^{-1}:0^0:\dots:0^{n-1}]=[\infty:0^0:\dots,0^{n-1}]$.
Similarly, we denote by $p_{\s}:=[0^{\s^{-1}(0)}:\cdots:0^{\s^{-1}(n)}]$ the unique torus fixed point in $U_{\s}$.
Since $\{ C_{\s}\}_{\s}$ is a polyhedral fan, it defines a toric variety $X$.
It turns out that $X$ is isomorphic to an iterated blowing-up of a product of projective lines: for a subset $I\subset\{1,\dots,n\}$, we set $D_I:=\left(\cap_{i\in I}\{x_i=0\}\right)\cup\left(\cap_{i\in I}\{x_i=\infty\}\right)$.
We write ${\rm Bl}_{D_I}$ for the blowing-up along $D_I$.
$$X=\prod_{\substack{I\subset\{1,\dots,n\}\\ |I|=2}}{\rm Bl}_{D_I}\cdots\prod_{\substack{I\subset\{1,\dots,n\}\\ |I|=n}}{\rm Bl}_{D_I}\left( \mathbb{P}^1\times \cdots\times\mathbb{P}^1\right).$$
The proof of the above identity is a direct comparison of local coordinates.

Given a pair of torus fixed points $p_1,p_2$ on $X$, we say $p_1$ is adjacent to $p_2$ if there is a one-dimensional torus orbit on $X$.
In our set-up, $p_{\s_1}$ is adjacent to $p_{\s_2}$ if $\s_1^{-1}\s_2$ is a consecutive transposition.
Then, for a pair of cones $C_{\s_1}$ and $C_{\s_2}$, the intersection $C_{\s_1}\cap C_{\s_2}$ is a common facet if and only if $\s_1$ is adjacent to $\s_2$.

One can also recognize a polytope embedded in $X$ which is a generalization of the hexagon appeared in {\bf Fig 1} in \S\ref{sec:coord trans}.
Indeed, by a well-known theorem on toric moment map \cite[\S4.2]{fulton1993introduction}, the Euclidean closure $X_{\geq 0}$ of the first quadrant $(0,\infty)^n$ in $X$ is homeomorphic to a polytope which is dual to the secondary fan $\{ C_{\s}\}_{\s\in S_{n+1}}$.
This polytope is called a permutohedron (\cite[\S7.3.C]{GKZbook}), which is a polytope whose edge graph is isomorphic to the Cayley graph of the permutation group $S_{n+1}$ with respect to the set of generators given by consecutive transpositions.

\medskip

\noindent
\underline{Case 2: $rr'= 0$}

Since the argument is symmetric, we may assume $r=0$.
Let $I\subset\{ 1,\dots,n\}$ be a subset.
We write $S_I$ for the set of bijections from $I$ to itself.
$S_I$ is a group whose product is given by the composition of maps.
For a given subset $I=\{ i_1,\dots,i_a\}\subset\{ 1,\dots,n\}$ and $\s\in S_I$, we set
\begin{equation}
    C_{I,\s}:=\{ y=(y_1,\dots,y_n)\in\R^n\mid y_i\leq 0\ (i\notin I),\ y_{\s(i_1)}\leq  \cdots \leq y_{\s(i_a)}\}.
\end{equation}
Each $C_{I,\s}$ defines an affine space $U_{I,\s}$ near a torus fixed point.
Then, the local coordinate ring of $U_{I,\s}$ specified by $(I,\s)$ is given by 
\begin{equation}
    \C\left[\frac{1}{x_i}\ (i\notin I),x_{\s(i_1)},\frac{x_{\s(i_2)}}{x_{\s(i_1)}},\dots,\frac{x_{\s(i_a)}}{x_{\s(i_{a-1})}}\right].
\end{equation}
The secondary fan is given by $\{ C_{I,\s}\}_{I\subset\{ 1,\dots,n\},\s\in S_{I}}$.
The number of maximal cones is $\sum_{i=0}^n\frac{n!}{i!}$.
The first few terms are $2,5,16,65,326,...$.
Since $\{ C_{I,\s}\}_{I,\s}$ is a polyhedral fan, it defines a toric variety $X$.
It turns out that $X$ is isomorphic to an iterated blow-up of a product of projective lines: for a subset $I\subset\{1,\dots,n\}$, we set $D_I:=\left(\cap_{i\in I}\{x_i=0\}\right)$.
$$X=\prod_{\substack{I\subset\{1,\dots,n\}\\ |I|=2}}{\rm Bl}_{D_I}\cdots\prod_{\substack{I\subset\{1,\dots,n\}\\ |I|=n}}{\rm Bl}_{D_I}\left( \mathbb{P}^1\times \cdots\times\mathbb{P}^1\right).$$

Now, let us determine the irreducibility.

\begin{thm}\label{thm:irreducibility}
    $M_A(c)$ is irreducible if and only if none of $\alpha_{1i_1}+\cdots+\alpha_{ni_n}+\gamma'_{k},\,\alpha'_{1i_1}+\cdots+\alpha'_{ni_n}+\gamma_{k},\alpha_{ij_1}+\alpha'_{ij_2},\,\gamma_{k_1}+\gamma'_{k_2}$ is an integer.    
\end{thm}

\begin{rem}
    Some of these conditions may be empty. For example, if $r'=0$, the condition is that none of $\alpha'_{1i_1}+\cdots+\alpha'_{ni_n}+\gamma_k,\,\alpha_{ij_1}+\alpha'_{ij_2}$ is an integer.
\end{rem}

\begin{proof}
In view of \cite{schulze2012resonance}, we need to write down the condition \eqref{eq:non-resonance}.
To do this, we note that \eqref{eq:b-vectors} is a dual configuration of $A$ (\cite[\S6.4]{ziegler2012lectures}).
By \cite[p88, item 1]{grunbaum2003convex}, for a subset $I\subset\{ \boldsymbol{\alpha},\boldsymbol{\gamma},\boldsymbol{\alpha}',\boldsymbol{\gamma}'\}$, $C(I)$ defines a face of ${\rm Cone}(A)$ if and only if $I=\{ \boldsymbol{\alpha},\boldsymbol{\gamma},\boldsymbol{\alpha}',\boldsymbol{\gamma}'\}$ or $0\in {\rm rel.int.}({\bf b}(i)\mid i\notin I)$ where ${\rm rel.int.}$ is the relative interior.
Therefore, an index set $I\subset\{ \boldsymbol{\alpha},\boldsymbol{\gamma},\boldsymbol{\alpha}',\boldsymbol{\gamma}'\}$ defines a facet of ${\rm Cone}(A)$ if and only if it is minimal with respect to inclusion.
The complements of such $I$ are listed as follows:
\begin{equation}
    \{\alpha_{1i_1},\dots,\alpha_{ni_n},\gamma'_{k}\},\ \{\alpha'_{1i_1},\dots,\alpha'_{ni_n},\gamma_{k}\},\ \{\alpha_{ij_1},\alpha'_{ij_2}\},\ \{\gamma_{k_1},\gamma'_{k_2}\}.
    \label{eq:cofacets}
\end{equation}
Let us identify the dual lattice $M^\vee$ as a set of integral vectors 
$$
\Z^{p_1}\oplus\cdots\oplus\Z^{p_n}\oplus\Z^r\oplus\Z^{p'_1}\oplus\cdots\oplus\Z^{p'_n}\oplus\Z^{r'}
$$
so that the duality pairing $\langle\psi,v\rangle$ for $\psi\in M^\vee$ and $v\in M$ is given by the dot product of integral vectors.
In this sense, we write $\psi(\delta)\in \Z$ for the $\delta$-th entry of $\psi$ for any $\delta\in \{\boldsymbol{\alpha},\boldsymbol{\gamma},\boldsymbol{\alpha}',\boldsymbol{\gamma}'\}$.
In view of the identification $(M/L)^\vee=L^\perp$, a dual vector $\phi\in (M/L)^\vee$ defining a face $I$ is represented by a vector $\psi\in M^\vee$ such that
\begin{equation}
\psi|_L\equiv 0\ \ \ \text{and}\ \ \ 
    \begin{cases}
        \psi(\delta)=0&(\delta\notin I),\\
        \psi(\delta)<0&(\delta\in I).
    \end{cases}
    \label{eq:equation_B}
\end{equation}
To the complement of each of \eqref{eq:cofacets}, we solve \eqref{eq:equation_B} to obtain the condition of Theorem \ref{thm:irreducibility} from \eqref{eq:non-resonance}.
For example, let us derive the condition $\alpha_{1i_1}+\cdots+\alpha_{ni_n}+\gamma'_k\notin \Z$.
We take $I=\{ \alpha_{1i_1},\dots,\alpha_{ni_n},\gamma'_k\}$.
The condition \eqref{eq:equation_B} reads 
\[
\psi(\gamma_k')=\psi(\alpha_{ip_i})\ (i=1,\dots,n)\ \ \text{and}\ \ \psi(\delta)=0\ (\delta\notin I).
\]
It follows that the primitive vector $\psi$ defining the facet $\{ \boldsymbol{\alpha},\boldsymbol{\gamma},\boldsymbol{\alpha}',\boldsymbol{\gamma}'\}\setminus I$ is given by 
\[
\psi(\delta)=
\begin{cases}
    -1&(\delta=\alpha_{1i_1},\dots,\alpha_{ni_n},\gamma_k'),\\
    0&(\text{otherwise}).
\end{cases}
\]
The condition \eqref{eq:non-resonance} for the facet $\{ \boldsymbol{\alpha},\boldsymbol{\gamma},\boldsymbol{\alpha}',\boldsymbol{\gamma}'\}\setminus I$ reads $\alpha_{1i_1}+\cdots+\alpha_{ni_n}+\gamma'_k\notin \Z$.

\end{proof}

\subsection{Local solutions and the connection problem}\label{subsec:connection}
In this section, we provide a general formula for the analytic continuation of the system $\mathcal M
 ^{\boldsymbol \alpha_1,\dots, \boldsymbol \alpha_n, \boldsymbol \gamma}
 _{\boldsymbol \alpha'_1,\dots,\boldsymbol \alpha'_n,\boldsymbol \gamma'}.
$
Let $(ij)$ denotes the transposition of $i$ and $j$.
We take the base point $b$ from the region $\{ |x_n|\ll|x_{n-1}|\ll\cdots\ll|x_1|\ll1\}$.
We first consider a path 
\begin{equation}
C_0:(0,\infty)\ni t\mapsto (-t,c_2(t),\dots,c_n(t))\in(-\infty,0)\times (0,\infty)^{n-1}
\label{eq:pathn}
\end{equation}
with small $c_i(t)$ so that $c_n(t)\ll c_{n-1}(t)\ll\cdots\ll c_2(t)$ and the system $\mathcal M
 ^{\boldsymbol \alpha_1,\dots, \boldsymbol \alpha_n, \boldsymbol \gamma}
 _{\boldsymbol \alpha'_1,\dots,\boldsymbol \alpha'_n,\boldsymbol \gamma'} 
$
has no singularities along the path.
Let $p_{\s_1}$ be adjacent to $p_{\s_2}$.
The coordinate transformations 
$$(x_1,\dots,x_n)\mapsto (x_{\s(1)},\dots,x_{\s(n)}),\left(\frac{1}{x_1},\dots,\frac{1}{x_n}\right),\left(\frac{x_1}{x_n},\dots,\frac{x_{n-1}}{x_n},\frac{1}{x_n}\right)$$
all induce automorphisms on $X$.
These automorphisms induce an action of $S_{n+1}\times S_2$ onto $X$.
For any element $g\in S_{n+1}\times S_2$, we write 
\begin{equation}
\varphi_g:X\to X    
\end{equation}
for the corresponding automorphism.
Then, there exists an element $g\in S_{n+1}\times S_2$ so that   $\varphi_g(p_{\s_1})=p_{\rm id}$ and $\varphi_g(p_{\s_2})=p_{(01)}$.
Then, we write $C(\s_2\s_1)$ for the path $\varphi_g\circ C_0$.
In view of Proposition \ref{prop:transformations2}, the connection problem between $p_{\s_1}$ and $p_{\s_2}$ is reduced to that between $p_{\rm id}$ and $p_{(01)}$.

Therefore, we discuss the connection problem between $p_{\rm id}$ and $p_{(01)}$ below.
For any $\varepsilon=(\varepsilon_1,\dots,\varepsilon_n)\in\{-1,+1\}^n$ such that $p_j+\varepsilon_jr=p'_j+\varepsilon_jr'$, we write $G^{p_1,\dots,p_n,r}_{p'_1,\dots,p_n',r',\varepsilon}\left(\substack{{\boldsymbol \alpha}_1,\dots,{\boldsymbol \alpha}_n,{\boldsymbol \gamma}\\ {\boldsymbol \alpha}'_1,\dots,{\boldsymbol \alpha}'_n,{\boldsymbol \gamma}'};x\right)$ for a series
\begin{equation}\label{eq:multivariate_G}
\sum_{m=(m_1,\dots,m_n)\in\Z^n_{\geq 0}}\frac{({\boldsymbol \alpha}_1)_{m_1}\cdots ({\boldsymbol \alpha}_n)_{m_n}({\boldsymbol \gamma})_{\varepsilon\cdot m}}{({\boldsymbol 1-\alpha}'_1)_{m_1}\cdots ({\boldsymbol 1-\alpha}'_n)_{m_n}({\boldsymbol 1-\gamma}')_{\varepsilon\cdot m}}x^m.
\end{equation}
Note that \eqref{eq:multivariate_F} is a special case of \eqref{eq:multivariate_G} with $\varepsilon=(+1,\dots,+1)$.
For simplicity, we write $G_\varepsilon\left(\substack{{\boldsymbol \alpha}_1,\dots,{\boldsymbol \alpha}_n,{\boldsymbol \gamma}\\ {\boldsymbol \alpha}'_1,\dots,{\boldsymbol \alpha}'_n,{\boldsymbol \gamma}'};x\right)$ for $G^{p_1,\dots,p_n,r}_{p'_1,\dots,p_n',r',\varepsilon}\left(\substack{{\boldsymbol \alpha}_1,\dots,{\boldsymbol \alpha}_n,{\boldsymbol \gamma}\\ {\boldsymbol \alpha}'_1,\dots,{\boldsymbol \alpha}'_n,{\boldsymbol \gamma}'};x\right)$.

\medskip

\noindent
\underline{A basis of solutions at $p_{\rm id}$:} it consists of the following two kinds of series.
The first $p_1'\cdots p_n'$ solutions to the system are
\begin{equation}
    x^{\boldsymbol{\alpha'_I}} F^{p_1,\dots,p_n,r}_{p'_1,\dots,p'_n,r'}\left(\begin{smallmatrix}
\boldsymbol{\alpha_1}+\alpha_{1i_1}'&\dots&\boldsymbol{\alpha_n}+\alpha_{ni_n}'&\boldsymbol{\gamma+|\boldsymbol{\alpha'_I}|}\\
\boldsymbol{\alpha'_1}-\alpha_{1i_1}'&\dots&\boldsymbol{\alpha'_n}-\alpha_{ni_n}'&\boldsymbol{\gamma'-|\alpha'_I|}\\
\end{smallmatrix};
x_1,\dots,x_n\right)
\end{equation}
with $\boldsymbol{\alpha'_I}:=(\alpha'_{1i_1},\dots,\alpha'_{ni_n})$ and $|\boldsymbol{\alpha'_I}|:=\alpha'_{1i_1}+\cdots+\alpha'_{ni_n}$.
To simplify the notation below, for any $\beta\in\C^n$, we set
\begin{equation}
    F\left(\begin{smallmatrix}
\boldsymbol{\alpha_1}&\dots&\boldsymbol{\alpha_n}&\boldsymbol{\gamma}\\
\boldsymbol{\alpha'_1}&\dots&\boldsymbol{\alpha'_n}&\boldsymbol{\gamma}'
\end{smallmatrix};
\beta;
x\right)
:=
x^{\beta} F\left(\begin{smallmatrix}
\boldsymbol{\alpha_1}+\beta_{1}&\dots&\boldsymbol{\alpha_n}+\beta_{n}&\boldsymbol{\gamma}+|\beta|\\
\boldsymbol{\alpha'_1}-\beta_{1}&\dots&\boldsymbol{\alpha'_n}-\beta_{n}&\boldsymbol{\gamma}'-|\beta|
\end{smallmatrix};
x\right)
\end{equation}
and
\begin{equation}
G_{\varepsilon}\left(
\begin{smallmatrix}
{\boldsymbol \alpha}_1&\dots&{\boldsymbol \alpha}_n&{\boldsymbol \gamma}\\ {\boldsymbol \alpha}'_1&\dots&{\boldsymbol \alpha}'_n&{\boldsymbol \gamma}'    
\end{smallmatrix}
;\beta; x\right)
:=
x^\beta G_{\varepsilon}\left(
\begin{smallmatrix}
{\boldsymbol \alpha}_1+\beta_1&\dots&{\boldsymbol \alpha}_n+\beta_n&{\boldsymbol \gamma}+\varepsilon\cdot \beta\\ 
{\boldsymbol \alpha}'_1-\beta_1&\dots&{\boldsymbol \alpha}'_n-\beta_n&{\boldsymbol \gamma}'-\varepsilon\cdot\beta
\end{smallmatrix}
; x\right).
\end{equation}

\noindent

For any $k=1,\dots,n$, the $r'p_1\cdots p_{k-1}p'_{k+1}\cdots p'_n$ solutions are

\begin{align}
    G_{\varepsilon(k)}
    \left(
    \begin{smallmatrix}
    \boldsymbol{\alpha'_1}&\dots&\boldsymbol{\alpha'_{k-1}}&\boldsymbol{\gamma}&\boldsymbol{\alpha_{k+1}}&\dots&\boldsymbol{\alpha_{n}}&\boldsymbol{\alpha}_k\\
    \boldsymbol{\alpha_1}&\dots&\boldsymbol{\alpha_{k-1}}&\boldsymbol{\gamma'}&\boldsymbol{\alpha'_{k+1}}&\dots&\boldsymbol{\alpha'_{n}}&\boldsymbol{\alpha}'_k
    \end{smallmatrix};
    \beta(k)
    ;\xi(k)
    \right)
    ,\label{eq:G solutions2}
\end{align}

\noindent
where we set 
\begin{align}
\varepsilon(k)&:=(\overset{k}{\overbrace{+1,\dots,+1}},\overset{n-k}{\overbrace{-1,\dots,-1}}),\\    
\beta(k)&:=({\alpha_{1i_1}},\dots,{\alpha_{(k-1)i_{k-1}}},{\gamma'_{i_k}},{\alpha'_{(k+1)i_{k+1}}},\dots,{\alpha'_{ni_{n}}}),\\
\xi(k)&:=(\epsilon x_1^{-1}x_k,\dots,\epsilon x_{k-1}^{-1}x_k,x_k,x_k^{-1}x_{k+1},\dots,x_k^{-1}x_{n}).
\end{align}

\medskip

\noindent
\underline{A basis of solutions at $p_{(01)}$:}

In view of Proposition \ref{prop:transformations2}, we obtain a relation
$$
T_{(x_1,\dots,x_n)\to \left(\frac{1}{x_1},\epsilon\frac{x_2}{x_1},\dots,\epsilon\frac{x_n}{x_1}\right)}\mathcal M
 ^{\boldsymbol \alpha_1,\dots, \boldsymbol \alpha_n, \boldsymbol \gamma}
 _{\boldsymbol \alpha'_1,\dots,\boldsymbol \alpha'_n,\boldsymbol \gamma'}=\mathcal M
 ^{\boldsymbol{\gamma}',\boldsymbol \alpha_2,\dots, \boldsymbol \alpha_n, \boldsymbol \alpha'_1}
 _{\boldsymbol{\gamma},\boldsymbol \alpha'_2,\dots,\boldsymbol \alpha'_n,\boldsymbol \alpha_1}.
$$
Thus, a basis of local solutions at $p_{(01)}$ is a transformations of that at $p_{\rm id}$ via the transformation $(x_1,\dots,x_n)\mapsto \left(\frac{1}{x_1},\epsilon\frac{x_2}{x_1},\dots,\epsilon\frac{x_n}{x_1}\right)$.
The first $rp'_2\cdots p'_n$ solutions are

\begin{equation}
F\left(\begin{smallmatrix}
\boldsymbol{\gamma}',\boldsymbol{\alpha_2}&\dots&\boldsymbol{\alpha_n}&\boldsymbol{\alpha}'_1\\
\boldsymbol{\gamma},\boldsymbol{\alpha'_2}&\dots&\boldsymbol{\alpha'_n}&\boldsymbol{\alpha}_1
\end{smallmatrix};
(\gamma_k,\alpha'_{2i_2},\dots,\alpha'_{ni_n});
\frac{1}{x_1},\epsilon\frac{x_2}{x_1},\dots,\epsilon\frac{x_n}{x_1}\right).
\end{equation}

In the same manner, we obtain $p_1p'_2\cdots p'_n$ solutions 
\begin{align}
    G_{\varepsilon(1)}
    \left(
    \begin{smallmatrix}
\boldsymbol{\alpha'_1}&\boldsymbol{\alpha}_2&\dots&\boldsymbol{\alpha}_{n}&\boldsymbol{\gamma}'\\
\boldsymbol{\alpha_1}&\boldsymbol{\alpha}'_2&\dots&\boldsymbol{\alpha}'_{n}&\boldsymbol{\gamma}    
\end{smallmatrix};
(\alpha_{1i_1},\alpha'_{2i_2},\dots,\alpha'_{ni_n})
    ;\eta(1)
    \right)
    \label{eq:G_solution_at_p}
\end{align}
with 
$$\eta(1)=\left(\frac{1}{x_1},\epsilon x_2,\dots,\epsilon x_n\right).$$
The other $r'\sum\limits_{k=2}^np_1\dots p_{k-1}p'_{k+1}\dots p'_n$ solutions are given by \eqref{eq:G solutions2} for $k=2,\dots,n$.
These solutions are holomorphic along the path $C_0$ of analytic continuation.

To state the formula, for any $\varepsilon,\tilde{\varepsilon}=(\tilde{\varepsilon}_1,\dots,\tilde{\varepsilon}_n)\in\{-1,+1\}^n$ and $\beta\in\C^n$, we put
\begin{equation}
    F\left(\begin{smallmatrix}
\boldsymbol{\alpha_1}&\dots&\boldsymbol{\alpha_n}&\boldsymbol{\gamma}\\
\boldsymbol{\alpha'_1}&\dots&\boldsymbol{\alpha'_n}&\boldsymbol{\gamma}'
\end{smallmatrix};
\beta;
\tilde{\varepsilon};
x\right)
:=
(\tilde{\varepsilon}_1x_1)^{\beta_1}\cdots (\tilde{\varepsilon}_nx_n)^{\beta_n} F\left(\begin{smallmatrix}
\boldsymbol{\alpha_1}+\beta_{1}&\dots&\boldsymbol{\alpha_n}+\beta_{n}&\boldsymbol{\gamma}+|\beta|\\
\boldsymbol{\alpha'_1}-\beta_{1}&\dots&\boldsymbol{\alpha'_n}-\beta_{n}&\boldsymbol{\gamma}'-|\beta|
\end{smallmatrix};
x\right),
\end{equation}

\begin{equation}
G_{\varepsilon}\left(
\begin{smallmatrix}
{\boldsymbol \alpha}_1&\dots&{\boldsymbol \alpha}_n&{\boldsymbol \gamma}\\ {\boldsymbol \alpha}'_1&\dots&{\boldsymbol \alpha}'_n&{\boldsymbol \gamma}'    
\end{smallmatrix}
;\beta
;\tilde{\varepsilon}
;x\right)
:=
(\tilde{\varepsilon}_1x_1)^{\beta_1}\cdots (\tilde{\varepsilon}_nx_n)^{\beta_n} G_{\varepsilon}\left(
\begin{smallmatrix}
{\boldsymbol \alpha}_1+\beta_1&\dots&{\boldsymbol \alpha}_n+\beta_n&{\boldsymbol \gamma}+\varepsilon\cdot \beta\\ 
{\boldsymbol \alpha}'_1-\beta_1&\dots&{\boldsymbol \alpha}'_n-\beta_n&{\boldsymbol \gamma}'-\varepsilon\cdot\beta
\end{smallmatrix}
; x\right).
\end{equation}
and
$$
\delta_{i_2\dots i_n}:=\sum_{j=2}^n\alpha'_{ji_j}.
$$

\begin{thm}\label{thm:connection}
One has the following connection formula along the path $C_0$.
\begin{align*}
&F\left(\begin{smallmatrix}
\boldsymbol{\alpha_1}&\dots&\boldsymbol{\alpha_n}&\boldsymbol{\gamma}\\
\boldsymbol{\alpha'_1}&\dots&\boldsymbol{\alpha'_n}&\boldsymbol{\gamma}'
\end{smallmatrix};
(\alpha'_{1i_1},\dots,\alpha'_{ni_n});
-\varepsilon(1);
x\right)\\
=&
\sum_{k=1}^r a^{i_2\dots i_n}_{i_1k}F\left(\begin{smallmatrix}
\boldsymbol{\gamma}',\boldsymbol{\alpha_2}&\dots&\boldsymbol{\alpha_n}&\boldsymbol{\alpha}'_1\\
\boldsymbol{\gamma},\boldsymbol{\alpha'_2}&\dots&\boldsymbol{\alpha'_n}&\boldsymbol{\alpha}_1
\end{smallmatrix};
(\gamma_k,\alpha'_{2i_2},\dots,\alpha'_{ni_n});
(-1,\epsilon,\dots,\epsilon);
\frac{1}{x_1},\epsilon\frac{x_2}{x_1},\dots,\right.\\
&\ \left.\epsilon\frac{x_n}{x_1}\right)
+\sum_{\ell=1}^{p_1}
b^{i_2\dots i_n}_{i_1\ell}
G_{\varepsilon(1)}
    \left(
    \begin{smallmatrix}
\boldsymbol{\alpha'_1}&\boldsymbol{\alpha}_2&\dots&\boldsymbol{\alpha}_{n}&\boldsymbol{\gamma}'\\
\boldsymbol{\alpha_1}&\boldsymbol{\alpha}'_2&\dots&\boldsymbol{\alpha}'_{n}&\boldsymbol{\gamma}    
\end{smallmatrix};
(\alpha_{1\ell},\alpha'_{2i_2},\dots,\alpha'_{ni_n})
;-\varepsilon(1)
    ;\eta(1)
    \right),
\allowdisplaybreaks\\
&a^{i_2\dots i_n}_{i_1k}=\frac
{ \prod\limits_{\nu\neq i_1} \!\Gamma(1+\alpha'_{1i_1}-\alpha'_{1\nu})
  \prod\limits_{\nu=1}^{r'}  \!\Gamma(1+\alpha'_{1i_1}+\gamma'_\nu+\delta_{i_2\dots i_n})
  }
  {
  \prod\limits_{\nu\ne i_1}   \!\Gamma(1-\gamma_k-\delta_{i_2\dots i_n}-\alpha'_{1\nu})
  \prod\limits_{\nu=1}^{r'}  \!\Gamma(1-\gamma_k-\gamma'_\nu)
}\\
&\qquad\qquad{}\times\frac
{
  \prod\limits_{\nu\ne k}    \!\Gamma(\gamma_\nu-\gamma_k)
  \prod\limits_{\nu=1}^{p_1}  \!\Gamma(\alpha_{1\nu}-\delta_{i_2\dots i_n}-\gamma_k)
  }
{
  \prod\limits_{\nu\ne k}   \!\Gamma(\alpha'_{1,i_1}+\gamma_\nu
  +\delta_{i_2\dots i_n}) 
  \prod\limits_{\nu=1}^{p_1} \!\Gamma(\alpha'_{1i_1}+\alpha_{1\nu})
},
\allowdisplaybreaks\\
&b^{i_2\dots i_n}_{i_1\ell}=
\frac{\prod\limits_{\nu\ne i_1}  \!\Gamma(1+\alpha'_{1i_1}-\alpha'_{1\nu})
  \prod\limits_{\nu=1}^{r'}     \!\Gamma(1+\alpha'_{1i_1}-\gamma'_\nu+\delta_{i_2\dots i_n})}
{\prod\limits_{\nu\ne i_1}    \!\Gamma(1-\alpha'_{1\nu}-\alpha_{1\ell})
  \prod\limits_{\nu=1}^{r'}     \!\Gamma(1-\gamma'_\nu+\delta_{i_2\dots i_n}-\alpha_{1\ell})}
\\
&\qquad\qquad{}\times 
 \frac{\prod\limits_{\nu\ne\ell}  \!\Gamma(\alpha_{1\nu}-\alpha_{1\ell})
  \prod\limits_{\nu=1}^r     \!\Gamma(\gamma_\nu+\delta_{i_2\dots i_n}-\alpha_{1\ell})}
  {\prod\limits_{\nu\ne\ell}  \!\Gamma(\alpha'_{1i_1}+\alpha_{1\nu})
  \prod\limits_{\nu=1}^r     \!\Gamma(\alpha'_{1i_1}+\gamma_\nu+\delta_{i_2\dots i_n})}
\intertext{and}
&G_{\varepsilon(1)}
    \left(
    \begin{smallmatrix}
    \boldsymbol{\gamma}&\boldsymbol{\alpha_{2}}&\dots&\boldsymbol{\alpha_{n}}&\boldsymbol{\alpha}_1\\
    \boldsymbol{\gamma'}&\boldsymbol{\alpha'_{2}}&\dots&\boldsymbol{\alpha'_{n}}&\boldsymbol{\alpha}'_1
    \end{smallmatrix};
    (\gamma'_{i_1},\alpha'_{2i_2},\dots,\alpha'_{ni_n});
    -\varepsilon(n);
    \xi(1)
    \right)\\
=&
\sum_{\ell=1}^r c^{i_2\dots i_n}_{i_1\ell}
F\left(\begin{smallmatrix}
\boldsymbol{\gamma}',\boldsymbol{\alpha_2}&\dots&\boldsymbol{\alpha_n}&\boldsymbol{\alpha}'_1\\
\boldsymbol{\gamma},\boldsymbol{\alpha'_2}&\dots&\boldsymbol{\alpha'_n}&\boldsymbol{\alpha}_1
\end{smallmatrix};
(\gamma_\ell,\alpha'_{2i_2},\dots,\alpha'_{ni_n});
(-1,\epsilon,\dots,\epsilon);
\frac{1}{x_1},\epsilon\frac{x_2}{x_1},\dots,
\right.\\
&
\left.\epsilon\frac{x_n}{x_1}\right)
+\sum_{k=1}^{p_1} 
d^{i_2\dots i_n}_{i_1k}
G_{\varepsilon(1)}
    \left(
    \begin{smallmatrix}
\boldsymbol{\alpha'_1}&\boldsymbol{\alpha}_2&\dots&\boldsymbol{\alpha}_{n}&\boldsymbol{\gamma}'\\
\boldsymbol{\alpha_1}&\boldsymbol{\alpha}'_2&\dots&\boldsymbol{\alpha}'_{n}&\boldsymbol{\gamma}    
\end{smallmatrix};
(\alpha_{1k},\alpha'_{2i_2},\dots,\alpha'_{ni_n})
;-\varepsilon(1)
    ;\eta(1)
    \right),
\allowdisplaybreaks\\
&c^{i_2\dots i_n}_{i_1\ell}=
\frac
{     \prod\limits_{\nu\ne i_1}   \!\Gamma(1+\gamma'_{i_1}-\gamma'_\nu)
     \prod\limits_{\nu=1}^{p_1'} \!\Gamma(1+\gamma'_{i_1}-\delta_{i_2\dots i_n}-\alpha'_{1\nu})
}
{  \prod\limits_{\nu\ne i_1} \!\Gamma(1-\gamma_\ell-\gamma'_\nu)
   \prod\limits_{\nu=1}^{p'_1} \!\Gamma(1-\gamma_\ell-\delta_{i_2\dots i_n}-\alpha'_{1\nu})
 }\\
&\qquad\qquad\times
\frac
{  
   \prod\limits_{\nu\ne\ell} \!\Gamma(\gamma_\nu-\gamma_\ell) 
   \prod\limits_{\nu=1}^{p_1}    \!\Gamma(\alpha_{1\nu}-\gamma_\ell-\delta_{i_2\dots i_n})
}
{
  \prod\limits_{\nu\neq \ell} \!\Gamma(\gamma'_{i_1}+\gamma_\nu) 
  \prod\limits_{\nu=1}^{p_1} \!\Gamma(\gamma'_{i_1}-\delta_{i_2\dots i_n}+\alpha_{1\nu})
 }
,
\allowdisplaybreaks\\
&d^{i_2\dots i_n}_{i_1k}=
\frac
{   \prod\limits_{\nu\ne i_1}   \!\Gamma(1+\gamma'_{i_1}-\gamma'_\nu)
    \prod\limits_{\nu=1}^{p'_1} \!\Gamma(1+\gamma'_{i_1}-\delta_{i_2\dots i_n}-\alpha'_{1\nu})
}{
  \prod\limits_{\nu\ne i_1}  \!\Gamma(1-\alpha_{1k}- \gamma'_\nu+\delta_{i_2\dots i_n})
  \prod\limits_{\nu=1}^{p'_1} \!\Gamma(1-\alpha_{1k}-\alpha'_{1\nu})
}
\\
&\qquad\qquad\times
\frac
{
  \prod\limits_{\nu\ne k}   \!\Gamma(\alpha_{1\nu}-\alpha_{1k})
  \prod\limits_{\nu=1}^r    \!\Gamma(\gamma_\nu+\delta_{i_2\dots i_n}-\alpha_{1k})
}
{
 \prod\limits_{\nu\ne k}   \!\Gamma(\gamma'_{i_1}-\delta_{i_2\dots i_n}+\alpha_{1\nu})
  \prod\limits_{\nu=1}^r  \!\Gamma(\gamma'_{i_1}+\gamma_\nu)
}.
\end{align*}
\end{thm}

The proof of Theorem \ref{thm:connection} is same as that of Theorem \ref{thm:conenct2}.
Indeed, local behavior of solutions are given as follows. 
\begin{align*}
&F\left(\begin{smallmatrix}
\boldsymbol{\alpha_1}&\dots&\boldsymbol{\alpha_n}&\boldsymbol{\gamma}\\
\boldsymbol{\alpha'_1}&\dots&\boldsymbol{\alpha'_n}&\boldsymbol{\gamma}'
\end{smallmatrix};
(\alpha'_{1i_1},\dots,\alpha'_{ni_n});
-\varepsilon(1);
x\right)\\
&\quad\overset{x_2,\dots,x_n\to0}\sim x_2^{\alpha_{2i_2}'}\cdots x_n^{\alpha_{ni_n}'}F_{p_1+r}\left(\begin{smallmatrix}
    \boldsymbol{\alpha}_1,\boldsymbol{\gamma}+\delta_{i_2\dots i_n}\\
    \boldsymbol{\alpha}'_1,\boldsymbol{\gamma}'-\delta_{i_2\dots i_n}
\end{smallmatrix}
;\alpha_{1i_1}';-;x_1\right),
\allowdisplaybreaks\\
&G_{\varepsilon(1)}
    \left(
    \begin{smallmatrix}
    \boldsymbol{\gamma}&\boldsymbol{\alpha_{2}}&\dots&\boldsymbol{\alpha_{n}}&\boldsymbol{\alpha}_1\\
    \boldsymbol{\gamma'}&\boldsymbol{\alpha'_{2}}&\dots&\boldsymbol{\alpha'_{n}}&\boldsymbol{\alpha}'_1
    \end{smallmatrix};
    (\gamma'_{i_1},\alpha'_{2i_2},\dots,\alpha'_{ni_n});
    -\varepsilon(n);
    \xi(1)
    \right)\\
&\quad
\overset{x_2,\dots,x_n\to0}\sim
 x_2^{\alpha_{2i_2}'}\cdots x_n^{\alpha_{ni_n}'}F_{p_1+r}\left(\begin{smallmatrix}
    \boldsymbol{\gamma}+\delta_{i_2\dots i_n},\boldsymbol{\alpha}_1\\
    \boldsymbol{\gamma}'-\delta_{i_2\dots i_n},\boldsymbol{\alpha}'_1
\end{smallmatrix}
;\gamma'_{i_1};-;x_1\right),
\allowdisplaybreaks\\
&F\left(\begin{smallmatrix}
\boldsymbol{\gamma}',\boldsymbol{\alpha_2}&\dots&\boldsymbol{\alpha_n}&\boldsymbol{\alpha}'_1\\
\boldsymbol{\gamma},\boldsymbol{\alpha'_2}&\dots&\boldsymbol{\alpha'_n}&\boldsymbol{\alpha}_1
\end{smallmatrix};
(\gamma_k,\alpha'_{2i_2},\dots,\alpha'_{ni_n});
(-1,\epsilon,\dots,\epsilon);
\frac{1}{x_1},\epsilon\frac{x_2}{x_1},\dots,\epsilon\frac{x_n}{x_1}\right)\\
&\quad\overset{x_2,\dots,x_n\to0}\sim x_2^{\alpha_{2i_2}'}\cdots x_n^{\alpha_{ni_n}'}F_{p_1+r}\left(\begin{smallmatrix}
    \boldsymbol{\gamma}'-\delta_{i_2\dots i_n},\boldsymbol{\alpha}'_1\\
    \boldsymbol{\gamma}+\delta_{i_2\dots i_n},\boldsymbol{\alpha}_1
\end{smallmatrix}
;\gamma_{k}+\delta_{i_2\dots i_n};-;\frac{1}{x_1}\right),
\allowdisplaybreaks\\
&G_{\varepsilon(1)}
    \left(
    \begin{smallmatrix}
\boldsymbol{\alpha'_1}&\boldsymbol{\alpha}_2&\dots&\boldsymbol{\alpha}_{n}&\boldsymbol{\gamma}'\\
\boldsymbol{\alpha_1}&\boldsymbol{\alpha}'_2&\dots&\boldsymbol{\alpha}'_{n}&\boldsymbol{\gamma}    
\end{smallmatrix};
(\alpha_{1k},\alpha'_{2i_2},\dots,\alpha'_{ni_n})
;-\varepsilon(1)
    ;\eta(1)
    \right)\\
&\quad
\overset{x_2,\dots,x_n\to0}\sim
x_2^{\alpha_{2i_2}'}\cdots x_n^{\alpha_{ni_n}'}F_{p_1+r}\left(\begin{smallmatrix}
    \boldsymbol{\gamma}'-\delta_{i_2\dots i_n},\boldsymbol{\alpha}'_1\\
    \boldsymbol{\gamma}+\delta_{i_2\dots i_n},\boldsymbol{\alpha}_1
\end{smallmatrix}
;\alpha_{1k};-;\frac{1}{x_1}\right).
\end{align*}

\subsection{Braid monodromy}\label{sec:braid}
In this section, we focus on the case $p_i=p'_i$ for any $i=1,\dots,n$ and $r=r'$.
By Theorem \ref{thm:Singn}, the singular locus of the system $\mathcal{M}^{\boldsymbol{\alpha}_1,\dots,\boldsymbol{\alpha}_n,\boldsymbol{\gamma}}_{\boldsymbol{\alpha}'_1,\dots,\boldsymbol{\alpha}'_n,\boldsymbol{\gamma}'}$ is the braid arrangement 
$$\mathcal{A}:=\biggl\{\prod\limits_{i=1}^nx_i(x_i-1)\prod\limits_{1\leq i<j\leq n}\!(x_i-x_j)\biggr\}\subset \C^n.$$
We show that the connection problem among adjacent points generate the global monodromy.
We identify $\C^n\setminus\mathcal{A}$ with the configuration space $\mathcal{M}_{0,n+3}$ of $n+3$ points on $\mathbb{P}^1$.
Any point $(x_1,\dots,x_n)\in \C^n\setminus\mathcal{A}$ corresponds to a configuration $[1,x_1,\dots,x_n,0,\infty]\in\mathcal{M}_{0,n+3}$.
We cite the following well-known description of its fundamental group.
Let $b=[1,x_1,\dots,x_{n},0,\infty]\in\mathcal{M}_{0,n+3}$ be a base point.
For any $i,j=0,\dots,n+3$ with $i<j$ and $(i,j)\neq (0,n+1),(0,n+2),(n+1,n+2)$, we define a loop $\gamma_{ij}$ on $\mathcal{M}_{0,n+2}$ as a move which starts from $b$, goes only around $x_{i}=x_j$ once and comes back to $b$. 
Here, we used a convention $x_0=1,x_{n+1}=1,x_{n+2}=\infty$.
The fundamental group $P_{n+3}:=\pi_1(b,\mathcal{M}_{0,n+3})$ is known as the pure braid group.
The following lemma can be found in \cite[Lemma 1.8.2]{birman1974braids}.
\begin{lem}
    $\gamma_{ij}$ generates $P_{n+3}$.
\end{lem}

Let $\s_1,\s_2\in S_{n+1}$ be given and assume that $\s_1$ is adjacent to $\s_2$.
Given a pair of indices $0\leq i<j\leq n$ we take a composition of paths of the form $C(\s_2\s_1)$ from $p_{\rm id}$ to some $p_\s$ with $\s(0)=j$ and $\s(1)=i$.
The composed path is denoted by $C$.
Then, the path $C^{-1}\circ \varphi_{\s}(\gamma_{(01)})\circ C$ is homotopic to $\gamma_{ij}$.
In the same way, one can construct $\gamma_{ij}$ for any $i,j=0,\dots,n+2$ with $i<j$ and $(i,j)\neq (0,n+1),(0,n+2),(n+1,n+2)$ using paths of the form $C(\s_2\s_1)$ and $\varphi_g(\gamma_{01})$ for some $g\in S_{n+1}\times S_2$.
By the discussion of \S\ref{subsec:connection} and transformations in Proposition \ref{prop:transformations2}, the analytic continuations along these paths are explicitly computable.
We summarize the discussion of this section as a theorem.

\begin{thm}
Suppose $p_i=p'_i$ for any $i=1,\dots,n$ and $r=r'$.
Then, the monodromy representation of the system $\mathcal{M}^{\boldsymbol{\alpha}_1,\dots,\boldsymbol{\alpha}_n,\boldsymbol{\gamma}}_{\boldsymbol{\alpha}'_1,\dots,\boldsymbol{\alpha}'_n,\boldsymbol{\gamma}'}$ is a representation of the pure braid group $P_{n+3}$.
By taking a basis of local solutions as in \S\ref{subsec:connection}, one can compute monodromy matrices in terms of Gamma functions and trigonometric functions.
\end{thm}

\section*{Acknowledgements}
The first author was supported by JSPS KAKENHI Grant Number 22K13930, and partially supported by JST CREST Grant Number JP19209317.
The second author was supported by JSPS KAKENHI Grant Number 18K03341.
The second author expresses his sincere gratitude to Jiro Sekiguchi who explains his study on Appell's hypergeometric functions to the author, which helps the author to write this paper.

\bibliographystyle{abbrv}
\bibliography{references}

\end{document}